\documentclass{article}

\usepackage{blindtext}
\usepackage{subfiles} 

\usepackage{commonpackage}
\usepackage{package_for_list_of_content}
\usepackage{package_to_delete}
\usepackage{macro}

\def\grad{\text{grad}}

\numberwithin{equation}{section}
\theoremstyle{plain}
\newtheorem{theorem}{Theorem}[section]

\newtheorem{proposition}[theorem]{Proposition}
\newtheorem{corollary}[theorem]{Corollary}
\newtheorem{lemma}[theorem]{Lemma}
\newtheorem{definition}[theorem]{Definition}
\newtheorem{remark}[theorem]{Remark}

\newtheorem{claim}[theorem]{Claim}

\title{Oscillatory integral operators on manifolds and related Kakeya and Nikodym problems}
\author{Song Dai, Liuwei Gong, Shaoming Guo, Ruixiang Zhang}
\date{}

\begin{document}

\maketitle

\begin{abstract}

We consider  Carleson-Sj\"{o}lin operators on Riemannian manifolds that arise naturally from the study of Bochner-Riesz problems on manifolds. They are special cases of H\"{o}rmander-type  oscillatory integral operators. We obtain improved $L^p$ bounds of Carleson-Sj\"{o}lin operators in two cases: The case where the underlying manifold has constant sectional curvature and the case where the manifold satisfies Sogge's chaotic curvature condition \cite{Sog99}. 

The two results rely on very different methods: To prove the former result, we show that on a Riemannian manifold, the distance function satisfies Bourgain's condition in \cite{GWZ22} if and only if the manifold has constant sectional curvature. To obtain the second result, we introduce the notion of  ``contact orders'' to H\"{o}rmander-type oscillatory integral operators, prove that if a H\"{o}rmander-type oscillatory integral operator is of a finite contact order, then it always has better $L^p$ bounds than ``worst cases'' (in spirit of Bourgain and Guth \cite{BG11} and Guth, Hickman and Iliopoulou \cite{GHI19}), and eventually verify that for Riemannian manifolds that satisfy Sogge's chaotic curvature condition, their distance functions alway have finite contact orders. 

As byproducts, we obtain new bounds for Nikodym maximal functions on manifolds of constant sectional curvatures. 
%
%
%
%
%
%

\end{abstract}

%
%
%

\tableofcontents

\section{Introduction}

We study several problems in harmonic analysis and their connections, including H\"ormander-type oscillatory integrals, Carleson-Sj\"olin operators on manifolds, curved Kakeya problems and Nikodym problems on manifolds. In the introduction, we will introduce these four problems, and give a brief review of known results. Experts can skip the slightly long introduction and go to Section \ref{231024section2} directly, where new results are stated. Notations are listed at the end of the introduction.

\subsection{H\"ormander-type oscillatory integrals}

Consider H\"ormander-type oscillatory integral operators
\begin{equation}\label{230717e1_1}
T^{(\phi)}_N f(x, t):=\int_{\mathbb{R}^{n-1}} e^{i N \phi(x, t; y)} a(x, t; y) f(y) \mathrm{d} y.
\end{equation}
Here $x\in \R^{n-1}, t\in \R, y\in \R^{n-1}$ and $N\in \R$ is a large real number. To simplify notation, we often write $\bfx=(x, t)$. Moreover,  $a(x, t; y)$ is a smooth function supported in a bounded open neighborhood of the origin. On the support of $a(x, t; y)$, let us assume that $\phi$ is a smooth function and that 
\begin{enumerate}
\item[(H1)] $\rank \nabla_{x}\nabla_y \phi(x, t; y)=n-1;$ 
\item[(H2)] if we define 
\begin{equation}
G_0(\bfx; y):=\partial_{y_1} \nabla_{\bfx} \phi(\bfx; y)\wedge\dots\wedge\partial_{y_{n-1}} \nabla_{\bfx} \phi(\bfx; y),
\end{equation}
then 
\begin{equation}
\left.\operatorname{det} \nabla_{y}^2\left\langle\nabla_{\mathbf{x}} \phi(\mathbf{x} ; y), G_0\left(\mathbf{x} ; y_0\right)\right\rangle\right|_{y=y_0} \neq 0.
\end{equation}
\end{enumerate}
The function $\phi(x, t; y)$ will be refereed to as the \underline{phase function} of the operator $T_N^{(\phi)}$, and $a(x, t; y)$ will be refereed to as its amplitude function. By saying that the phase function $\phi(x, t; y)$ satisfies \underline{H\"ormander's non-degeneracy condition}, we mean that it satisfies (H1) and (H2) above.  \\

We are interested in proving estimates of the form 
\begin{equation}\label{230324e1_5}
\norm{T^{(\phi)}_N f}_{L^p(\R^n)} \lesim_{\phi, a, p, \epsilon} N^{-\frac{n}{p}+\epsilon} \norm{f}_{L^p(\R^{n-1})},
\end{equation}
for every $\epsilon>0$ and every $N\ge 1$, and for a range of exponents $p$ that is as large as possible. \\

For estimates of the form \eqref{230324e1_5}, the simplest and perhaps the most interesting phase function  is 
\begin{equation}\label{230812e1_7}
\phi(x, t; y)= x\cdot y+ t|y|^2.
\end{equation}
{\bf Fourier restriction conjecture.} The estimate \eqref{230324e1_5} holds for all 
\begin{equation}\label{230813e1_8}
p\ge \frac{2n}{n-1},
\end{equation}
with the phase function $\phi(x, t; y)$ given by \eqref{230812e1_7}.\\

H\"ormander \cite{Hor73} asked whether for phase functions $\phi$ satisfying (H1) and (H2), the estimate \eqref{230324e1_5} could still hold for the same range of $p$ as in \eqref{230813e1_8}. Let us briefly review known results on \eqref{230324e1_5}.

To simplify our discussion, we will always work in a small neighborhood of the origin, and therefore we will pick a sufficiently small $\epsilon_{\phi}>0$ depending on $\phi$, and assume that $a(x, t; y)$ is supported in $\B^{n-1}_{\epsilon_{\phi}}\times \B^1_{\epsilon_{\phi}}\times \B^{n-1}_{\epsilon_{\phi}}$.   We often without loss of generality assume that $\phi(\bfx; y)$ is in its normal form at the origin, that is, 
\begin{equation}\label{230324e1_4}
\phi(\bfx; y)=x\cdot y+ t\inn{y}{Ay}+ O(|t| |y|^3+ |\bfx|^2 |y|^2),
\end{equation}
where $A$ is an $(n-1)\times (n-1)$ non-degenerate matrix. Normal forms \eqref{230324e1_4} were introduced by H\"ormander \cite{Hor73} and Bourgain \cite{Bou91} to simplify calculations. It is elementary to see (\cite[page 323]{Bou91}) that, after simple transformations, all phase functions $\phi(x, t; y)$ satisfying H\"ormander's non-degeneracy condition can be written in normal forms. \\

\begin{theorem}[Hickman and Iliopoulou, \cite{HI22}]\label{230409theorem1_1}
Let $\phi$ be a phase function of the form \eqref{230324e1_4}. Let $s_0$ be the signature of the matrix $A$.  Then \eqref{230324e1_5} holds for all 
\begin{equation}\label{230409e1_6}
p \geqslant 
\begin{cases}
2 \cdot \frac{s_0+2(n+1)}{s_0+2(n-1)} & \text { if } n \text { is odd } \\ 
2 \cdot \frac{s_0+2 n+3}{s_0+2 n-1} & \text { if } n \text { is even }
\end{cases}
\end{equation}
\end{theorem}
Hickman and Iliopoulou \cite{HI22}, by generalizing the examples constructed earlier by Bourgain \cite{Bou91}, Wisewell \cite{Wis05}, Minicozzi and Sogge \cite{MS97} and Bourgain and Guth \cite{BG11}, also showed that the range of $p$ given by \eqref{230409e1_6} is sharp. More precisely, there exists  a phase function of the form \eqref{230324e1_4} with $s_0$ being the signature of $A$, and the estimate \eqref{230324e1_5} fails for $p$ outside the ranges given in \eqref{230409e1_6}.\\

Several special cases of Theorem \ref{230409theorem1_1} are particularly interesting and were proven earlier.

\begin{theorem}[Stein \cite{Ste84}, Bourgain and Guth \cite{BG11}]\label{230816theorem1_2}
Let $\phi$ be a phase function of the form \eqref{230324e1_4}. Assume that $A$ is of smallest possible signature, that is, $\sgn(A)=1$ when $n$ is even and $\sgn(A)=0$ when $n$ is odd. Then 
\eqref{230324e1_5} holds for all 
\begin{equation}\label{230409e1_7a}
p \geqslant 
\begin{cases}
2 \cdot \frac{n+1}{n-1} & \text { if } n \text { is odd } \\ 
2 \cdot \frac{n+2}{n} & \text { if } n \text { is even }
\end{cases}
\end{equation}
\end{theorem}

\begin{theorem}[Lee \cite{Lee06}; Guth, Hickman, Iliopoulou, \cite{GHI19}]\label{230814theorem1_3}
Let $\phi$ be a phase function of the form \eqref{230324e1_4}. Assume that $A$ is positive definite. Then \eqref{230324e1_5} holds for all 
\begin{equation}\label{230409e1_7}
p \geqslant 
\begin{cases}
2 \cdot \frac{3n+1}{3n-3} & \text { if } n \text { is odd } \\ 
2 \cdot \frac{3n+2}{3n-2} & \text { if } n \text { is even }
\end{cases}
\end{equation}
\end{theorem}

Recently, Guo, Wang and Zhang \cite{GWZ22} imposed extra assumptions on the phase function $\phi$ and proved \eqref{230324e1_5} for some $p$ that goes beyond the sharp range given by \eqref{230409e1_7}.

\begin{definition}[Bourgain's condition, \cite{Bou91}, \cite{GWZ22}]\label{230903definition1_4}
 Let $\phi$ be a phase function satisfying H\"ormander's non-degeneracy conditions. We say that it satisfies Bourgain's condition at $(\bfx_0; y_0)$ if 
\begin{equation}\label{230719e1_9}
\left(\left(G_0 \cdot \nabla_{\mathbf{x}}\right)^2 \nabla_{y}^2 \phi\right)\left(\mathbf{x}_0 ; y_0\right) \text { is a multiple of }\left(\left(G_0 \cdot \nabla_{\mathbf{x}}\right) \nabla_{y}^2 \phi\right)\left(\mathbf{x}_0 ; y_0\right) \text {. }
\end{equation}
The constant here is allowed to depend on $\bfx_0$ and $y_0$. 
\end{definition}

\begin{theorem}[Guo, Wang and Zhang \cite{GWZ22}]\label{231013theorem1_5}
Let $\phi$ be a phase function of the form \eqref{230324e1_4} with $A$ positive definite. Moreover, assume that $\phi$ satisfies Bourgain's condition for every $(\bfx_0; y_0)$. Then  \eqref{230324e1_5} holds for all 
\begin{equation}\label{230409e1_10}
p>p_{\mathrm{GWZ}}(n):= 2+\frac{2.5921}{n}+O\left(n^{-2}\right).
\end{equation}
\end{theorem}

Bourgain's condition is very natural when studying H\"ormander-type oscillatory integrals. On the one hand, Bourgain \cite{Bou91} proved that if the phase function $\phi$ fails Bourgain's condition at some $(\bfx_0; y_0)$, then \eqref{230324e1_5} can not hold for all 
\begin{equation}
p\ge \frac{2n}{n-1},
\end{equation}
the range of the Fourier restriction conjecture (see \eqref{230813e1_8}). More precisely, there exists $p> 2n/(n-1)$ depending only on the dimension $n$ such that \eqref{230324e1_5} fails at this $p$. 

On the other hand, it is conjectured in \cite{GWZ22} that if $\phi$ satisfies Bourgain's condition at every $(\bfx_0; y_0)$, then \eqref{230324e1_5} holds for all $p\ge 2n/(n-1)$.

\subsection{Curved Kakeya problem}\label{230903subsection1_2}

Associated to H\"ormander-type oscillatory integrals, one can define Kakeya sets. Given a phase function $\phi(\bfx; y)$ satisfying (H1) and (H2), we pick $\epsilon_{\phi}>0$ to be a sufficiently small constant depending on $\phi$.

\begin{definition}[Curved tubes]\label{230617defi1_6}
For $y\in \B^{n-1}_{\epsilon_{\phi}}$, $\bfx\in \B^{n}_{\epsilon_{\phi}}$ and $0<\delta< \epsilon_{\phi}$, define
\begin{equation}
\begin{aligned}
& \Gamma^{(\phi)}_y(\bfx):=\left\{\bfx' \in \mathbb{R}^n\cap \B^n_{2\epsilon_{\phi}}: \nabla_y \phi(\bfx'; y)=\nabla_y \phi(\bfx; y)\right\} \\
& T_y^{\delta, (\phi)}(\bfx):=\left\{\bfx' \in \mathbb{R}^n\cap \B^n_{2\epsilon_{\phi}}:\left|\nabla_y \phi(\bfx'; y)-\nabla_y \phi(\bfx; y)\right|<\delta\right\}
\end{aligned}
\end{equation}
If $\bfx$ is of the form $(\omega, 0)$, that is, the last coordinate is $0$, then we often abbreviate $\Gamma^{(\phi)}_y(\bfx)$ and $T_y^{\delta, (\phi)}(\bfx)$ to $\Gamma^{(\phi)}_y(\omega)$ and $T_y^{\delta, (\phi)}(\omega)$. If it is clear from the context which $\phi$ is involved, we will abbreviate $\Gamma^{(\phi)}_y(\bfx), T_y^{\delta, (\phi)}(\bfx)$ to $\Gamma_y(\bfx), T_y^{\delta}(\bfx)$, respectively.  We will call $T^{\delta}_y(\omega)$ the $\delta$-tube associated to the phase function $\phi(\bfx; y)$ with frequency $y$ and initial location $\omega$; $\Gamma_y(\omega)$ will be called the central curve of $T^{\delta}_y(\omega)$. \\

\end{definition}

\begin{definition}[Curved Kakeya  sets] A set $E \subset \mathbb{R}^n$ with $\mc{L}^n(E)=0$ is a curved Kakeya set (associated to $\phi$ ) if for all $y \in \mathbb{B}^{n-1}_{\epsilon_{\phi}}$ there exists an $\omega \in \mathbb{B}^{n-1}_{\epsilon_{\phi}}$ such that $\Gamma_y(\omega) \subset E$. 
\end{definition}

\begin{definition}[Curved Kakeya maximal function, \cite{Bou91}]
Given a phase function $\phi(\bfx; y)$ satisfying (H1) and (H2).  For $y\in \B^{n-1}_{\epsilon_{\phi}}$ and $0<\delta< \epsilon_{\phi}$, we define 
\begin{equation}
\mc{K}^{(\phi)}_{\delta} f(y):=
\sup_{
\omega \in \mathbb{B}^{n-1}_{\epsilon_{\phi}}
}
\frac{1}{
\mc{L}^n(T_y^{\delta, (\phi)}(\omega))
}
\int_{T_y^{\delta, (\phi)}(\omega)}
|f|
\end{equation}
If it is clear from the context which $\phi$ is involved, then we often abbreviate $\mc{K}^{(\phi)}_{\delta} f$ to $\mc{K}_{\delta} f$.
\end{definition}

The problem of studying the Hausdorff dimensions of curved Kakeya sets will be referred to as \emph{curved Kakeya problems}. Among all the phase functions $\phi$, the one that is perhaps the most interesting is given by \eqref{230812e1_7}, that is, 
\begin{equation}\label{230812e1_7hhh}
\phi(x, t; y)= x\cdot y+ t|y|^2.
\end{equation}
 In this case, the central curve $\Gamma_y(\omega)$ becomes 
\begin{equation}
\{(x, t): x+ 2t y=\omega\}
\end{equation}
which is a straight line. Kakeya sets associated to \eqref{230812e1_7hhh} will be referred to as the traditional Kakeya sets, or the traditional straight line Kakeya sets.  Moreover, for this special phase function, we have \\

\noindent {\bf Kakeya conjecture.} Let $\phi$ be given by \eqref{230812e1_7hhh}. Then every Kakeya set associated to $\phi$ must have full Hausdorff dimension, that is, it must have Hausdorff dimension $n$. \\

\noindent {\bf Maximal Kakeya conjecture.} Let $\phi$ be given by \eqref{230812e1_7hhh}. Take $\epsilon_{\phi}=1$. Then the associated Kakeya maximal operator $\mc{K}_{\delta} $ satisfies 
\begin{equation}
\norm{\mc{K}_{\delta} f}_{L^n(\mathbb{B}^{n-1})} \lesim_{n, \epsilon} \delta^{-\epsilon} \norm{f}_{L^n(\R^n)},
\end{equation}
for every $\epsilon>0$ and every $\delta\in (0, 1)$. \\

It is well-known that 
\begin{equation}
\begin{split}
& \text{Fourier restriction conjecture}\\
  \implies & \text{Maximal Kakeya conjecture} \\
 \implies & \text{Kakeya conjecture}.
\end{split}
\end{equation}
Similar to the above implications, for H\"ormander-type oscillatory integrals and curved Kakeya problems, we also have

\begin{theorem}[Wisewell \cite{Wis05}]\label{230706thm1_8}
Given a phase function $\phi(\bfx; y)$ satisfying (H1) and (H2). Let $\epsilon_{\phi}>0$ be a sufficiently small constant depending on $\phi$. 
\begin{enumerate}
\item[(1)] Suppose that 
\begin{equation}\label{230324e1_5zzz}
\norm{T^{(\phi)}_N f}_{L^p(
\B^{n}_{\epsilon_{\phi}}
)} \lesim_{\phi, p, n, a} N^{-\frac{n}{p}} \norm{f}_{L^p(\R^{n-1})},
\end{equation}
for some $p>1$, every $N\ge 1$ and every $a$ supported in $\B^{n-1}_{\epsilon_{\phi}}\times \B^1_{\epsilon_{\phi}}\times \B^{n-1}_{\epsilon_{\phi}}$, then the curved Kakeya maximal function is of restricted weak type $(q, q)$ with norm at most $\delta^{
-2(
\frac{n}{q}-1
)
}$ where $q:=(p/2)'$. In particular, as $p\to \frac{2n}{n-1}$, we see that $q\to n$. 
\item[(2)] If \eqref{230324e1_5zzz} holds for all $p>\frac{2n}{n-1}$, the largest possible range, then every curved Kakeya set associated to $\phi$ must have  Hausdorff dimension $n$. 
\end{enumerate}
\end{theorem}

Curved Kakeya problems are studied intensively in Wisewell's thesis \cite{Wis03} and her paper \cite{Wis05}. We refer interested readers to these two works for more results she obtained.

\subsection{Carleson-Sj\"olin operators on manifolds}\label{230814sub1_3}

Let us work with a smooth Riemannian metric $\{g_{ij}(\bfx)\}_{1\le i, j\le n}$ defined on a small open neighborhood of $0\in \R^n$. The Riemannian manifold is denoted by $\mc{M}$. We will only study curvature properties of $\mc{M}$ near the origin. Let $\epsilonm>0$ be a small constant depending on $\mc{M}$.  Let $a(\bfx; \bfy): \mc{M}\times \mc{M}\to \R$ be a compactly supported smooth function supported on $\mathbb{B}^n_{\epsilonm}\times \mathbb{B}^n_{\epsilonm}$, and supported away from the diagonal. Define 
\begin{equation}\label{230509e1_1}
T_N^{(\mc{M})} f(\bfx):=\int_{\mc{M}} 
e^{i N\dist(\bfx, \bfy)} a(\bfx; \bfy) f(\bfy) d\bfy,
\end{equation}
where $\dist$ refers to the distance function on $\mc{M}$, and call it a \underline{Carleson-Sj\"olin operator} on the manifold $\mc{M}$. Our goal is to prove \begin{equation}\label{230407e1_9}
\norm{T^{(\mc{M})}_N f}_{L^p(\mc{M})} \lesim_{\mc{M}, p, a, \epsilon} N^{-\frac{n}{p}+\epsilon} \norm{f}_{L^p(\mc{M})},
\end{equation}
for every $\epsilon>0, N\ge 1$, and for a range of $p$ that is as large as possible.\\

If one takes the metric $g$ to the identity matrix at every point, then we have \\

\noindent {\bf Bochner-Riesz conjecture.} The estimate \eqref{230407e1_9} holds for all 
\begin{equation}
p\ge \frac{2n}{n-1},
\end{equation}
if $\mc{M}$ is taken to be the Euclidean space. \\

Moreover, Tao \cite{Tao99} proved that 
\begin{equation}
\text{Bochner-Riesz conjecture}\implies \text{Fourier restriction conjecture.}
\end{equation}
The study of Carleson-Sj\"olin operators on general Riemannian manifolds also has a long history, and the operator \eqref{230509e1_1} already appeared in Minicozzi and Sogge \cite{MS97} (see also Sogge \cite[page 290]{Sog17}). 
 To prove bounds of the form \eqref{230407e1_9}, we will follow the Carleson-Sj\"olin reduction (see Carleson-Sj\"olin \cite{CS72}). Let $\mc{M}'$ be a hyperplane of $\R^n$ intersecting $\B^n_{\epsilonm}$. Let $a(\bfx)$ be a smooth function supported on $\B^n_{\epsilonm}$ satisfying 
 \begin{equation}\label{230815e1_24}
 \dist\pnorm{
 \supp(a), \mc{M}'
  }>0.
 \end{equation}
To prove  \eqref{230407e1_9}, it suffices to prove 
\begin{equation}\label{230407e1_9z}
\norm{R^{(\mc{M}, \mc{M}')}_N f}_{L^p(\mc{M})} \lesim_{\mc{M}, \mc{M}', p, a, \epsilon} N^{-\frac{n}{p}+\epsilon} \norm{f}_{L^p(\mc{M}')},
\end{equation}
for all $f$ supported on $\mc{M}'\cap\B^n_{\epsilonm}$, where 
\begin{equation}
R_N^{(\mc{M}, \mc{M}')} f(\bfx):=\int_{\mc{M}'}
e^{i N\dist(\bfx, \bfy)} a(\bfx) f(\bfy) d\mc{H}^{n-1}(\bfy).
\end{equation}
We will call $R^{(\mc{M}, \mc{M}')}_N$ a \underline{reduced Carleson-Sj\"olin operator} on the manifold $\mc{M}$.\\

We will see below that the range of $p$ for which \eqref{230407e1_9z} holds often determines curvature properties of the manifold $\mc{M}$ near the origin. \\

\subsection{Nikodym problems on manifolds}\label{230617sub1_4}

Take the manifold $\mc{M}$ as in Subsection \ref{230814sub1_3}. Recall that $\epsilonm>0$ is a small real number that is allowed to be sufficiently small depending on $\mc{M}$, and that we use $\dist$ to denote its distance function.  For $\bfx\in \B^n_{\epsilonm/2}$, we use $\gamma_{\bfx}$ to denote the portion of a geodesic passing through $\bfx$ that lies in $\B^n_{\epsilonm}$. Moreover, for $\lambda\in (0, 1)$, denote 
\begin{equation}
\gamma_{\bfx, \lambda-\mathrm{trun}}:=\{\bfx'\in \gamma_{\bfx}: \dist(\bfx, \bfx')\ge 1-\lambda\}.
\end{equation}

\begin{definition}[Nikodym set, Sogge \cite{Sog99}]
Let $\lambda\in (0, 1)$. A set $E\subset \R^n$ is said to be a $\lambda$-Nikodym set if 
\begin{equation}
\mc{L}^n( \{\bfx\in \B^n_{\epsilonm/2}: \text{ There exists } \gamma_{\bfx} \text{ such that } 
\gamma_{\bfx, \lambda-\mathrm{trun}}\subset E
\})>0.
\end{equation}
A set $E$ is said to be Nikodym if it is $\lambda$-Nikodym for every $\lambda<1$. \footnote{The definition of Nikodym sets has a slightly different formulation from that in Sogge \cite{Sog99}, but they are essentially the same. }
\end{definition}

For $\delta>0$, let $\gamma_{\bfx}^{\delta}$ be the $\delta$-neighborhood of $\gamma_{\bfx}$. Similarly, we define $\gamma_{\bfx, \lambda-\mathrm{trun}}^{\delta}$ 

\begin{definition}[Nikodym maximal function, Sogge \cite{Sog99}]\label{230905defi1_11}
Let $\delta\in (0, 1)$ and  $\lambda\in (0, 1)$. For a function $f$ defined on $\mc{M}$ and $\bfx\in \B^n_{\epsilonm/2}$, define 
\begin{equation}
f_{\delta}^*(\bfx):=\sup_{\gamma_{\bfx}} \delta^{-1} \int_{\gamma_{\bfx}^{\delta}} |f|,
\end{equation}
and call it the Nikodym maximal function.  Moreover, define 
\begin{equation}
\mc{N}_{\delta, \lambda}f(\bfx):=\sup_{\gamma_{\bfx}} \delta^{-1} \int_{\gamma_{\bfx, \lambda-\mathrm{trun}}^{\delta}} |f|,
\end{equation}
and call it the $\lambda$-Nikodym maximal function. 
\end{definition}

\begin{theorem}[Sogge \cite{Sog99}, Xi \cite{Xi17}]\label{230906theorem1_12}
 Assume that the manifold $\mc{M}$ with dimension $n\ge 3$ has a constant sectional curvature. Then for sufficiently small $\epsilonm>0$ depending on $\mc{M}$, it holds that 
\begin{equation}
\norm{f^*_{\delta}}_{L^q(\B^n_{\epsilonm/2})}\lesim_{q, p, \mc{M}, \epsilon} 
\delta^{
1-\frac{n}{p}-\epsilon
}
\norm{f}_{L^p(\mc{M})},
\end{equation}
for all $\delta\in (0, 1), \epsilon>0$ and all 
\begin{equation}
1\le p\le \frac{n+2}{2}, q=(n-1)p'.
\end{equation}
Consequently, every Nikodym set in $\mc{M}$ must have Minkowski dimension at least $\frac{n+2}{2}$. 
\end{theorem}

In Sogge \cite{Sog99}, the author, after investigating bounds for Nikodym maximal operators for three dimensional manifolds of constant curvatures, also considered manifolds whose sectional curvatures are not constant. To state Sogge's result, let us first recall several concepts from Riemannian geometry. 

Let $\mc{M}$ be a three-dimensional manifold as in Subsection \ref{230814sub1_3} with Riemannian metric $g=\{g_{ij}\}_{1\le i, j\le 3}$. Denote by $\ricci$ the Ricci tensor on $\mc{M}$, which is a $(0, 2)$-tensor. Denote  
\begin{equation}
\bar{\ricci}: T\mc{M}\to T\mc{M}
\end{equation}
which satisfies 
\begin{equation}
g(
\bar{\ricci}(X_1), X_2
):= \ricci(X_1, X_2), \ \ \forall X_1, X_2\in T\mc{M}.
\end{equation}

\begin{definition}[Chaotic curvature, Sogge \cite{Sog99}]\label{230821defi1_11}\footnote{The definition of chaotic curvature here is formulated slightly differently from Sogge's, see \cite[Definition 3.1]{Sog99}.}
Let $\gamma$ be a geodesic parametrized by arclength with $\gamma(0)=0\in \mc{M}$. Take a unit vector $X(0)\in T_{\gamma(0)}\mc{M}$ with $\dot{\gamma}(0)\perp X(0)$. Let $X(t)\in T_{\gamma(t)}\mc{M}$ be the parallel transport of $X(0)$ along $\gamma$. Denote 
\begin{equation}
Y(t):=\bar{\ricci}(X(t)),
\end{equation}
and let $Y^{\perp}(t)$ be the projection of $Y(t)$ to the orthonormal compliment of the space spanned by $X(t)$ and $\dot{\gamma}(t)$. We say that the manifold $\mc{M}$ satisfies the chaotic curvature at the origin, if 
\begin{equation}
|Y^{\perp}(0)|+ 
|\nabla_{\dot{\gamma}}
Y^{\perp}(0)
|\neq 0,
\end{equation}
for all geodesics $\gamma$ passing through the origin, and all $X(t)$ given above. 
\end{definition}

\begin{theorem}[Sogge \cite{Sog99}]
Let $\mc{M}$ be a three-dimensional manifold as in Subsection \ref{230814sub1_3}. Assume that $\mc{M}$ satisfies the chaotic curvature condition at the origin, then every Nikodym set on $\mc{M}$ must have Minkowski dimension $\ge 7/3$.
\end{theorem}

\bigskip

\noindent {\bf Notations. } We list notations that are used in the introduction and in the rest of the paper. 
\begin{enumerate}
\item For $\epsilon>0$ and $\bfx\in \R^n$, we let $\B^n_{\epsilon}(\bfx)$ denote the ball of radius $\epsilon$ in $\R^n$ centered at $\bfx$. If $\epsilon=1$, we often abbreviate $\B^n_1(\bfx)$ to $\B^n(\bfx)$; if $\bfx=0$, then we often abbreviate $\B^n_{\epsilon}(\bfx)$ to $\B^n_{\epsilon}$.  
\item For $\bfx, \bfy\in \R^n$, their last components often play a distinct role compared with the first $n-1$ one, and therefore we often write $\bfx=(x, t), \bfy=(y, \tau)$, with $x, y\in \R^{n-1}$. 
\item For a vector $\bfv\in \R^n$, we use $|\bfv|$ to denote its standard Euclidean length. For a manifold $\mc{M}$ with metric tensor $g$, we use $|\bfv|$ to denote its length $\sqrt{g_p(\bfv, \bfv)}$ for $p\in \mc{M}, \bfv\in T_p\mc{M}$. 
\item We will use $\widehat{f}$ or $(f)^{\wedge}$ to denote the Fourier transform of $f$. 
\item For a set $E$, we will use $\mathbbm{1}_E$ to denote its indicator function. 
\item For a set $E\subset \R^n$ and $\delta>0$, we use $\mc{N}_{\delta}(E)$ to denote the $\delta$-neighborhood of $E$. 
\item For two non-negative real numbers $a, b$ and a parameter $p$, we use $a\lesim_p b$ to mean that there exists a constant $C_p$ depending only on $p$ such that $a\le C_p b$. For instance, let $T: L^p(\R^n)\to L^p(\R^n)$ be an operator. We use 
\begin{equation}
\norm{T f}_{L^p(\R^n)}\lesim_p \norm{f}_{L^p(\R^n)}
\end{equation}
to mean that there exits $C_p\in \R$ depending only on the Lebesgue exponent $p$ such that  
\begin{equation}
\norm{T f}_{L^p(\R^n)}\le C_p \norm{f}_{L^p(\R^n)}
\end{equation}
for all functions $f$. If it is clear from the context which parameters $p$ are involved, then we often abbreviate $a\lesim_p b$ to $a\lesim b$. Similarly, we define $a\gtrsim_p b$. Moreover, we use $a\simeq_p b$ to mean $a\lesim_p b$ and $a\gtrsim_p b$. 
\item For a set $E\subset \R^n$, we use $\mc{L}^n(E)$ to refer to its Lebesgue measure. 
\item For $p\in [1, \infty]$, we use $p'$ to denote its H\"older dual, that is, $1/p+1/p'=1$. 
\item All manifolds in the current paper are assumed to be smooth.
\item For a rectangle $\Box\subset \R^n$ and $r>0$, we use $r\Box$ to denote the rectangle with the same center as $R$, but dilated by $r$ with respect to the center of $R$. 
\item We try to make sure that the same notations are not repeatedly used within a same section, unless otherwise specified. However, if a same notation appears in different sections, it may refer to different things. 
\end{enumerate}

\bigskip 

\noindent {\bf Acknowledgements.} S. G. would like to thank Sigurd Angenent, Yanyan Li and Ruobing Zhang for discussions on relevant Riemannian geometry materials in the paper. Part of the work was done during Guo's multiple visits to Rutgers University; he would like to thank Yanyan Li for the invitations and the hospitality. S. D. is partly supported by NSF of China (No. 11971244 and No. 12071338); S. G. is partly supported by NSF-2044828; R. Z. is supported by NSF DMS-2207281(transferred from DMS-1856541), NSF DMS-2143989 and the Sloan Research Fellowship.

\section{Statement of main results}\label{231024section2}

The first result is about the connections among the operators introduced above, and is well known. Roughly speaking, it says that 
\begin{equation}
\begin{split}
& \text{H\"ormander-type oscillatory integrals} \\\implies & \text{reduced Carleson-Sj\"olin on manifolds}\\
\implies &\lambda\text{-Nikodym maximal functions. }
\end{split}
\end{equation}
Let us be more precise. Let $\mc{M}$ be a Riemannian manifold as in Subsection \ref{230814sub1_3} of dimension $n\ge 3$, and let $\epsilon_{\mc{M}}>0$ be a small constant that is allowed to depend on $\mc{M}$. Recall reduced Carleson-Sj\"olin operators defined in \eqref{230407e1_9z}.
\begin{theorem}\label{230617thm2_1}
\begin{enumerate}
\item[a)] Reduced Carleson-Sj\"olin operators satisfy H\"ormander's non-degeneracy conditions. 
\item[b)] Assume that 
\begin{equation}
\norm{R^{(\mc{M}, \mc{M}')}_N f}_{L^p(\mc{M})} \lesim_{\mc{M}, \mc{M}', p, a} N^{-\frac{n}{p}} 
\norm{f}_{L^p(\mc{M}')},
\end{equation}
for some $p>1$, every $N\ge 1$, every hyperplane $\mc{M}'$ intersecting $\B^n_{\epsilon_{\mc{M}}}$ and every smooth amplitude function satisfying the separation condition \eqref{230815e1_24}. Then the $\lambda$-Nikodym maximal function $\mc{N}_{\delta, \lambda}$ is of restricted weak type $(q, q)$ with norm 
\begin{equation}
\lesim_{\lambda, \mc{M}, p} \delta^{
-2(\frac{n}{q}-1)
}, 
\end{equation}
where $q:=(p/2)'$, for every $\lambda<1$. 
\end{enumerate}
\end{theorem}

The formulation of item b) in Theorem \ref{230617thm2_1} is taken from Wisewell's Theorem \ref{230706thm1_8} in Subsection \ref{230903subsection1_2}. Indeed, the proof of item b) is also essentially the same as that of Theorem \ref{230706thm1_8}. 

We should also mention that curved Kakeya maximal operators are also closely related to $\lambda$-Nikodym maximal operators. Indeed in many interesting cases, they are essentially the same objects. Later we will use this relation, for instance in the proof of Theorem \ref{230323theorem3_4}. However, due to purely technical reasons, we do not have a clean way to state such relations. 

In the appendix, we will explain a key difference between curved Kakeya maximal operators and $\lambda$-Nikodym maximal operators.\\

Before stating the next result, let us recall the result in \cite{GWZ22}, as stated in Theorem \ref{231013theorem1_5}. In \cite{GWZ22}, the authors considered H\"ormander-type oscillatory integrals, and showed that if the phase function $\phi(\bfx; y)$ satisfies Bourgain's condition everywhere, then all the current techniques that have been developed so far in the study of the Fourier restriction conjecture can also be applied to H\"ormander-type oscillatory integrals. 

Indeed, the same principle applies also to the study of curved Kakeya problems. Let $n\ge 3$. For phase functions $\phi$ satisfying the same assumptions as in Theorem \ref{231013theorem1_5}, the associated curved Kakeya sets satisfy the same dimension bounds as what Hickman, Rogers and Zhang \cite{HRZ22} obtained for the traditional straight line Kakeya sets. Let us be more precise. Denote 
\begin{equation}
q_{\mathrm{HRZ}}(n):=1+\min _{2 \leqslant k \leqslant n} \max \left\{\frac{2 n}{(n-1) n+(k-1) k}, \frac{1}{n-k+1}\right\},
\end{equation}
which is the exponent that appeared in \cite[Theorem 1.2]{HRZ22}. Moreover, denote 
\begin{equation}
p_{\mathrm{HRZ}}(n):= (q_{\mathrm{HRZ}}(n))'.
\end{equation}
Then one can follow the same argument as in \cite{HRZ22}, use the (strong) polynomial Wolff axioms for $\phi$ obtained in \cite[Theorem 1.2, Theorem 6.2]{GWZ22}, a standard equivalence argument (see for instance \cite[Proposition 22.6]{Mat15}) and obtain 
\begin{equation}\label{230815e2_4}
\norm{
\mc{K}_{\delta}^{(\phi)} f
}_{
L^p(\B^{n-1}_{\epsilon_{\phi}})
}
\lesim_{\phi, \epsilon, p}
\delta^{-\epsilon}
\delta^{
-(n-1-\frac{n}{p'})
}
\norm{f}_{L^p(\R^n)},
\end{equation}
for every 
\begin{equation}
1<p\le p_{\mathrm{HRZ}}(n),
\end{equation}
 every $\epsilon>0, \delta\in (0, 1)$. In particular, for every $p$ in the above range, the exponent of $\delta$ on the right hand side is sharp. Moreover, it seems reasonable to conjecture that \eqref{230815e2_4} holds for all $1<p\le n$, that is, the same range of $p$ as in the maximal Kakeya conjecture. \\

 As a corollary of \eqref{230815e2_4} (see \cite[Theorem 22.9]{Mat15}), we obtain that every curved Kakeya set associated to $\phi$ must have Hausdorff dimension at least 
\begin{equation}\label{230815e2_5}
n- \pnorm{n-1-\frac{n}{q_{\mathrm{HRZ}}(n)}} p_{\mathrm{HRZ}}(n):= d_{\mathrm{HRZ}}(n).
\end{equation}
If we let $n\to \infty$, then (see \cite[Subsection 9.2]{HRZ22})
\begin{equation}\label{230905e2_9}
\eqref{230815e2_5}= (2-\sqrt{2})n+O(1).
\end{equation}
It is worth mentioning that the asymptotic \eqref{230905e2_9} obtained by Hickman, Rogers and Zhang \cite{HRZ22} for the dimensions of the traditional Kakeya sets (for straight lines) is the same as that of Katz and Tao \cite{KT02}. Moreover, these two results \cite{HRZ22} and \cite{KT02} together give the currently best known results for the dimensions of the traditional Kakeya sets in high dimensions. \\

\bigskip

Let us state our next result, which involves bounds for reduced Carleson-Sj\"olin operators defined in \eqref{230407e1_9z}, and for $\lambda$-Nikodym maximal operators in Definition \ref{230905defi1_11}. Let $\mc{M}$ be a Riemannian manifold as in Subsection \ref{230814sub1_3} of dimension $n\ge 3$. let $\epsilon_{\mc{M}}>0$ be a small constant that is allowed to depend on $\mc{M}$.

\begin{theorem}\label{230323theorem3_4}
\begin{enumerate}
\item[(a)]  Assume that $\mc{M}$ is of a constant sectional curvature. Recall the definition of $p_{\mathrm{GWZ}}(n)$ in \eqref{230409e1_10}. Then 
\begin{equation}\label{230322e3_3}
\Norm{
R_N^{(\mc{M}, \mc{M}')} f
}_{L^p(\mc{M})} \lesim_{\mc{M}, \mc{M}', a, p, \epsilon} N^{-\frac{n}{p}+\epsilon} \norm{f}_{L^p(\mc{M}')}, 
\end{equation}
for all $p\ge p_{\mathrm{GWZ}}(n), \epsilon>0, N\ge 1$, all hyperplanes $\mc{M}'$ intersecting $\B^n_{\epsilonm}$, and all functions $a(\bfx)$ satisfying the separation condition \eqref{230815e1_24}. 
\item[(b)] Assume $\mc{M}$ is analytic and its sectional curvature is not constant. Then we can find $p> \frac{2n}{n-1}$, a small positive $\epsilonm>0$, a hyperplane $\mc{M}'$ intersecting $\B^n_{\epsilonm}$, a smooth function $a(\bfx)$ satisfying the separation condition \eqref{230815e1_24}, such that \eqref{230322e3_3} may fail.
\item[(c)] Assume that $\mc{M}$ is of  a constant sectional curvature. Then 
\begin{equation}
\norm{
\mc{N}_{\delta, \lambda}f
}_{
L^p(
\B^n_{
\epsilon_{\mc{M}}/2
}
)
}
\lesim_{
p, \epsilon, \lambda, \mc{M}
}
\delta^{-\epsilon}
\delta^{
-(n-1-\frac{n}{p'})
}
\norm{f}_{
L^p(\mc{M})
},
\end{equation}
for every $\lambda<1$, $\epsilon>0, \delta\in (0, 1)$ and every 
\begin{equation}
1< p\le p_{\mathrm{HRZ}}(n).
\end{equation}
Consequently, every Nikodym set must have Minkowski dimension $\ge d_{\mathrm{HRZ}}(n)$.

\end{enumerate}

\end{theorem}

Recall Carleson-Sj\"olin operators defined in \eqref{230509e1_1}:
\begin{equation}\label{230509e1_1bb}
T_N^{(\mc{M})} f(\bfx):=\int_{\mc{M}} 
e^{i N\dist(\bfx, \bfy)} a(\bfx; \bfy) f(\bfy) d\bfy,
\end{equation}
where $a(\bfx; \bfy): \mc{M}\times \mc{M}\to \R$ is a compactly supported smooth function supported on $\mathbb{B}^n_{\epsilonm}\times \mathbb{B}^n_{\epsilonm}$, and supported away from the diagonal. By item a) of Theorem \ref{230323theorem3_4} and Fubini's theorem (Carleson-Sj\"olin reduction as explained in Subsection \ref{230814sub1_3}), we  obtain
\begin{corollary}
Let $\mc{M}$ be a Riemannian manifold of constant sectional curvature. Let $n\ge 3$ be the dimension of $\mc{M}$. Then 
\begin{equation}\label{230906e2_14}
\norm{T^{(\mc{M})}_N f}_{L^p(\mc{M})} \lesim_{\mc{M}, p, a, \epsilon} N^{-\frac{n}{p}+\epsilon} \norm{f}_{L^p(\mc{M})},
\end{equation}
for all 
\begin{equation}\label{230906e2_15}
p\ge p_{\mathrm{GWZ}}(n),
\end{equation}
all $\epsilon>0$ and $N\ge 1$. 
\end{corollary}
In particular, if we take $\mc{M}$ to be the standard Euclidean space, then the bound \eqref{230906e2_14} with the range \eqref{230906e2_15} is precisely what Guo, Wang and Zhang \cite{GWZ22} obtained for the (Euclidean) Bochner-Riesz operator. Moreover, this bound is the currently best known bound for the Bochner-Riesz conjecture (stated in Subsection \ref{230814sub1_3}). Our bound \eqref{230906e2_14} generalizes that of Guo, Wang and Zhang \cite{GWZ22} for the Euclidean distance functions to distance functions on manifolds of constant curvatures. \\

Item (b) in Theorem \ref{230323theorem3_4} says that if $\mc{M}$ does not have constant sectional curvature, then the reduced Carleson-Sj\"olin operator will not satisfy as good bounds as those on manifolds of constant sectional curvature. By adapting the argument in Sogge \cite[page 290]{Sog17}, one can also show that for manifolds whose sectional curvatures are not constant, the estimate \eqref{230906e2_14} also fails for some $p> 2n/(n-1)$.

Item (c) in Theorem \ref{230323theorem3_4} is an improvement over the result of Xi \cite{Xi17} as stated in Theorem \ref{230906theorem1_12}. Recall that  Sogge \cite{Sog99} proved that Nikodym sets on three dimensional manifolds of constant curvatures must have Minkowski dimension $\ge 5/2$. Xi \cite{Xi17} generalized this result to higher dimensions, and proved that Nikodym sets on $n$ dimensional manifolds of constant curvatures must have Minkowski dimension $\ge (n+2)/2$.

It is reasonable to conjecture that for manifolds of constant sectional curvatures, the bound \eqref{230322e3_3} holds for the largest possible range 
\begin{equation}
p\ge \frac{2n}{n-1},
\end{equation}
the same as the range \eqref{230813e1_8} in the Fourier restriction conjecture. It is also reasonable to conjecture that every Nikodym set on manifolds of constant sectional curvatures must have a full Hausdorff dimension. Theorem \ref{230323theorem3_4} provides some partial evidence for such conjectures.  \\

So far we have studied the curved Kakeya problem for $\phi$ satisfying Bourgain's condition, bounds for (reduced) Carleson-Sj\"olin operators on manifolds of constant sectional curvatures, and Nikodym problems on manifolds of constant sectional curvatures. The settings in these problems are perhaps the ``best" possible in the sense that we conjecture all these problems would eventually have the same answers to their Euclidean counterparts. 

Moreover, Bourgain \cite{Bou91}, Wisewell \cite{Wis05}, Minicozzi and Sogge \cite{MS97}, Bourgain and Guth \cite{BG11}, Guth, Hickman and Iliopoulou \cite{GHI19} and Hickman and Iliopoulou \cite{HI22} have constructed ``worst" possible examples in these problems. 

Our next goal is to study ``intermediate" examples. We will only consider the case $n=3$. \\

Let $\phi(x, t; y)$ be a phase function satisfying H\"ormander's non-degeneracy conditions. Let $\epsilon_{\phi}>0$ be a small constant depending on $\phi$. Fix $(x_0, t_0; y_0)\in \B^{n-1}_{\epsilon_{\phi}}\times \B^1_{\epsilon_{\phi}}\times \B^{n-1}_{\epsilon_{\phi}}$. Let $X_0(t)$ be the unique solution to 
\begin{equation}\label{230816e2_14hh}
\nabla_y \phi(X_0(t)+x_0, t+t_0; y_0)=\nabla_y \phi(x_0, t_0; y_0).
\end{equation}   
The existence and uniqueness of the solution are guaranteed by H\"ormander's non-degeneracy conditions, and that $\epsilon_{\phi}$ is chosen sufficiently small. 
Denote 
\begin{equation}
\phi_0(x, t; y):=\phi(x+x_0, t+t_0; y+y_0)-\phi(x_0, t_0; y+y_0).
\end{equation}
Moreover, denote 
\begin{equation}
D_{ij}(t):= 
\partial_{y_i}
\partial_{y_j}
\phi_0(
X_0(t), t; 0
), \ \ 1\le i, j\le 2,
\end{equation}
and 
\begin{equation}
\mathfrak{D}(t):=
\det
\begin{bmatrix}
D_{11}(t), & D_{12}(t)\\
D_{21}(t), & D_{22}(t)
\end{bmatrix}
\end{equation}
Take an integer $k\ge 4$. We say that the phase function $\phi$ is of a \underline{contact order} $\le k$ at the point $(x_0, t_0)$ if the matrix 
\begin{equation}\label{230906e2_21}
\begin{bmatrix}
\mathfrak{D}'(0), & \mathfrak{D}''(0), & \dots, & \mathfrak{D}^{(k)}(0),\\
D_{11}'(0), & D_{11}''(0), & \dots, & D_{11}^{(k)}(0)\\
D_{12}'(0), & D_{12}''(0), & \dots, & D_{12}^{(k)}(0)\\
D_{22}'(0), & D_{22}''(0), & \dots, & D_{22}^{(k)}(0)
\end{bmatrix}
\end{equation}
has rank $4$. The matrix \eqref{230906e2_21} already implicitly appeared in Bourgain's work \cite{Bou91}, see equation (6.160) in \cite[page 364]{Bou91}.

\begin{theorem}\label{230816theorem2_4}
Let $k\in \N$ and $k\ge 4$. Take $n=3$ and $\phi(x, t; y)$ a smooth phase function of the normal form \eqref{230324e1_4}, that is, 
\begin{equation}\label{230816e2_17}
\phi(\bfx; y)=x\cdot y+ t\inn{y}{Ay}+ O(|t| |y|^3+ |\bfx|^2 |y|^2).
\end{equation} 
Assume that $A$ is positive definite and that  the contact order of $\phi$ at the origin $x=0, t=0, y=0$ is $\le k$. 
Then there exists $\epsilon_{\phi}>0$ depending on $\phi$, and
\begin{equation}
\epsilon_k= \frac{1}{9k-6},
\end{equation}
 such that 
\begin{equation}
\norm{T^{(\phi)}_N f}_{L^p(\R^3)} \lesim_{\phi, a, p, \epsilon} N^{-\frac{3}{p}+\epsilon} \norm{f}_{L^p(\R^{2})}
\end{equation}
holds for all
\begin{equation}\label{231023e2_25zz}
p\ge \frac{10}{3}-\epsilon_k,
\end{equation}
all $\epsilon>0, N\ge 1$, and all smooth amplitude functions $a$ supported in $\B^{2}_{\epsilon_{\phi}}\times \B^1_{\epsilon_{\phi}}\times \B^{2}_{\epsilon_{\phi}}$.
\end{theorem}

If we take $k=2$, then the range \eqref{231023e2_25zz} becomes $p\ge 3.25$, the range of Guth \cite{Gut16}. This is not a coincidence, and in the proof of Theorem \ref{230816theorem2_4} we will generalize the polynomial Wolff axioms in \cite{Gut16} to H\"ormander-type oscillatory integrals.  

If we take $k=4$, the smallest value of $k$ that is allowed by Theorem \ref{230816theorem2_4}, then the range \eqref{231023e2_25zz} becomes $p\ge 3.3$, the range of Bourgain and Guth \cite[Theorem 2]{BG11}. This seems unlikely not a coincidence.

\bigskip

The notion of contact orders also in some sense already appeared in Sogge's work \cite{Sog99}.

\begin{theorem}\label{230711thm2_5}
Let $\mc{M}$ be a smooth manifold as in Subsection \ref{230814sub1_3} of dimension three. Assume that $\mc{M}$ satisfies Sogge's chaotic curvature condition at the origin. Then there exists $\epsilon_{\mc{M}}>0$ depending only on $\mc{M}$, such that if we let 
\begin{equation}\label{230816e2_25}
\phi(x, t; y):=\dist((x, t), (y, y_3)),
\end{equation}
where $y_3\in (0, \epsilon_{\mc{M}})$, then $\phi$ has a contact order $\le 4$ at $(x, t; y)=(0, 0; 0)$. 
\end{theorem}

As an immediate corollary of Theorem \ref{230711thm2_5} and Theorem \ref{230816theorem2_4}, we obtain 
\begin{corollary}
Let $\mc{M}$ be a three dimensional manifold as in Theorem \ref{230711thm2_5}. Then the Carleson-Sj\"olin operator $T^{(\mc{M})}_N$ on $\mc{M}$ satisfies 
\begin{equation}
\norm{T^{(\mc{M})}_N f}_{L^p(\mc{M})} \lesim_{\mc{M}, p, a, \epsilon} N^{-\frac{3}{p}+\epsilon} \norm{f}_{L^p(\mc{M})},
\end{equation}
for all $p\ge 3.3$, $N\ge 1$ and $\epsilon>0$. 
\end{corollary}

Let us briefly explain the geometric intuition behind chaotic curvatures. Interested readers should read \cite[Section 3]{Sog99}, especially Proposition 3.2 there, in order to have a better understanding of Sogge's chaotic curvature.

Let $\mc{M}$ be a three dimensional manifold. Assume that $\mc{M}$ satisfies Sogge's chaotic curvature condition. Take a coordinate patch, and express it in Fermi coordinates (see Definition \ref{230906defi3_1} below). Take a coordinate plane in the Fermi coordinate. Then $\mc{M}$ satisfying chaotic curvature conditions means that no any geodesic in $\mc{M}$ can have a contact order $\ge 5$ with the given coordinate plane. In other words, if we have a geodesic that is first order tangent, second order tangent and third order tangent to the coordinate plane, then it can not be fourth order tangent to it.

\section{Proof of Theorem \ref{230711thm2_5}}

\subsection{Preliminaries in Riemannian geometry}

Let $\mc{M}$ be a smooth Riemannian manifold of dimension $n$.  Let $g$ be the metric tensor on $\mc{M}$. Let $\{e_i\}_{i=1}^n$ be a basis of the tangent space. Denote 
\begin{equation}
g_{ij}:=g({e_i}, {e_j}). 
\end{equation}
We use $(g^{kl})_{1\le k, l\le n}$ to denote the inverse of $g$. Denote
\begin{equation}
g_{ij, k}:=\nabla_{e_k} g_{ij}, \ \ g_{ij, kl}:=\nabla_{e_k}\nabla_{e_l}g_{ij}.
\end{equation}
The covariant derivative of a function $f$ defined on $\mc{M}$ is defined to be 
\begin{equation}
\bar{\nabla} f:=\sum_{ij} g^{ij}\cdot (\nabla_{e_j} f) e_i.
\end{equation}
Here we add a bar on top of $\nabla$ just to distinguish $\bar{\nabla}f$ from $\nabla f$; we want $\bar{\nabla}f$ to be a vector field ($(1, 0)$-tensor), while $\nabla f$ is a one-form ($(0, 1)$-tensor).  The covariant derivative of a vector field $Y$ ($(1, 0)$-tensor) is defined to be the $(1, 1)$-tensor $\nabla Y$ given by 
\begin{equation}
(\nabla Y)(X):=\nabla_X Y,
\end{equation}
where $X$ is a vector field. We then define the Hessian of a function $f$ to be a $(0, 2)$-tensor given by 
\begin{equation}\label{230820e3_4}
(\mathrm{Hessian} f)(X, Y):=g(
\nabla_X \bar{\nabla}f, Y
)
\end{equation}
where $X, Y$ are vector fields. \\

\underline{Christopher symbols} $\Gamma^k_{ij}$ are defined by \footnote{In the literature and in textbooks, the notation $\Gamma^k_{\ ij}$ is also often used. }
\begin{equation}
\nabla_{e_i}e_j=\sum_k \Gamma^k_{ij} e_k.
\end{equation}
In local coordinates they can be computed by  
\begin{equation}
\Gamma^k_{ij}= 
\frac12 \sum_m g^{km}(
g_{mi, j}+g_{mj, i}-g_{ij, m}
).
\end{equation}
Let $R: T\mc{M}\times T\mc{M}\times T\mc{M}\to T\mc{M}$ be the 
\underline{Riemannian curvature tensor}, defined by 
\begin{equation}
R(X, Y) Z=\nabla_X \nabla_Y Z-\nabla_Y \nabla_X Z-\nabla_{[X, Y]} Z.
\end{equation}
 Define $R_{ijk}^{\ \ \ l}$ by 
\begin{equation}
R(e_i, e_j) e_k=\sum_l R_{ijk}^{\ \ \ l} e_l.
\end{equation}
If we denote 
\begin{equation}
R_{ijkl}=g(R(e_i, e_j)e_k, e_l),
\end{equation}
then we have 
\begin{equation}
R_{ijkl}= \sum_p R_{ijk}^{\ \ \ p} g_{pl}.
\end{equation}
In local coordinates, we have 
\begin{equation}\label{230821e3_10zz}
R^{\ \ \ q}_{lki}=-\pnorm{
\partial_k \Gamma^q_{li}-
\partial_l \Gamma^q_{ki}
-\sum_p \Gamma^q_{lp}\Gamma^p_{ki}+\sum_p \Gamma^q_{kp}\Gamma^p_{li}
}.
\end{equation}
The \underline{Ricci tensor} is defined by 
\begin{equation}\label{230821e3_11}
\ricci(Y, Z):=\mathrm{tr}(X\mapsto R(X, Y)Z),
\end{equation}
where the right hand side means the trace of the map $X\mapsto R(X, Y)Z$.

\begin{definition}[Fermi coordinates]\label{230906defi3_1} Take a three dimensional manifold $\mc{M}$. Let $\gamma: [0, L]\to \mc{M}$ be an arclength parametrized geodesic. At $s=0$, pick two unit tangent vectors $E_1(0), E_2(0)\in T_{\gamma(0)}\mc{M}$ that form an orthonormal basis together with $\dot{\gamma}(0)$. Let $E_1(s), E_2(s)$ be the parallel transports of $E_1(0), E_2(0)$ along $\gamma(s)$, respectively. One then assigns Fermi coordinates $(x_1, x_2, t)$ to a point if 
\begin{enumerate}
\item this point lies on the geodesic passing through the point $\gamma(t)$ with the unit tangent vector at this point being 
\begin{equation}
\frac{1}{|(x_1, x_2)|} \pnorm{
x_1 E_1(t)+ x_2 E_2(t)
}
\end{equation}
where $|(x_1, x_2)|$ means Euclidean norm;
\item the distance from this point to $\gamma(t)$ is $|(x_1, x_2)|$. 
\end{enumerate}
\end{definition}

In Fermi coordinates, metrics satisfy (see \cite[page 5]{Sog99})
\begin{enumerate}
\item
\begin{equation}\label{230713e3_4}
\begin{bmatrix}
g_{11}(\bfx), & g_{12}(\bfx), & g_{13}(\bfx)\\
g_{21}(\bfx), & g_{22}(\bfx), & g_{23}(\bfx)\\
g_{31}(\bfx), & g_{32}(\bfx), & g_{33}(\bfx)
\end{bmatrix}
\begin{bmatrix}
x_1\\
x_2\\
0
\end{bmatrix}
=\begin{bmatrix}
x_1\\
x_2\\
0
\end{bmatrix},
\ \ \bfx=(x_1, x_2, t),
\end{equation}
\item 
\begin{equation}\label{230713e3_5}
g_{ij}(\bfx)=\delta_{ij}, \ \ \text{whenever } x_1=x_2=0.
\end{equation}
In other words, the matrix $(g_{i, j})_{1\le i, j\le 3}$ is always the identity matrix along the geodesic $\gamma$. 
\item 
\begin{equation}\label{230713e3_6}
\frac{\partial}{\partial x_i} g_{jk}(\bfx)=0, \ \ i=1, 2, \ \forall j, k, \ \ \text{whenever } x_1=x_2=0.
\end{equation}
\end{enumerate}
Indeed, \eqref{230713e3_6} can also be written as 
\begin{equation}\label{230713e3_7}
\frac{\partial}{\partial x_i} g_{jk}(\bfx)=0, \ \ \ \forall i, j, k, \ \ \text{whenever } x_1=x_2=0.
\end{equation}

\bigskip

In the end of this subsection, we state one lemma that will be used multiple times in the rest of the paper. 

Let $\epsilon_{\mc{M}}>0$ be a small constant depending on $\mc{M}$. Fix $\epsilon\in (0, \epsilon_{\mc{M}})$. Denote  
\begin{equation}
\phi_{\epsilon}(x, t; y):=\dist((x, t), (y, \epsilon)).
\end{equation}
We always consider $(x, t)$ in a small neighborhood of the origin, and $(y, \epsilon)$ in a small neighborhood of $(0, \epsilon)$.

\begin{lemma}\label{230411lemmaA_1}
Fix a point $(y_0, \epsilon)$. Let $\gamma$ be a geodesic passing through $(y_0, \epsilon)\in \mc{M}$. Then 
\begin{equation}
(
\partial_{y_1} \phi_{\epsilon}(\bfx; y_0), \dots, 
\partial_{y_{n-1}} \phi_{\epsilon}
(\bfx; y_0)
)
\end{equation}
 stays constant when $\bfx=(x, t)$ moves along $\gamma$. 
\end{lemma}

\begin{proof}[Proof of Lemma \ref{230411lemmaA_1}]
	We parameterize $\gamma$ by an arclength parameter $s$ such that $\gamma(0)=\bfx$ and $\gamma(L)=\bfy_0:=(y_0, \epsilon)$, where $L=d_{\mcm}(\bfx, \bfy_0)$. Pick a normal coordinate $z=(z_1,z_2, \dots,  z_n)$ centered at $\bfx$ and $\gamma'(0)=\partial_{z_1}$ at the origin point. Then for any $V=\sum_i V_i\partial_{z_i}\in T_{\bfx}\mc{M}$, the geodesic $\gamma_V$ starting at $\bfx$ with initial velocity $V$ is represented in normal coordinate by the radial line segment
	$$
	\gamma_V(s)=\exp_{\bfx}((sV_1, sV_2, \dots, sV_n)).
	$$
	In our case, there is $\gamma(s)=\exp_{\bfx}((s, 0, \dots, 0))$. In particular, $\bfy_0=\exp_{\bfx}((L, 0, \dots,  0))$ and the tangent vector
	\begin{equation}\label{tangent vector}
		\gamma'(L)=d\exp_{\bfx}|_{z=(L, 0, \dots,  0)}\partial_{z_1}.
	\end{equation}
The Riemannian distance function can be expressed explicitly in normal coordinate (see Corollary 6.12 in \cite[page 162]{Lee}, which is a corollary of the Gauss Lemma), and we have 
	$$
	d_{\mcm}(\bfx,\exp_{\bfx}(z))=\sqrt{z_1^2+z_2^2+\dots +z_n^2}.
	$$ 
	Then we can compute that
	\begin{equation}\label{gradient}
		\grad_{\bfy}d_{\mcm}(\bfx,\bfy_0)=d\exp_{\bfx}|_{z=(L, 0, \dots, 0)}\left(\sum_i \frac{z_i\partial_{z_i}}{|z|}\right)=d\exp_{\bfx}|_{z=(L,0,\cdots,0)}\partial_{z_1},
	\end{equation}
	where $\grad_{\bfy}$ is the covariant gradient taken in the $\bfy=(y, \tau)$ variable. 
	Combining equations (\ref{tangent vector}) and (\ref{gradient}), we can get 
	\begin{equation}\label{231023e3_22}
	\grad_{\bfy}d_{\mcm}(\bfx,\bfy_0)=\gamma'(L).
	\end{equation}
From this we see that $\grad_{\bfy}d_{\mcm}(\bfx,\bfy_0)$ stays constant when $\bfx$ moves along $\gamma$. 
In the end, by the chain rule we observe that 
	$$
	\partial_{y_i} \phi
	_{\epsilon}
	(\bfx; y_0)=
	\sum_{k, j}
	(\grad_{\bfy}d_{\mcm}(\bfx,\bfy_0))_{k}(g(\bfy_0))_{kj}\left(\frac{\partial \bfy}{\partial y}\right)_{ji}\Big|_{y=y_0},
	$$
which verifies that, once $\bfy_0$ is fixed,  $\partial_{y_i}\phi_{\epsilon}
(\bfx; y_0)$ stays constant when $\bfx$ moves along $\gamma$. 
\end{proof}

\subsection{Sogge's chaotic curvature condition}\label{230903section3_2}

In this part, we will prove Theorem \ref{230711thm2_5}.\\

First of all, it is elementary to see that contact orders are invariant under changes of coordinates. We therefore without loss of generality assume that we are in a Fermi coordinate.  Let $\epsilon>0$ be a small number to be chosen; its smallness will depend on the manifold $\mc{M}$. Assume that our Fermi coordinate is based on the geodesic 
\begin{equation}\label{230821e3_10}
\gamma(s)=(0, s), \ \forall s\in [0, \epsilon].
\end{equation}
We use $e_1, e_2, e_3$ to denote the coordinate vectors. 
Define vector fields on $\gamma$ by 
\begin{equation}\label{230903e3_19}
E_3(s):=\frac{\partial}{\partial \tau}\in T_{\gamma(s)}\mc{M}, \ \ 
E_i(s):=\frac{\partial}{\partial y_i}\in T_{\gamma(s)}\mc{M}, \ i=1, 2.
\end{equation}
In other words, $E_i(s)=e_i(\gamma(s)), i=1, 2, 3$, and therefore $\{E_i(s)\}_{i=1}^3$ forms an orthonormal basis for $T_{\gamma(s)}\mc{M}$.  Denote 
\begin{equation}\label{230903e3_20xx}
\phi_{\epsilon}(x, t; y):=\dist((x, t), (y, \epsilon))
\end{equation}
and 
\begin{equation}
\Phi(x, t; y, \tau):=\dist((x, t), (y, \tau)).
\end{equation}
Moreover, denote 
\begin{equation}
\phi_{0, \epsilon}(x, t; y):=\phi_{\epsilon}(x, t; y)-\phi_{\epsilon}(0, 0; y),
\end{equation}
and 
\begin{equation}\label{230903e3_23xx}
\Phi_{0}(x, t; y, \tau):=\Phi(x, t; y, \tau)-\Phi(0, 0; y, \tau).
\end{equation}
Here $(x, t)$ and $(y, \tau)$ are points on $\mc{M}$. Our goal is prove that if $\mc{M}$ satisfies Sogge's chaotic curvature condition at the origin, then the contact order of $\phi_{\epsilon}$ is $\le 4$ at $(x, t)=(0, 0), y=0$ if $\epsilon$ is chosen to be sufficiently small.\\

Before computing contact orders, let us first compute what Sogge's chaotic curvature says in the Fermi coordinate. Recall Definition \ref{230821defi1_11}. We pick the geodesic in Definition \ref{230821defi1_11} to be \eqref{230821e3_10}. The vector field $X(s)$ in Definition \ref{230821defi1_11} is the parallel transport of a unit vector $X(0)\perp \dot{\gamma}(0)$. Therefore, we can write 
\begin{equation}
X(s)= c_1 E_1(s)+ c_2 E_2(s), \ \ c_1^2+c_2^2=1.
\end{equation}
The vector field $Y(s)$ in Definition \ref{230821defi1_11} is therefore 
\begin{equation}
Y(s)= \bar{\ricci}(X(s)),
\end{equation}
and 
\begin{equation}
\begin{split}
Y^{\perp}(s)& = g\pnorm{
Y(s), c_2 E_1(s)-c_1 E_2(s)
}(c_2 E_1(s)-c_1 E_2(s))\\
&=
(c_1, c_2) 
\begin{bmatrix}
\ricci_{11}(s), & \ricci_{12}(s)\\
\ricci_{21}(s), & \ricci_{22}(s)
\end{bmatrix}
\begin{bmatrix}
0, & 1\\
-1, & 0
\end{bmatrix}
\begin{pmatrix}
c_1\\
c_2
\end{pmatrix}
(c_2 E_1(s)-c_1 E_2(s))
\end{split}
\end{equation}
where 
\begin{equation}
\ricci_{ij}(s):=g(\bar{\ricci}(E_i(s)), E_j(s))=\ricci(E_i(s), E_j(s)).
\end{equation}
Moreover, 
\begin{equation}
\begin{split}
& \nabla_{\dot{\gamma}(s)} Y^{\perp}(s)\\
& =(c_1, c_2) 
\begin{bmatrix}
\partial_s \ricci_{11}(s), & \partial_s \ricci_{12}(s)\\
\partial_s \ricci_{21}(s), & \partial_s \ricci_{22}(s)
\end{bmatrix}
\begin{bmatrix}
0, & 1\\
-1, & 0
\end{bmatrix}
\begin{pmatrix}
c_1\\
c_2
\end{pmatrix}
(c_2 E_1(s)-c_1 E_2(s)),
\end{split}
\end{equation}
where we used the fact that $E_1(s), E_2(s)$ are parallel transports along $\gamma(s)$, which says 
\begin{equation}
\nabla_{\dot{\gamma}(s)} E_j(s)=0, \ \ j=1, 2. 
\end{equation}
As we assume Sogge's chaotic curvature, we therefore can conclude that 
\begin{equation}\label{230822e3_45hh}
\begin{split}
& \anorm{
(c_1, c_2) 
\begin{bmatrix}
\ricci_{11}(0), & \ricci_{12}(0)\\
\ricci_{21}(0), & \ricci_{22}(0)
\end{bmatrix}
\begin{bmatrix}
0, & 1\\
-1, & 0
\end{bmatrix}
\begin{pmatrix}
c_1\\
c_2
\end{pmatrix}
}\\
& + 
\anorm{
(c_1, c_2) 
\begin{bmatrix}
\partial_s \ricci_{11}(0), & \partial_s\ricci_{12}(0)\\
\partial_s\ricci_{21}(0), & \partial_s\ricci_{22}(0)
\end{bmatrix}
\begin{bmatrix}
0, & 1\\
-1, & 0
\end{bmatrix}
\begin{pmatrix}
c_1\\
c_2
\end{pmatrix}
} \neq 0,
\end{split}
\end{equation}
for all $c_1, c_2$ satisfying $c_1^2+c_2^2=1$. \\

We compute $\ricci_{ij}(0)$ and  $\partial_s\ricci_{ij}(0)$. By the definition of Ricci tensors in \eqref{230821e3_11}, and that we are working with an orthonormal basis, we have 
\begin{equation}\label{230821e3_45}
\begin{bmatrix}
\ricci_{11}(s), & \ricci_{12}(s)\\
\ricci_{21}(s), & \ricci_{22}(s)
\end{bmatrix}
=
\begin{bmatrix}
R_{2112}(s)+ R_{3113}(s), & R_{3123}(s)\\
R_{3123}(s), & R_{1221}(s)+R_{3223}(s)
\end{bmatrix}
\end{equation}
where 
\begin{equation}
R_{ijkl}(s):=g(
R(E_i(s), E_j(s))E_k(s), E_l(s)
),
\end{equation}
and 
we used basic symmetries of the Riemannian curvature tensor. We continue to compute the right hand side of \eqref{230821e3_45}. 
\begin{claim}\label{230822claim3_4}
For $1\le i, j\le 2$, we have 
\begin{equation}
R_{i33j}(\bfx)=-\frac12 g_{33, ij}(\bfx),
\end{equation}
whenever $x_1=x_2=0$.
\end{claim}
\begin{proof}[Proof of Claim \ref{230822claim3_4}]
By the formula \eqref{230821e3_10zz} and the property \eqref{230713e3_6} of Fermi coordinates, we have 
\begin{equation}\label{230822e3_48cc}
\begin{split}
& R_{i33j}=R^{\ \ \ j}_{i33} =\partial_i \Gamma^j_{33}\\
&=
\frac12 \partial_i g^{jk}(g_{3k, 3}+g_{3k, 3}-g_{33, k})
+\frac12 g^{jk}(g_{3k, 3i}+g_{3k, 3i}-g_{33, ki})\\
&= \frac12 (2g_{3j, 3i}-g_{33, ji})
\end{split}
\end{equation}
To prove the claim, it remains to show that 
\begin{equation}\label{230822e3_49}
g_{3j, 3i}(\bfx)=0, \ \ 1\le i, j\le 2,
\end{equation}
whenever $x_1=x_2=0$. \\

To show this, we will differentiate both sides of \eqref{230713e3_4}, that is, 
\begin{equation}\label{230713e3_4zzz}
\begin{bmatrix}
g_{11}(\bfx), & g_{12}(\bfx), & g_{13}(\bfx)\\
g_{21}(\bfx), & g_{22}(\bfx), & g_{23}(\bfx)\\
g_{31}(\bfx), & g_{32}(\bfx), & g_{33}(\bfx)
\end{bmatrix}
\begin{bmatrix}
x_1\\
x_2\\
0
\end{bmatrix}
=\begin{bmatrix}
x_1\\
x_2\\
0
\end{bmatrix}
\end{equation}
Take $\partial^2_{x_1}$ on both sides, we obtain 
\begin{equation}
\frac{\partial^2 g}{\partial x_1^2}
\begin{bmatrix}
x_1 \\
x_2\\
0
\end{bmatrix}
+ 
2\frac{\partial g}{\partial x_1}
\begin{bmatrix}
1 \\
0\\
0
\end{bmatrix}
=0
\end{equation}
By taking $\partial_{3}$, that is, taking $\partial_t$, we further obtain 
\begin{equation}
\frac{\partial^3 g}{\partial t\partial x_1^2}
\begin{bmatrix}
x_1 \\
x_2\\
0
\end{bmatrix}
+ 
2\frac{\partial^2 g}{\partial t\partial x_1}
\begin{bmatrix}
1 \\
0\\
0
\end{bmatrix}
=0
\end{equation}
from which we conclude that 
\begin{equation}
g_{13, 13}(\bfx)=0, 
\end{equation}
whenever $x_1=x_2=0$. Similarly, we can prove that \begin{equation}
g_{23, 23}(\bfx)=0, 
\end{equation}
whenever $x_1=x_2=0$. To prove the rest of \eqref{230822e3_49}, we take $\partial_1 \partial_2 \partial_3$ on both sides of \eqref{230713e3_4zzz}, and we obtain 
\begin{equation}\label{230822e3_55}
g_{23, 13}(\bfx)+g_{13, 23}(\bfx)=0,
\end{equation}
whenever $x_1=x_2=0$. Recall the calculations in \eqref{230822e3_48cc}. By basic symmetries of Riemannian curvature tensors, we have $R_{i33j}=R_{j33i}$, which further implies that 
\begin{equation}
g_{3j, 3i}=g_{3i, 3j}.
\end{equation}
This, combined with \eqref{230822e3_55}, will finish all the cases of \eqref{230822e3_49}. 
\end{proof}

We return to computing the right hand side of \eqref{230821e3_45}. There is one term $R_{2112}(s)$ that is not covered by Claim \ref{230822claim3_4}. It is natural to try to follow the same strategy of Claim \ref{230822claim3_4} and compute it. Whenever $x_1=x_2=0$, we have 
\begin{equation}\label{230822e3_57}
\begin{split}
-R_{2112}& =\partial_1 \Gamma^2_{21}-\partial_2 \Gamma^2_{11}\\
&=
\frac12 \partial_1
\pnorm{
\sum_l g^{2l} (
g_{2l, 1}+g_{1l, 2}-g_{21, l}
)
}
-
\frac12 \partial_2 \pnorm{
g^{2l} (
g_{1l, 1}+g_{1l, 1}-g_{11, l}
)
}\\
&= 
\frac12 \sum_l g^{2l} (
g_{2l, 11}+g_{1l, 21}-g_{21, l1}
)
-\frac12 g^{2l} (
g_{1l, 11}+g_{1l, 11}-g_{11, l1}
)\\
&=
\frac12(g_{22, 11}
-g_{12, 12}-g_{12, 12}+g_{11, 22}
)
\end{split}
\end{equation}
By taking derivatives on both sides of \eqref{230713e3_4zzz}, one can further simplify \eqref{230822e3_57} to 
\begin{equation}
R_{2112}=-\frac32 g_{11, 22}.
\end{equation}
This finishes computing \eqref{230821e3_45}, and we have 
\begin{equation}
\begin{split}
& \begin{bmatrix}
\ricci_{11}(s), & \ricci_{12}(s)\\
\ricci_{21}(s), & \ricci_{22}(s)
\end{bmatrix}\\
& =
-\begin{bmatrix}
\frac32 g_{11, 22}(s)+\frac12 g_{33, 11}(s), & \frac12 g_{33, 21}(s)\\
\frac12 g_{33, 21}(s), & \frac32 g_{11, 22}(s)+ \frac12 g_{33, 22}(s)
\end{bmatrix}
\end{split}
\end{equation}
where 
\begin{equation}\label{230823e3_44}
g_{ij, kl}(s):=g_{ij, kl}(\gamma(s)). 
\end{equation}
Let us go back to \eqref{230822e3_45hh} and try to express it using $g_{ij, kl}$ and their derivatives. Direct calculations show that \eqref{230822e3_45hh} is equivalent to saying that 
\begin{equation}\label{230822e3_61}
\begin{split}
& |(c_2^2-c_1^2)g_{33, 21}(0)
+c_1 c_2 (g_{33, 11}(0)-g_{33, 22}(0))
|\\
&
+
|(c_2^2-c_1^2)g_{33, 213}(0)
+c_1 c_2 (g_{33, 113}(0)-g_{33, 223}(0))
|\neq 0
\end{split}
\end{equation}
for all $c_1^2+c_2^2=1$. Let us point out here that the term $R_{2112}=-\frac32 g_{11, 22}$ does not appear in \eqref{230822e3_61} as it gets cancelled out in the middle of the computation. This is what Sogge meant in his paper \cite{Sog99} by saying that the chaotic curvature conditions involves ``off-diagonal" parts of the Ricci tensors. \\

So far we have finished computing the chaotic curvature. Now we start computing the contact order of $\phi_{\epsilon}$. 
\begin{claim}\label{230820claim3_1}
We have 
\begin{equation}\label{231023e3_53}
\frac{\partial}{\partial y_i} \frac{\partial}{\partial y_j} \phi_{0, \epsilon}\Big|_{\substack{
(x, t)=\gamma(s),\\
 y=0
 }}= \pnorm{\hessian \Phi_0}\Big|_{
 \substack{
 (x, t)=\gamma(s), \\
 (y, \tau)=(0, \epsilon)}
 }\pnorm{
\frac{\partial}{\partial y_i}, \frac{\partial}{\partial y_j}
}
\end{equation}
for every $s\in [0, \epsilon)$, $1\le i, j\le 2$. Here $\hessian$ is the covariant Hessian in the $(y, \tau)$ variables (see \eqref{230820e3_4}). Moreover, 
\begin{equation}\label{231023e3_54}
\pnorm{\hessian \Phi_0}\Big|_{
 \substack{
 (x, t)=\gamma(s), \\
 (y, \tau)=(0, \epsilon)}
 }\pnorm{
Y, \frac{\partial}{\partial \tau}
}=0, \ \ \forall Y\in T_{(0, \epsilon)}\mc{M},
\end{equation}
for every $s\in [0, \epsilon)$, $1\le i\le 2$.
\end{claim}
\begin{proof}[Proof of Claim \ref{230820claim3_1}]
By definition, 
\begin{equation}
\begin{split}
\pnorm{\hessian \Phi_0}\pnorm{
\frac{\partial}{\partial y_i}, \frac{\partial}{\partial y_j}
}
& = 
\pnorm{
\nabla_{
\frac{\partial}{\partial y_i}
} (\nabla \Phi_0)
}(\frac{\partial}{\partial y_j})\\
& =
\nabla_{\frac{\partial}{\partial y_i}} \pnorm{
\nabla_{\frac{\partial}{\partial y_j}} \Phi_0
}
-\nabla \Phi_0\pnorm{
\nabla_{\frac{\partial}{\partial y_i}} 
\frac{\partial}{\partial y_j}
}.
\end{split}
\end{equation}
By \eqref{231023e3_22}, we obtain 
\begin{equation}
\nabla \Phi_0\Big|_{
\substack{
(x, t)=\gamma(s)\\
(y, \tau)=(0, \epsilon)
}
}=0
\end{equation}
for every $s$. This finishes the proof of \eqref{231023e3_53}. The other identity \eqref{231023e3_54} can be proven similarly. 
\end{proof}

\begin{remark}
The proof of Claim \ref{230820claim3_1} does not reply on Fermi coordinate, and Claim \ref{230820claim3_1} holds true for distance functions in general coordinates. This will be used later in Lemma \ref{230323lemma5_1}, in particular, in \eqref{230903e5_33}.  
\end{remark}

Let us write 
\begin{equation}\label{230820e3_18}
\begin{split}
& \pnorm{\hessian \Phi_0}\Big|_{
 \substack{
 (x, t)=\gamma(s), \\
 (y, \tau)=(0, \epsilon)}
 }
\pnorm{
\frac{\partial}{\partial y_i}, \frac{\partial}{\partial y_j}
} \\
& = \hessian \Phi_0\big|_{\gamma(s)}
\pnorm{
E_i(\epsilon), E_j(\epsilon)
}=: W_{ij}(s, \epsilon).
\end{split}
\end{equation}
In the middle term, we have left out the valuation $(y, \tau)=(0, \epsilon)$, as this information can be read from the fact that $E_i(\epsilon)\in T_{\gamma(\epsilon)}\mc{M}$ and $\gamma(\epsilon)=(0, \epsilon)$. Denote 
\begin{equation}
W(s, \epsilon):=
\begin{bmatrix}
W_{11}(s, \epsilon), & W_{12}(s, \epsilon)\\
W_{21}(s, \epsilon), & W_{22}(s, \epsilon)
\end{bmatrix}
\end{equation}
and 
\begin{equation}
\mathfrak{W}(s, \epsilon):=\det W(s, \epsilon).
\end{equation}
 Our goal is now to compute $W_{ij}$ for $1\le i, j\le 2$, given on the right hand side of \eqref{230820e3_18}, and show that the matrix 
\begin{equation}\label{230821e3_19}
\begin{bmatrix}
\partial_{s}\mathfrak{W}(0, \epsilon), & \partial_{s}^2\mathfrak{W}(0, \epsilon), &\partial_{s}^3 \mathfrak{W}(0, \epsilon), & \partial_{s}^4\mathfrak{W}(0, \epsilon)\\
\partial_{s}W_{11}(0, \epsilon), & \partial_{s}^2W_{11}(0, \epsilon), & \partial_{s}^3 W_{11}(0, \epsilon), & \partial_{s}^4 W_{11}(0, \epsilon)\\
\partial_{s}W_{12}(0, \epsilon), & \partial_{s}^2W_{12}(0, \epsilon), & \partial_{s}^3 W_{12}(0, \epsilon), & \partial_{s}^4 W_{12}(0, \epsilon)\\
\partial_{s}W_{22}(0, \epsilon), & \partial_{s}^2W_{22}(0, \epsilon), & \partial_{s}^3 W_{22}(0, \epsilon), & \partial_{s}^4 W_{22}(0, \epsilon)
\end{bmatrix}
\end{equation}
has rank $4$, if $\epsilon>0$ is chosen sufficiently small. \\

Define Jacobi fields $X_j(s, s'), j=1, 2,$ by 
\begin{equation}\label{230820e3_19}
\begin{split}
& \nabla^2_{\dot{\gamma}(s')} X_j(s, s')+R(X_j(s, s'), \dot{\gamma}(s'))\dot{\gamma}(s')=0, \\
& X_j(s, s)=0, \ \ X_j(s, \epsilon)=E_j(\epsilon).
\end{split}
\end{equation}
Here $R$ stands for the Riemannian curvature tensor, and 
\begin{equation}
\nabla_{\dot{\gamma}(s')}^2 Y:= \nabla_{\dot{\gamma}(s')}(\nabla_{\dot{\gamma}(s')} Y)
\end{equation}
for a vector field $Y$. 
 
\begin{claim}\label{230820claim3_3}
For $j=1, 2$ and fixed $s$, it holds that 
\begin{equation}
X_j(s, s')\perp \dot{\gamma}(s')
\end{equation}
for all $s'\in [s, \epsilon]$. 
\end{claim}
\begin{proof}[Proof of Claim \ref{230820claim3_3}]
We need to prove that 
\begin{equation}
g\pnorm{
X_j(s, s'), \dot{\gamma}(s')
}=0, \ \ \forall s'\in [s, \epsilon].
\end{equation}
Note that by the initial conditions in \eqref{230820e3_19}, we have 
\begin{equation}
g\pnorm{
X_j(s, s), \dot{\gamma}(s)
}=0, \ \ 
g\pnorm{
X_j(s, \epsilon), \dot{\gamma}(\epsilon)
}=0.
\end{equation}
By the chain rule, we have 
\begin{equation}\label{230831e3_55}
\begin{split}
& \frac{\partial}{\partial s'} g\pnorm{
X_j(s, s'), \dot{\gamma}(s')
}\\
& = 
g\pnorm{\nabla_{\dot{\gamma}(s')}
X_j(s, s'), \dot{\gamma}(s')
}+ g\pnorm{
X_j(s, s'), \nabla_{\dot{\gamma}(s')}\dot{\gamma}(s')
}.
\end{split}
\end{equation}
As $\gamma$ is a geodesic, we have 
\begin{equation}
\nabla_{\dot{\gamma}(s')}\dot{\gamma}(s')=0.
\end{equation}
In order to use the information from Jacobi fields in \eqref{230820e3_19}, we take another derivative on both sides of \eqref{230831e3_55}, and obtain 
\begin{equation}
\begin{split}
\frac{\partial^2}{\partial (s')^2} g\pnorm{
X_j(s, s'), \dot{\gamma}(s')
}&
= 
g\pnorm{\nabla^2_{\dot{\gamma}(s')}
X_j(s, s'), \dot{\gamma}(s')
}\\
& =
-g\pnorm{
R(X_j(s, s'), \dot{\gamma}(s'))\dot{\gamma}(s'), \dot{\gamma}(s')
}=0.
\end{split}
\end{equation}
In the last step, we used basic symmetries of the Riemannian curvature tensors. The claim follows from the initial conditions in the Jacobi equation. 
\end{proof}

To proceed, we need Remark 4.11 in Sakai \cite[page 110]{Sak96}. 
\begin{lemma}[Sakai \cite{Sak96}]\label{230902lemma3_5gg}
Take two distinct points $(x, t), (y, \tau)\in \mc{M}$ and let $d_0$ be their distance. Let $\gamma_0:[0, d_0]\to \mc{M}$ be the arc-length parametrized geodesic connecting these two points.  Take two vectors $V_1, V_2\in T_{(y, \tau)}\mc{M}$. Let $Z_j(s)$ be the Jacobi field along $\gamma_0$ satisfying the boundary conditions 
\begin{equation}
Z_j(0)=0, \ \ Z_j(d_0)= V_j, \ \ j=1, 2. 
\end{equation}
Moreover, let 
\begin{equation}
Z_j^{\perp}(s):= 
Z_j(s)-
g(
Z_j(s), \dot{\gamma_0}(s)
)
\dot{\gamma_0}(s), \ \ j=1, 2.
\end{equation}
Then 
\begin{equation}
\pnorm{
\hessian \Phi((x, t), (y, \tau))
}
\pnorm{
V_1, V_2
}
=
g\pnorm{
\nabla_{\dot{\gamma_0}(d_0)} Z^{\perp}_1(d_0), 
Z_2^{\perp}(d_0)
}.
\end{equation}
Here $\hessian$ is the covariant Hessian in the $(y, \tau)$ variables. 
\end{lemma}

By Claim \ref{230820claim3_3} and Lemma \ref{230902lemma3_5gg}, we obtain 
\begin{equation}\label{230820e3_22}
\hessian \Phi\big|_{\gamma(s)}
\pnorm{
E_i(\epsilon), E_j(\epsilon)
}
=
g\pnorm{
\nabla_{\dot{\gamma}(s')}X_i(s, s'), X_j(s, s')
}\Big|_{s'=\epsilon}.
\end{equation}
Note that on the left hand side of \eqref{230820e3_22}, we have $\Phi$ but not $\Phi_0$. We introduce a matrix 
\begin{equation}
A(s, s')=
\begin{bmatrix}
a_{11}(s, s'), & a_{12}(s, s')\\
a_{21}(s, s'), & a_{22}(s, s')
\end{bmatrix}
\end{equation}
by writing 
\begin{equation}
X_j(s, s')= a_{j1}(s, s') E_1(s')+ a_{j2}(s, s')E_2(s'), \ \ j=1, 2.
\end{equation}
Let us rewrite \eqref{230820e3_19} using the new notation. The first equation in \eqref{230820e3_19} becomes 
\begin{multline}
 \partial_{s'}^2
a_{j1}(s, s') E_1(s')+
\partial_{s'}^2 
a_{j2}(s, s') E_2(s') \\
+
R(
a_{j1}(s, s') E_1(s')+ a_{j2}(s, s')E_2(s'), \dot{\gamma}(s')
)\dot{\gamma}(s')=0,
\end{multline}
which is equivalent to the following two equations
\begin{multline}
\partial_{s'}^2
a_{j1}(s, s')
+ 
a_{j1}(s, s') 
g\pnorm{
R(E_1(s'), \dot{\gamma}(s'))\dot{\gamma}(s'),
E_1(s')
}\\
 +
a_{j2}(s, s') 
g\pnorm{
R(E_2(s'), \dot{\gamma}(s'))\dot{\gamma}(s'),
E_1(s')
}=0,
\end{multline}
and 
\begin{multline}
\partial_{s'}^2
a_{j2}(s, s')
+ 
a_{j1}(s, s') 
g\pnorm{
R(E_1(s'), \dot{\gamma}(s'))\dot{\gamma}(s'),
E_2(s')
}\\
+
a_{j2}(s, s') 
g\pnorm{
R(E_2(s'), \dot{\gamma}(s'))\dot{\gamma}(s'),
E_2(s')
}=0.
\end{multline}
Denote 
\begin{equation}
R(s'):=
\begin{bmatrix}
R_{1331}(s'), & R_{1332}(s')\\
R_{2331}(s'), & R_{2332}(s')
\end{bmatrix}
\end{equation}
where
\begin{equation}
R_{ijkl}(s'):=g(
R(E_i(s'), E_j(s'))E_k(s'), E_l(s')
).
\end{equation}
Then \eqref{230820e3_19} can be written as 
\begin{equation}\label{230822e3_34hh}
\begin{split}
& \partial_{s'}^2 A(s, s')+ A(s, s') R(s')=0, \\
& A(s, s)=0, \ A(s, \epsilon)=I_{2\times 2}.
\end{split}
\end{equation}
Here $0$ stands for the zero matrix of order $2\times 2$, $I_{2\times 2}$ is the $2\times 2$ identity matrix. Moreover, 
\begin{equation}\label{230821e3_29}
\begin{split}
\eqref{230820e3_22}&=
g\pnorm{
\nabla_{\dot{\gamma}(s')}(
a_{i1}(s, s')E_1(s')+ a_{i2}(s, s')E_2(s'), 
a_{j1}(s, s')E_1(s')+ a_{j2}(s, s')E_2(s')
)
}\Big|_{s'=\epsilon}\\
&=
\partial_{s'} a_{i1}(s, s')a_{j1}(s, s')
+ 
\partial_{s'} a_{i2}(s, s')a_{j2}(s, s')\big|_{s'=\epsilon}= 
\partial_{s'} a_{ij}\big|_{s'=\epsilon},
\end{split}
\end{equation}
where in the last step we used the initial condition at $s'=\epsilon$.\\

Recall that our goal was to compute \eqref{230820e3_18} and prove \eqref{230821e3_19}. By \eqref{230821e3_29}, we obtain 
\begin{equation}\label{230821e3_30}
W_{ij}(s, \epsilon)=\hessian \Phi_0\big|_{\gamma(s)}
\pnorm{
E_i(\epsilon), E_j(\epsilon)
}=
\partial_{s'}a_{ij}(s, \epsilon)-\partial_{s'}a_{ij}(0, \epsilon).
\end{equation}
When computing \eqref{230820e3_18} and proving  \eqref{230821e3_19}, we will consider the Taylor expansion of \eqref{230821e3_30} at $s=0$:
\begin{equation}\label{230822e3_39hh}
W(s, \epsilon)=\frac{s}{1!} \partial_s \partial_{s'} A(0, \epsilon)+\frac{s^2}{2!} \partial^2_s \partial_{s'} A(0, \epsilon)+\dots
\end{equation}
In the following computation, it will be convenient to introduce some notation. We will write 
\begin{equation}
R'(s'):=\frac{\partial}{\partial s'} R(s'),
\end{equation}
where partial derivatives are taken component-wise.  Moreover, we will write 
\begin{equation}
R_0:=R(0), \ \ R'_0:=R'(0).
\end{equation}
\begin{claim}\label{230822claim3_5}
We have 
\begin{equation}
\begin{split}
& \partial_s\partial_{s'}A(0, \epsilon)= 
\frac{1}{\epsilon^2}\pnorm{
I_{2\times 2}+\frac{R_0}{3}\epsilon^2+ \frac{R'_0}{6} \epsilon^3+ O(\epsilon^4)
}\\
& \partial^2_s\partial_{s'}A(0, \epsilon)= 
\frac{2!}{\epsilon^3}\pnorm{
I_{2\times 2}+ \frac{R'_0}{12}\epsilon^3+O(\epsilon^4)
}\\
& \partial^3_s\partial_{s'}A(0, \epsilon)= 
\frac{3!}{\epsilon^4}\pnorm{
I_{2\times 2}+ O(\epsilon^4)
}\\
& \partial^4_s\partial_{s'}A(0, \epsilon)= 
\frac{4!}{\epsilon^5}\pnorm{
I_{2\times 2}+ O(\epsilon^4)
}
\end{split}
\end{equation}
The implicit constants in $O(\epsilon^4)$ depend only on the manifold. 
\end{claim}

The proof of Claim \ref{230822claim3_5} involves heavy calculations, and will therefore be postponed to the end of this subsection. 

We stop the Taylor expansion in \eqref{230822e3_39hh} at the fifth order, and write 
\begin{equation}\label{230823e3_65}
\begin{split}
W(s, \epsilon)& = 
\frac{s}{\epsilon^2}\pnorm{
I_{2\times 2}+\frac{R_0}{3}\epsilon^2+ \frac{R'_0}{6} \epsilon^3+ O(\epsilon^4)
}+ \frac{s^2}{\epsilon^3}\pnorm{
I_{2\times 2} + \frac{R'_0}{12}\epsilon^3+O(\epsilon^4)
}\\
& + 
\frac{s^3}{\epsilon^4}\pnorm{
I_{2\times 2}+ O(\epsilon^4)
}
+ 
\frac{s^4}{\epsilon^5}\pnorm{
I_{2\times 2}+ O(\epsilon^4)
}
+O_{\epsilon}(s^5).
\end{split}
\end{equation}
Here $O_{\epsilon}(s^5)$ means that the implicit constant there is allowed to depend on $\epsilon$. This dependence is harmless because when computing contact orders (in the $s$ variable), we will always fix $\epsilon$ and consider $s\to 0$. Recall that we need to prove \eqref{230821e3_19}. To simplify notation, we will make the change of variables $s\mapsto \epsilon s$, and consider the matrix 
\begin{equation}
\begin{split}
\epsilon W(\epsilon s, \epsilon)=& 
s\pnorm{
I_{2\times 2}+\frac{R_0}{3}\epsilon^2+ \frac{R'_0}{6} \epsilon^3+ O(\epsilon^4)
}+ s^2\pnorm{
I_{2\times 2}+ \frac{R'_0}{12}\epsilon^3+O(\epsilon^4)
}\\
& + 
s^3\pnorm{
I_{2\times 2}+ O(\epsilon^4)
}
+ 
s^4\pnorm{
I_{2\times 2}+ O(\epsilon^4)
}
+O_{\epsilon}(s^5).
\end{split}
\end{equation} 
%
By Claim \ref{230822claim3_4}, 
\begin{equation}
R(s)=
-\begin{bmatrix}
\frac{g_{33, 11}(s)}{2}, & \frac{g_{33, 12}(s)}{2}\\
\frac{g_{33, 21}(s)}{2}, & \frac{g_{33, 22}(s)}{2}
\end{bmatrix}
\end{equation}
and therefore 
\begin{equation}
R'(s)=
-\begin{bmatrix}
\frac{g_{33, 113}(s)}{2}, & \frac{g_{33, 123}(s)}{2}\\
\frac{g_{33, 213}(s)}{2}, & \frac{g_{33, 223}(s)}{2}
\end{bmatrix}
\end{equation}
where similarly to \eqref{230823e3_44}, we use the notation 
\begin{equation}
g_{ij, klm}(s):= g_{ij, klm}(\gamma(s)).
\end{equation}
Moreover, we denote 
\begin{equation}
\mfg_{ij, kl}:= g_{ij, kl}(0), \ \ \mfg_{ij, klm}:=g_{ij, klm}(0),
\end{equation}
which will significantly simplify the notation. Write 
\begin{equation}
\epsilon W(\epsilon s, \epsilon)=
\begin{bmatrix}
(\star)_{11}, & (\star)_{12}\\
(\star)_{21}, & (\star)_{22}
\end{bmatrix}
+O_{\epsilon}(s^5),
\end{equation}
where 
\begin{multline}
(\star)_{11}:= s(
1-
\frac16 \mfg_{33, 11}\epsilon^2- \frac{1}{12} \mfg_{33, 113}\epsilon^3+ O(\epsilon^4)
)\\
+ 
s^2(
1-
\frac{1}{24} \mfg_{33, 113}\epsilon^3+ O(\epsilon^4)
)
+ s^3(1+ O(\epsilon^4))
+ s^4(1+ O(\epsilon^4)),
\end{multline}
\begin{multline}
(\star)_{22}:= s(
1-
\frac16 \mfg_{33, 22}\epsilon^2- \frac{1}{12} \mfg_{33, 223}\epsilon^3+ O(\epsilon^4)
)\\
+
s^2(
1-
 \frac{1}{24} \mfg_{33, 223}\epsilon^3+ O(\epsilon^4)
)
+ s^3(1+ O(\epsilon^4))
+ s^4(1+ O(\epsilon^4)),
\end{multline}
and 
\begin{multline}
(\star)_{12}=(\star)_{21}:= 
s(
-\frac16 \mfg_{33, 12}\epsilon^2- \frac{1}{12} \mfg_{33, 123}\epsilon^3+ O(\epsilon^4)
)\\
+ 
s^2(
-\frac{1}{24} \mfg_{33, 123}\epsilon^3+ O(\epsilon^4)
)
+ s^3  O(\epsilon^4)
+ s^4 O(\epsilon^4)
\end{multline}
We compute $\det(\epsilon W(\epsilon s, \epsilon))$ and obtain 
\begin{equation}
\begin{split}
(\star)_{11}(\star)_{22}-(\star)_{12}(\star)_{21}= 
 (\star)_1 s^2+ 
(\star)_2 s^3 + 
(\star)_3 s^4 + O_{\epsilon}(s^5),
\end{split}
\end{equation}
where 
\begin{equation}
\begin{split}
& (\star)_1:=1-\frac16 \epsilon^2(\mfg_{33, 11}+\mfg_{33, 22})
-\frac{1}{12}\epsilon^3 (\mfg_{33, 113}+\mfg_{33, 223})+O(\epsilon^4),\\
& (\star)_2:=2-\frac16 \epsilon^2(\mfg_{33, 11}+\mfg_{33, 22})
-\frac{1}{8}\epsilon^2 (\mfg_{33, 113}+\mfg_{33, 223})+O(\epsilon^4),\\
& (\star)_3:=3-\frac16 \epsilon^2(\mfg_{33, 11}+\mfg_{33, 22})
-\frac{1}{8}\epsilon^2 (\mfg_{33, 113}+\mfg_{33, 223})+O(\epsilon^4).
\end{split}
\end{equation}
To prove \eqref{230821e3_19}, it suffices to show that the matrix
\begin{equation}
\begin{bmatrix}
(\star\star)_{11}, & 1-
\frac{1}{24} \mfg_{33, 113}\epsilon^3+ O(\epsilon^4), & 1+O(\epsilon^4), & 1+O(\epsilon^4)\\
(\star\star)_{21}, & 1-
\frac{1}{24} \mfg_{33, 223}\epsilon^3+ O(\epsilon^4), & 1+O(\epsilon^4), & 1+O(\epsilon^4)\\
(\star\star)_{31}, & -\frac{1}{24} \mfg_{33, 123}\epsilon^3+ O(\epsilon^4), & O(\epsilon^4), & O(\epsilon^4)\\
0, & (\star)_1, & (\star)_2, & (\star)_3
\end{bmatrix}
\end{equation}
with 
\begin{equation}
\begin{split}
& (\star\star)_{11}:=1-
\frac16 \mfg_{33, 11}\epsilon^2- \frac{1}{12} \mfg_{33, 113}\epsilon^3+ O(\epsilon^4),\\
& (\star\star)_{21}:=1-
\frac16 \mfg_{33, 22}\epsilon^2- \frac{1}{12} \mfg_{33, 223}\epsilon^3+ O(\epsilon^4),\\
& (\star\star)_{31}:=-\frac16 \mfg_{33, 12}\epsilon^2- \frac{1}{12} \mfg_{33, 123}\epsilon^3+ O(\epsilon^4),
\end{split}
\end{equation}
is non-degenerate whenever $\epsilon>0$ is picked to be small enough. We subtract the second row from the first row, and obtain 
\begin{equation}
\begin{bmatrix}
(\star\star\star)_{11}, &
-\frac{1}{24} (\mfg_{33, 113}-\mfg_{33, 223})\epsilon^3+ O(\epsilon^4), & O(\epsilon^4), & O(\epsilon^4)\\
(\star\star\star)_{21}, & 1-
\frac{1}{24} \mfg_{33, 223}\epsilon^3+ O(\epsilon^4), & 1+O(\epsilon^4), & 1+O(\epsilon^4)\\
(\star\star\star)_{31}, & -\frac{1}{24} \mfg_{33, 123}\epsilon^3+ O(\epsilon^4), & O(\epsilon^4), & O(\epsilon^4)\\
0, & (\star)_1, & (\star)_2, & (\star)_3
\end{bmatrix}
\end{equation}
where 
\begin{equation}
\begin{split}
& (\star\star\star)_{11}:= -\frac16 (\mfg_{33, 11}-\mfg_{33, 22})\epsilon^2- \frac{1}{12} (\mfg_{33, 113}-\mfg_{33, 223})\epsilon^3+ O(\epsilon^4),\\
& (\star\star\star)_{21}:=1-
\frac16 \mfg_{33, 22}\epsilon^2- \frac{1}{12} \mfg_{33, 223}\epsilon^3+ O(\epsilon^4), \\
& (\star\star\star)_{31}:=-\frac16 \mfg_{33, 12}\epsilon^2- \frac{1}{12} \mfg_{33, 123}\epsilon^3+ O(\epsilon^4).
\end{split}
\end{equation}
It is not difficult to see that the determinant is $O(\epsilon^5)$. Let us compute the coefficient of $\epsilon^5$, and we obtain 
\begin{equation}\label{230823e3_77}
-\frac{1}{6* 24}\pnorm{
(\mfg_{33, 11}-\mfg_{33, 22})\mfg_{33, 123}-
(\mfg_{33, 113}-\mfg_{33, 223})\mfg_{33, 12}
}
\end{equation}
Recall from \eqref{230822e3_61} that Sogge's chaotic curvature implies that 
\begin{equation}\label{230822e3_61mmm}
\begin{split}
& |(c_2^2-c_1^2)\mfg_{33, 21}
+c_1 c_2 (\mfg_{33, 11}-\mfg_{33, 22})
|\\
&
+
|(c_2^2-c_1^2)\mfg_{33, 213}
+c_1 c_2 (\mfg_{33, 113}-\mfg_{33, 223})
|\neq 0
\end{split}
\end{equation}
for all $c_1^2+c_2^2=1$. Let us write \eqref{230822e3_61mmm} slightly differently, but equivalently,  as 
\begin{equation}\label{230822e3_61pp}
\begin{split}
& |c_1 \mfg_{33, 21}
+c_2 (\mfg_{33, 11}-\mfg_{33, 22})
|\\
&
+
|c_1 \mfg_{33, 213}
+c_2 (\mfg_{33, 113}-\mfg_{33, 223})
|\neq 0
\end{split}
\end{equation}
for all $c_1^2+c_2^2=1$. To see that $\eqref{230823e3_77}\neq 0$, we just pick 
\begin{equation}
c_1= c(\mfg_{33, 113}-\mfg_{33, 223}), \ c_2=-c(\mfg_{33, 123}),
\end{equation}
in \eqref{230822e3_61pp}, for some appropriately chosen constant $c$; the only thing we need to make sure is that $\mfg_{33, 113}-\mfg_{33, 223}$ and $\mfg_{33, 123}$ do not vanish simultaneously. However, it is elementary to see that if they do vanish  simultaneously, then \eqref{230822e3_61pp} can never hold. This finishes the proof that $\eqref{230823e3_77}\neq 0$, thus the item (a) of Theorem \ref{230711thm2_5}, modulo the proof of Claim \ref{230822claim3_5}.\\

\begin{proof}[Proof of Claim \ref{230822claim3_5}]

To compute \eqref{230822e3_39hh}, let us first recall how the matrix $A$ is defined in \eqref{230822e3_34hh}. To study this linear systems of equations, we will introduce two auxiliary linear systems of equations: 
\begin{equation}\label{230823e3_81}
\begin{split}
& B''_1(s')+ B_1(s')R(s')=0_{2\times 2},\\
& B_1(0)=I_{2\times 2}, \ \ B'_1(0)=0_{2\times 2},
\end{split}
\end{equation}
and 
\begin{equation}\label{230823e3_82}
\begin{split}
& B''_2(s')+ B_2(s')R(s')=0_{2\times 2},\\
& B_2(0)=0_{2\times 2}, \ \ B'_2(0)=I_{2\times 2},
\end{split}
\end{equation}
where $B_1(s')$ and $B_2(s')$ are two $2\times 2$ matrices with entries being functions depending on $s'$. As all these systems of equations are linear, we know we can find smooth matrix-valued functions $C_1(s)$ and $C_2(s)$ such that 
\begin{equation}\label{230823e3_83}
A(s, s')= C_1(s)B_1(s')+ C_2(s) B_2(s').
\end{equation}
Before computing partial derivatives of $A$, let us collect some useful data. By taking derivatives on \eqref{230823e3_81} and \eqref{230823e3_82}, we obtain 
\begin{equation}
\begin{split}
& 
B_1(0)=I, \ B'_1(0)=0, \ B''_1(0)=-R_0, \ B'''_1(0)=-R'_0;\\
& B_2(0)=0, \ B'_2(0)=I, \ B''_2(0)=0, \ B'''_2(0)=-R_0, \ B^{(4)}_2(0)=-2R'_0,
\end{split}
\end{equation}
where we abbreviate $I_{2\times 2}$ to $I$. By Taylor's expansion, we have 
\begin{equation}
\begin{split}
& B_1(\epsilon)=I-\frac{R_0}{2!} \epsilon^2-\frac{R'_0}{3!}\epsilon^3+O(\epsilon^4),\\
& B_2(\epsilon)=\epsilon
\pnorm{
I-\frac{R_0}{3!} \epsilon^2-2\frac{R'_0}{4!}\epsilon^3+O(\epsilon^4)
}.
\end{split}
\end{equation}
By taking inverses, we obtain 
\begin{equation}
B_2^{-1}(\epsilon)= 
\frac{1}{\epsilon}
\pnorm{
I+
\frac{R_0}{3!}\epsilon^2+
\frac{2R'_0}{4!}\epsilon^3+O(\epsilon^4)
}.
\end{equation}
Direct computation shows that 
\begin{equation}\label{230831e3_95}
\begin{split}
B_2^{-1}(\epsilon) B_1(\epsilon)
&
=
\frac{1}{\epsilon}
\pnorm{
I+
\frac{R_0}{3!}\epsilon^2+
\frac{2R'_0}{4!}\epsilon^3+O(\epsilon^4)
}
\pnorm{
I-\frac{R_0}{2!} \epsilon^2-\frac{R'_0}{3!}\epsilon^3+O(\epsilon^4)
}\\
&=
\frac{1}{\epsilon}
\pnorm{
I
-
\frac{R_0}{3} \epsilon^2
-
\frac{R'_0}{12} \epsilon^3+ O(\epsilon^4)
}.
\end{split}
\end{equation}
Recall the initial condition in \eqref{230822e3_34hh}, that is, the condition that 
\begin{equation}\label{230823e3_87}
A(s, s)=0, \ A(s, \epsilon)=I_{2\times 2}.
\end{equation}
By taking derivatives in $s$ in \eqref{230823e3_87}, we obtain that 
\begin{equation}\label{230823e3_88}
\begin{split}
& \partial_s A(s, s)+\partial_{s'}A(s, s)=0, \ \partial_s A(s, \epsilon)=0,\\
& \partial^2_s A(s, s)+
2\partial_s\partial_{s'}A(s, s)+
\partial_{s'}^2A(s, s)=0, \ \partial^2_s A(s, \epsilon)=0,\\
&\partial^3_s A(s, s)+
 3\partial^2_s\partial_{s'}A(s, s)+
  3\partial_s\partial^2_{s'}A(s, s)+
\partial_{s'}^3A(s, s)=0, \ \partial^3_s A(s, \epsilon)=0.
\end{split}
\end{equation}
Taking $s=0$ in \eqref{230823e3_88}, we obtain 
\begin{equation}\label{230831e3_97}
\begin{split}
& C_1(0)=0, \ \ C_2(0)=B_2^{-1}(\epsilon),\\
& C'_1(0)=-B_2^{-1}(\epsilon), \ \ C'_2(0)=B_2^{-1}(\epsilon)B_1(\epsilon)B_2^{-1}(\epsilon),  \\
& C''_1(0)=-2 C'_2(0), \ \ C''_2(0)=-C''_1(0) B_1(\epsilon)B_2^{-1}(\epsilon),\\
& C'''_1(0)=-3 C''_2(0)-2 B_2^{-1}(\epsilon)R_0, \ \ C'''_2(0)=-C'''_1(0)B_1(\epsilon)B_2^{-1}(\epsilon).
\end{split}
\end{equation}
We compute $\partial_s\partial_{s'} A(0, \epsilon)$. Using the relation \eqref{230823e3_83}, we obtain 
\begin{equation}
\begin{split}
\partial_s\partial_{s'} A(0, \epsilon)& =
C'_1(0)B'_1(\epsilon)+C'_2(0) B'_2(\epsilon)\\
&=
-B_2^{-1}(\epsilon)B'_1(\epsilon)+
B_2^{-1}(\epsilon)B_1(\epsilon)B_2^{-1}(\epsilon) B'_2(\epsilon).
\end{split}
\end{equation}
This is further equal to 
\begin{equation}
\begin{split}
& 
\pnorm{
\frac{1}{\epsilon} I+\frac{R_0}{3!}\epsilon+ 
\frac{2R'_0}{4!}\epsilon^2+O(\epsilon^3)
}
\pnorm{
R_0 \epsilon+\frac{R'_0}{2!}\epsilon^2+O(\epsilon^3)
}\\
&+\frac{1}{\epsilon^2} 
\pnorm{
I+ \frac{R_0}{3!}\epsilon^2+ \frac{2R'_0}{4!}\epsilon^3+ O(\epsilon^4)
}
\pnorm{
I-\frac{R_0}{2!} \epsilon^2-\frac{R'_0}{3!}\epsilon^3+O(\epsilon^4)
}\\
& \pnorm{
I+ \frac{R_0}{3!}\epsilon^2+ \frac{2R'_0}{4!}\epsilon^3+ O(\epsilon^4)
}
\pnorm{
I-\frac{R_0}{2!}\epsilon^2-\frac{2R'_0}{3!}\epsilon^3+O(\epsilon^4)
}\\
&= \pnorm{
R_0+ \frac{R'_0}{2}\epsilon+ O(\epsilon^2)
}
+ 
\frac{1}{\epsilon^2}
\pnorm{
I-\frac{2}{3}R_0 \epsilon^2-\frac{2}{3!}R'_0\epsilon^3+ O(\epsilon^4)
}\\
&=
\frac{1}{\epsilon^2}\pnorm{
I+\frac{R_0}{3}\epsilon^2+ \frac{R'_0}{6} \epsilon^3+ O(\epsilon^4)}
\end{split}
\end{equation}
This finishes the calculation for the first identity in the claim. \\

We compute $\partial^2_s\partial_{s'} A(0, \epsilon)$. First, observe that 
\begin{equation}\label{230904e3_104}
\partial^2_s\partial_{s'} A(0, \epsilon)=
2 B_2^{-1}(\epsilon) B_1(\epsilon) \partial_s\partial_{s'} A(0, \epsilon).
\end{equation}
We use the first identity in the claim, and see that the last display is equal to 
\begin{multline}
\frac{2}{\epsilon^3} 
\pnorm{
I+ \frac{R_0}{3!}\epsilon^2+ \frac{2R'_0}{4!}\epsilon^3+ O(\epsilon^4)
}\\
\pnorm{
I-\frac{R_0}{2!} \epsilon^2-\frac{R'_0}{3!}\epsilon^3+O(\epsilon^4)
}
\pnorm{
I+\frac{R_0}{3}\epsilon^2+ \frac{R'_0}{6} \epsilon^3+ O(\epsilon^4)},
\end{multline}
which is further equal to 
\begin{equation}
\begin{split}
\frac{2}{\epsilon^3}\pnorm{
I+
\frac{2R'_0}{4!}\epsilon^3+O(\epsilon^4)
}= \frac{2}{\epsilon^3} I+
\frac{R'_0}{3!}+O(\epsilon).
\end{split}
\end{equation}
This finishes the second identity in the claim.\\

Next, we compute $\partial^3_s\partial_{s'} A(0, \epsilon)$. By \eqref{230823e3_83}, we obtain 
\begin{equation}
\partial_s^3 \partial_{s'} A(0, \varepsilon)=C_1^{\prime \prime \prime}(0) \cdot B_1^{\prime}(\varepsilon)+C_2^{\prime \prime \prime}(0) \cdot B_2^{\prime}(\varepsilon).
\end{equation}
By \eqref{230831e3_97}, this is equal to 
\begin{equation}
\begin{aligned}
&C_1^{\prime \prime \prime}(0) B_1^{\prime}(\varepsilon)-C_1^{\prime \prime \prime}(0) B_1(\varepsilon) B_2^{-1}(\varepsilon) B_2^{\prime}(\varepsilon) \\
& =C_1^{\prime \prime \prime}(0)\left(B_1^{\prime}(\varepsilon)-B_1(\varepsilon) B_2^{-1}(\varepsilon) B_2^{\prime}(\varepsilon)\right)\\
& 
=-\left(3 C_2^{\prime \prime}(0)+2 B_2^{-1}(\varepsilon) R(0)\right)\left(B_1^{\prime}(\varepsilon)-B_1(\varepsilon) B_2^{-1}(\varepsilon) B_2^{\prime}(\varepsilon)\right)\\
&= 
-\left(6 B_2^{-1}(\varepsilon) B_1(\varepsilon) B_2^{-1}(\varepsilon) B_1(\varepsilon) B_2^{-1}(\varepsilon)+2 B_2^{-1}(\varepsilon) R_0\right)\\
& \times \left(B_1^{\prime}(\varepsilon)-B_1(\varepsilon) B_2^{-1}(\varepsilon) B_2^{\prime}(\varepsilon)\right).
\end{aligned}
\end{equation}
We compute the two factors in the last term separately. The first fact is equal to 
\begin{equation}
\begin{split}
& 
\frac{6}{\epsilon^3}
\pnorm{
I
-
\frac{R_0}{3} \epsilon^2
-
\frac{R'_0}{12} \epsilon^3+ O(\epsilon^4)
}
\pnorm{
I
-
\frac{R_0}{3} \epsilon^2
-
\frac{R'_0}{12} \epsilon^3+ O(\epsilon^4)
}\\
& \times 
\pnorm{
I+
\frac{R_0}{3!}\epsilon^2+
\frac{2R'_0}{4!}\epsilon^3+O(\epsilon^4)
}\\
& +
\frac{2}{\epsilon}\left(I+\frac{R_0}{3 !} \epsilon^2+\frac{2 R^{\prime}_0}{4 !} \epsilon^3+O\left(\epsilon^4\right)\right) \cdot R_0\\
&=
\frac{6}{\epsilon^3}\left(I-\frac{R_0}{2} \epsilon^2-\frac{R^{\prime}_0}{12} \epsilon^3+O\left(\epsilon^4\right)\right)+\frac{2}{\epsilon} \cdot R_0+O(\epsilon)\\
& =
\frac{6}{\epsilon^3}\left(I-\frac{R_0}{6} \epsilon^2-\frac{R^{\prime}_0}{12} \epsilon^3+O\left(\epsilon^4\right)\right).
\end{split}
\end{equation}
The second factor is equal to 
\begin{equation}
\begin{split}
& -R_0 \cdot \epsilon-\frac{R^{\prime}_0}{2} \cdot \epsilon^2+O\left(\epsilon^3\right)\\
& -\frac{1}{\epsilon}\left(I-\frac{R_0}{2} \epsilon^2-\frac{R_0^{\prime}}{6} \epsilon^3+O\left(\epsilon^4\right)\right)\left(I+\frac{R_0}{6} \epsilon^2+\frac{2 R^{\prime}_0}{4 !} \epsilon^3+O\left(\epsilon^4\right)\right)\\
& \times 
\left(I-\frac{R_0}{2} \epsilon^2-\frac{R_0^{\prime}}{3} \cdot \epsilon^3 + O(\epsilon^4)\right)\\
&=
-R_0 \epsilon-\frac{R_0^{\prime}}{2} \cdot \epsilon^2+0\left(\epsilon^3\right)-\frac{1}{\epsilon}\left(I-\frac{5}{6} R_0 \epsilon^2-\frac{5}{12} \cdot R_0^{\prime} \epsilon^3+O\left(\epsilon^4\right)\right)\\
&=
-\frac{1}{\epsilon}\left(I+\frac{1}{6} R_0 \epsilon^2+\frac{1}{12} R_0^{\prime} \epsilon^3+O\left(\epsilon^4\right)\right).
\end{split}
\end{equation}
Multiplying these two factors, we obtain the desired equation for $\partial^3_s\partial_{s'} A(0, \epsilon)$.\\

The last equation involving $\partial^4_s\partial_{s'}A(0, \epsilon)$ can be proven similarly. We leave out the computations. \end{proof}

\section{Proof of Theorem \ref{230816theorem2_4}}
\label{230405section6}

For $N\ge 1$, $\bfx=(x, t)\in \R^3, y\in \R^2$, denote 
\begin{equation}
\phi^N(\bfx; y):=
N\phi(\bfx/N; y), \ \ a^{N}(\bfx; y):=a(\bfx/N; y).
\end{equation}
Let $\epsilon_{\phi}>0$ be a sufficiently small constant depending on $\phi$. Define an operator 
\begin{equation}
T^N f(\bfx):= \int e^{
i\phi^N(\bfx; y)
} a^N(\bfx; y)dy.
\end{equation}
Note that $T^N f$ is just a rescaled version of $T_N^{(\phi)}f$. The goal of this section is to prove the following theorem. 

\begin{theorem}\label{230826theorem4_1}
Under the same assumptions as in Theorem \ref{230816theorem2_4}, we have 
\begin{equation}\label{230825e4_3}
\norm{T^N f}_{L^p(\R^3)} \lesim_{\phi, a, p, \epsilon} N^{\epsilon} \norm{f}_{L^p(\R^2)},
\end{equation}
for all 
\begin{equation}\label{230827e4_4}
p\ge \frac{10}{3}-\epsilon_k,
\end{equation}
all $\epsilon>0$ and $N\ge 1$. 
\end{theorem}

To prove Theorem \ref{230826theorem4_1}, we follow the idea of Guth \cite{Gut18} and Bourgain and Guth \cite{BG11}, and reduce it to ``broad" estimates. 

Let $K\ge 1$. We divide $\B^2_{\epsilon_{\phi}}$ into dyadic squares $\tau$ of side length $K^{-1}$. Denote $f_{\tau}:=f\cdot \mathbbm{1}_{\tau}$. Fix a ball $\B^3_{K^2}(\bfx_0)$. Define 
\begin{equation}
\mu_{T^N f}(\B^3_{K^2}(\bfx_0)):= 
\min_{\tau_1, \dots, \tau_{A_0}}
\pnorm{
\max_{\tau\neq \tau_{A'_0}, 1\le A'_0\le A_0}
\norm{T^N f_{\tau}}_{
L^p(\B_{K^2}^3(\bfx_0))
}^p
},
\end{equation}
where $A_0$ is a large parameter whose choice will become clear later.  For $U\subset \R^3$, define 
\begin{equation}
\left\|T^N f\right\|_{\mathrm{BL}_{A_0}^p(U)}:=\left(\sum_{\B^3_{K^2}(\bfx_0)} \frac{\left|\B^3_{K^2}(\bfx_0) \cap U\right|}{\left|\B_{K^2}^3(\bfx_0)\right|} \mu_{T^N f}\left(\B^3_{K^2}(\bfx_0)\right)\right)^{1 / p},
\end{equation}
where the sum runs over a finitely overlapping collection of balls $\B^3_{K^2}(\bfx_0)$ that covers $\R^3$. This is called the \underline{broad part} of $T^N f$. 

\begin{theorem}\label{230826theorem4_2}
For every $\epsilon>0$, there exists $A_0$ such that 
\begin{equation}
\left\|T^N f\right\|_{\mathrm{BL}_{A_0}^p(\R^3)} \lesssim_{K, \epsilon} N^\epsilon
\|f\|_{L^2}^{2 / p}\|f\|_{L^{\infty}}^{1-2 / p},
\end{equation}
for every $p$ satisfying \eqref{230827e4_4},
 every $K\ge 1, \epsilon>0$, and $N\ge 1$.  Moreover, the implicit constant depends polynomially on $K$. 
\end{theorem}

Reducing Theorem \ref{230826theorem4_1} to Theorem \ref{230826theorem4_2} can be done via standard arguments in the literature (for instance Guth, Hickman and Iliopoulou \cite[Proposition 11.1]{GHI19}). 

%
%

The rest of this section is devoted to the proof of Theorem \ref{230826theorem4_2}.

\subsection{Preliminaries in polynomial partitionings}\label{230902subsection4_1}

We follow the notation from \cite{GWZ22}. Take 
\begin{equation}
1\le r\le R= N.
\end{equation}
The only reason of introducing the parameter $R$ is just for the forthcoming notation to be consistent with that in  \cite{GWZ22}. Take a collection $\Theta_r$ of dyadic cubes of side length $\frac{9}{11} r^{-1 / 2}$ covering the ball $\B^{2}$, the unit ball in $\R^2$. We take a smooth partition of unity $\left(\psi_\theta\right)_{\theta \in \Theta_r}$ with $\operatorname{supp} \psi_\theta \subset \frac{11}{10} \theta$ for the ball $\B^{2}$ such that
$$
|\partial_{\xi}^\alpha \psi_\theta(\xi)| \lesssim_\alpha r^{\|\alpha\|_{\infty} / 2},
$$
for all $\xi\in \R^2$ and all multi-indices $\alpha=(\alpha_1, \alpha_2) \in \mathbb{N}_0^{2}$ with $\|\alpha\|_{\infty}:=|\alpha_1|+|\alpha_2|$. We denote by $\xi_\theta$ the center of $\theta$. Given a function $h$, we perform a Fourier series decomposition to the function $h \psi_\theta$ on the region $\frac{11}{9} \theta$ and obtain
$$
h(\xi) \psi_\theta(\xi) \cdot \mathbbm{1}_{\frac{11}{10} \theta}(\xi)=\left(\frac{r^{1 / 2}}{2 \pi}\right)^{2} \sum_{v \in r^{1 / 2} \mathbb{Z}^{2}}
\left(h \psi_\theta\right)^{\wedge}(v) e^{2 \pi i v \cdot \xi} \mathbbm{1}_{\frac{11}{10} \theta}(\xi) .
$$
Let $\widetilde{\psi}_\theta$ be a non-negative smooth cutoff function supported on $\frac{11}{9} \theta$ and equal to 1 on $\frac{11}{10} \theta$. We can therefore write
$$
h(\xi) \psi_\theta(\xi) \cdot \widetilde{\psi}_\theta(\xi)=\left(\frac{r^{1 / 2}}{2 \pi}\right)^{2} \sum_{v \in r^{1 / 2} \mathbb{Z}^{2}}\left(h \psi_\theta\right)^{\wedge}(v) e^{2 \pi i v \cdot \xi} \widetilde{\psi}_\theta(\xi)
$$
If we also define
$$
h_{\theta, v}(\xi):=\left(\frac{r^{1 / 2}}{2 \pi}\right)^{2}\left(h \psi_\theta\right)^{\wedge}(v) e^{2 \pi i v \cdot \xi} \widetilde{\psi}_\theta(\xi),
$$
then we have 
\begin{equation}
h=\sum_{(\theta, v)\in \Theta_r\times r^{1/2}\Z^2} h_{\theta, v}.
\end{equation}
For $\omega\in \B^2_{\epsilon_{\phi}}$ and $\xi\in \B^2_{\epsilon_{\phi}}$, we define $\Phi=\Phi(\omega, t; \xi)$ by 
\begin{equation}\label{230828e4_10}
\nabla_{\xi} \phi\pnorm{
\Phi(\omega, t; \xi), t; \xi
}=\omega.
\end{equation}
Let 
\begin{equation}\label{230828e4_11oo}
\delta:=\epsilon^C, \ \  C \text{ large universal constant,}
\end{equation}
where $\epsilon$ is as in \eqref{230825e4_3}. We define curved $r^{1/2+\delta}$-tubes as 
\begin{equation}\label{230828e4_12}
T_{\theta, v}:=
\Big\{(x, t): 
\anorm{
\frac{x}{N}-\Phi\pnorm{
\frac{v}{N}, \frac{t}{N}; \xi_{\theta}
}
}\le 
\frac{r^{1/2+\delta}}{N}, \ t\in [0, r]
\Big\}
\end{equation}
This finishes the wave packet decomposition for the ball $\B^3_r\subset\R^3$. We use $\T[
\B^3_r
]$ to denote the collection of the wave packets $T_{\theta, v}$. 
Similarly, for $\bfx_0\in \B^3_N$, we can define wave packet decompositions for the ball $\B^3_{r}(\bfx_0)$, and use $ \mathbb{T}\left[\B^3_r\left(\mathbf{x}_0\right)\right]$ to denote the collection of the resulting wave packets. We have 
\begin{equation}
T^N h(\mathbf{x})=\sum_{T \in \mathbb{T}\left[\B^3_r\left(\mathbf{x}_0\right)\right]} T^N h_T(\mathbf{x})
\end{equation}
whenever $|\bfx-\bfx_0|\lesim r$. \\

Next, we introduce a few key notions that will appear in the forthcoming polynomial partitioning algorithms.

\begin{definition}[Cells]
Let $P: \R^3\to \R$ be a non-zero polynomial. Denote 
\begin{equation}
Z(P):=\{z\in \R^3: P(z)=0\}.
\end{equation}
We let $\cell(P)$ denote the collection of all the connected components of $\R^3\setminus Z(P)$. Each element in $\cell(P)$ will be refereed to as a cell of $P$. 
\end{definition}

\begin{definition}[Transverse complete intersection] Take the dimension $n=3$ and $0\le m\le 2$. Let $P_1, \ldots, P_{n-m}: \mathbb{R}^n \rightarrow$ $\mathbb{R}$ be polynomials. We consider the common zero set
\begin{equation}\label{230827e4_14}
Z\left(P_1, \ldots, P_{n-m}\right):=\left\{x \in \mathbb{R}^n: P_1(x)=\cdots=P_{n-m}(x)=0\right\} .
\end{equation}
Suppose that for all $z \in Z\left(P_1, \ldots, P_{n-m}\right)$, one has
\begin{equation}
\bigwedge_{j=1}^{n-m} \nabla P_j(z) \neq 0 .
\end{equation}
Then a connected branch of this set, or a union of connected branches of this set, is called an m-dimensional transverse complete intersection. Given a set $Z$ of the form \eqref{230827e4_14}, the degree of $Z$ is defined by
$$
\min \left(\prod_{i=1}^{n-m} \operatorname{deg}\left(P_i\right)\right)
$$
where the minimum is taken over all possible representations of $$Z=Z\left(P_1, \ldots, P_{n-m}\right).$$
\end{definition}

\begin{lemma}[Polynomial partitioning lemma, Guth \cite{Gut18}, Hickman and Rogers \cite{HR19}]\label{230827lemma4_5}
Fix $r \gg 1, d \in \mathbb{N}$ and suppose $F \in L^1\left(\mathbb{R}^3\right)$ is non-negative and supported on $\B^3_r(\bfx_0) \cap \mathcal{N}_{r^{1 / 2+\delta_{\circ}}}(Z)$ for some $\bfx_0$ and $0<\delta_{\circ} \ll 1$, where $Z$ is an $m$ dimensional transverse complete intersection of degree at most $d$. At least one of the following cases holds:\\

\noindent \underline{Cellular case.} There exists a polynomial $P: \mathbb{R}^3 \rightarrow \mathbb{R}$ of degree $O(d)$ with the following properties:
\begin{enumerate}
\item[(1)] We can find a sub-collection of cells $\cell'(P)\subset \cell(P)$ with  $\# \operatorname{cell}'(P) \simeq d^m$ and each $O^{\prime} \in \operatorname{cell}'(P)$ has diameter at most $r / d$.
\item[(2)] If we define
$$
\mathcal{O}:=\left\{O^{\prime} \backslash \mathcal{N}_{r^{1 / 2+\delta_{\circ}}}(Z(P)): O^{\prime} \in \operatorname{cell}'(P)\right\},
$$
then
\begin{equation}
\int_O F \simeq d^{-m} \int_{\mathbb{R}^n} F,
\end{equation}
for all $O\in \mc{O}$. 
\end{enumerate}

\noindent \underline{Algebraic case.} There exists an $(m-1)$-dimensional transverse complete intersection $Y$ of degree at most $O(d)$ such that
$$
\int_{\B^3_r(\bfx_0) \cap \mathcal{N}_{r^{1 / 2+\delta_{\circ}}}(Z)} F \lesssim \int_{\B^3_r(\bfx_0) \cap \mathcal{N}_{r^{1/2+\delta_{\circ}}}(Y)} F
$$
\end{lemma}

%
%

\subsection{Polynomial partitioning algorithms}

Our goal is to prove Theorem \ref{230826theorem4_2}, that is, we will prove 
\begin{equation}\label{230828e4_17}
\left\|T^N f\right\|_{\mathrm{BL}_{A_0}^p(\R^3)} \lesssim_{K, \epsilon} N^\epsilon
\|f\|_{L^2}^{2 / p}\|f\|_{L^{\infty}}^{1-2 / p},
\end{equation}
for all
\begin{equation}\label{230828e4_18qq}
p\ge \frac{10}{3}-\epsilon_k,
\end{equation}
all $K\ge 1$, $\epsilon>0$ and $N\ge 1$. \\

We will recycle the polynomial partitioning algorithm in \cite{GWZ22}, which is a variant of that in Hickman and Rogers \cite{HR19}. Let us be slightly more precise. We will repeat precisely the polynomial partitioning algorithm in Subsections 5.2-5.4 in \cite{GWZ22}, where Lemma \ref{230827lemma4_5} was repeatedly applied. In this algorithm, we will fix small parameters $\delta_j, j=0, 1, 2, 3$, satisfying 
\begin{equation}\label{230902e4_19}
\delta \ll \delta_3\ll \delta_2\ll \delta_1\ll \delta_0\ll \epsilon.
\end{equation}
For instance, one can take $\delta_j= \delta_{j-1}^{10}$ for $j=3, 2, 1$ and $\delta_0=\epsilon^{10}, \delta=\delta_3^{10}$. Here $\delta$ is the same as that in \eqref{230828e4_11oo}. We partition $\B^3_R$ into a finitely overlapping collection of balls $\{B_{\iota}\}_{\iota}$, each of which is of radius $R^{1-\delta}$.  Let us recall the output of this algorithm in \cite[page 48]{GWZ22}. 
\begin{enumerate}
\item[Output1] We obtain a sequence of nodes 
\begin{equation}
\mathfrak{n}_0^*, \mathfrak{n}_1^*, \ldots, \mathfrak{n}_{\ell_0}^*.
\end{equation}
Each node $\mfn^*_{\ell}$ with $0\le \ell\le \ell_0\in \N$ is a collection of open sets in $\R^3$, and is assigned several parameters: A dimension parameter $\dim(\mfn^*_{\ell})$ and a radius parameter $\rho(\mfn^*_{\ell})$. As we are in $\R^3$, we have $\dim(\mfn^*_{\ell})$ takes values in $\{2, 3\}$, and is non-increasing in $\ell$. Let $\ell_-\in \{0, 1, \dots, \ell_0\}$ be such that 
\begin{equation}
\dim(\mfn^*_{\ell})=3, \ \ \text{if } 0\le \ell\le \ell_-,
\end{equation}
and $\dim(\mfn^*_{\ell})=2$ otherwise.  It may happen that $\ell_-=\ell_0$. However, this case is easy to handle, and we therefore assume that we always have $\ell_-< \ell_0$. 
\item[Output2] The node $\mfn^*_{\ell_-}$ is particularly important. Denote
\begin{equation}
\mathfrak{S}_3:= \mfn^*_0, \ \ \mathfrak{S}_2:= \mfn^*_{\ell_-},
\end{equation}
and 
\begin{equation}
r_3:=\rho(\mathfrak{S}_3)=R, \ \ r_2:= \rho(\mathfrak{S}_2), \ \ r_1:=1. 
\end{equation}
The node $\mfn^*_0$ consists of only one element, $\B^3_R$. Elements in $\mathfrak{S}_2$ are of the form 
\begin{equation}\label{230827e4_24}
B_{r_2}\cap \mc{N}_{
r_2^{1/2+\delta_2}
}(S_2),
\end{equation}
where $B_{r_2}\subset \B^3_R$ is a ball of radius $r_2$ and $S_2$ is an algebraic variety of dimension two. \footnote{This is why in the previous item we let $\dim(\mfn_{\ell_-}^*)=2$. } To simplify notation, we will use $S_2$ to refer to \eqref{230827e4_24}, whenever it is clear from the context that we are talking about $\mathfrak{S}_2$.
\item[Output 3] Each open set $O\in \mfn^*_{\ell_0}$ has diameter at most $R^{\delta_0}$. This is the stopping condition of the algorithm (see \cite[page 41]{GWZ22}), which says that the algorithm terminates whenever we reach a scale $\le R^{\delta_0}$. Each $O\in \mfn^*_{\ell_0}$
 is associated with a function $f_{\iota, O}$, which is built with a collection of tubes from $\T[B_{\rho(\mfn^*_{\ell_0})}]$, that is, 
\begin{equation}
f_{\iota, O}=\sum_{T\in \T'[B_{\rho(\mfn^*_{\ell_0})}]} f_T,
\end{equation}
where $\T'[B_{\rho(\mfn^*_{\ell_0})}]$ is a sub-collection of tubes in $\T[B_{\rho(\mfn^*_{\ell_0})}]$, and  $B_{\rho(\mfn^*_{\ell_0})}\subset \R^3$ is the ball of radius $\rho(\mfn^*_{\ell_0})$ that contains $O$. 
\item[Output4] For each $\iota$ and each $S_2\in \mathfrak{S}_2$ with $S_2\cap B_{\iota}\neq \emptyset$, there is an associated function $f^*_{\iota, S_2}$. This function is built with a collection of tubes from $\T[B_{r_2}]$, where $B_{r_2}$ is as in \eqref{230827e4_24}. Most importantly, all the tubes in this collection are contained in $S_2$.

\item[Output5] In the end, we have parameters $D_3, D_2, D_1$ that are integer powers of $d$ satisfying $D_3=1$, and 
\begin{equation}\label{230828e4_26cc}
D_1\le r_2, \ \ r_1 D_1 D_2 D_3\le R, \ \ r_2 D_2 D_3\le R. 
\end{equation}
See Lemma 5.10 in \cite[page 50]{GWZ22}.\footnote{The bound $D_1\le r_2$ is not stated explicitly in Lemma 5.10 in \cite[page 50]{GWZ22}, but it can proven easily via the same argument. }
\end{enumerate}
Next, let us state what estimates the outputs satisfy. Denote $p_3=p$, $\alpha_3=\beta_3=1$, $\beta_2=\alpha_2$, and $\alpha_2\in [0, 1]$ is to be determined. Let $p_2$ be such that 
\begin{equation}\label{230828e4_27cc}
\frac{1}{p_3}=\frac{1-\alpha_2}{2}+\frac{\alpha_2}{p_2}.
\end{equation} 
Then the above outputs satisfy the following properties. 
\begin{enumerate}
\item[Property 1] We can find $C_p>0$ depending only on $p$ and $A_2\in \N$ with $A_2\le A_0$ depending only on $\epsilon$, such that  
\begin{equation}\label{230418e7_20}
\begin{split}
\|
T^N f
\|_{\mathrm{BL}_{A_0}^p\left(\B^3_R\right)} & \lesssim 
R^{C_p \delta_0} 
(r_2 D_2)^{\frac{1}{2}(1-\beta_2)}
\|f\|_2^{1-\beta_{2}}\\
& \left(\sum_{S_{2}
\in \mathfrak{S}_2
}
\sum_{\iota}
\left\|
T^N f^*_{\iota, S_{2}}
\right\|^{p_2}_{\mathrm{BL}_{A_2}^{p_{2}}\left(B_{r_{2}}\right)}\right)^{\frac{\beta_{2}}{p_{2}}},
\end{split}
\end{equation}
where $B_{r_2}$ is the ball of radius $r_2$ that contains $S_2$, given as in \eqref{230827e4_24}. We remark here that $C_p$ will also appear below, and its precise values will not be important and may change from line to line. 
\item[Property 2] We also have 
\begin{equation}
\sum_{S_2
\in \mathfrak{S}_2
} 
\norm{f^*_{\iota, S_2}}_2^2 \lesim R^{C_p \delta_0} D_2 \norm{f}_2^2.
\end{equation}
\item[Property 3] In the end, we have 
\begin{equation}\label{230828e4_30}
\max_{S_2
\in \mathfrak{S}_2
} \norm{f^*_{\iota, S_2}}_2^2 \lesim R^{C_p \delta_0}
D_2^{-2} \norm{f}_2^2
\end{equation}
and 
\begin{equation}\label{230828e4_31}
\max_{S_2
\in \mathfrak{S}_2
}  \max_{\ell(\theta)=\rho^{-1/2}}
\norm{f^*_{\iota, S_2}}_{L^2_{\mathrm{avg}}(\theta)}^2 \lesim R^{C_p \delta_0} 
\max_{\ell(\theta)=\rho^{-1/2}}
\norm{f}^2_{
L^2_{\mathrm{avg}}(\theta)
}
\end{equation}
for all $1\le \rho\le r_2$, where the max runs over all squares $\theta\subset \R^2$ of side length $\ell(\theta)=\rho^{-1/2}$, and 
\begin{equation}
\|h\|^2_{L^2_{\mathrm{avg}}(\theta)}:= \frac{1}{\mathcal{L}^2(\theta)} \int_{\theta}|h|^2
\end{equation}
for a function $h: \R^2\to \R$. 
\end{enumerate}
These three properties are taken from \cite[page 51]{GWZ22}, where we set $n=3, n'=2$. We do not need Property 4 there because it is relevant only for estimates in $\R^n$ with $n\ge 4$. \\

In Property 1, we connect the scale $r_3=\rho(\mfn^*_0)$ with the scale $r_2=\rho(\mfn^*_{\ell-})$. Now we state a few estimates that connect the scale  $r_3=\rho(\mfn^*_0)$ with the smallest scale $\rho(\mfn^*_{\ell_0})$. More precisely, we have 
\begin{equation}\label{230828e4_33}
\|
T^N f
\|_{\mathrm{BL}_{A_0}^p\left(\B^3_R\right)}
\lesim
R^{C_p \delta_0}
(r_2 D_2)^{\frac{1}{2}(1-\beta_2)}
D_1^{\frac{\beta_2}{p_2}} D_2^{\frac{\beta_2}{p_2}}
\norm{f}_2^{\frac{2}{p_3}} 
\max_{O\in 
\mfn^*_{\ell_0}
} \norm{f_{\iota, O}}_2^{1-\frac{2}{p_3}},
\end{equation}
where $\iota$ refers to the unique $B_{\iota}\subset \B^3_{R}$ of radius $R^{1-\delta}$ containing $O$. This estimate is precisely from equations (9.1)-(9.2) in \cite[page 86]{GWZ22}. Moreover, 
\begin{equation}\label{230828e4_34}
\max _{O \in 
\mfn^*_{\ell_0}
}\left\|f_{\iota, O}\right\|_2^2 
\lesssim
R^{C_p \delta_0}
r_2^{-\frac{1}{2}} D_{1}^{-1}
\max_{S_2\in 
\mathfrak{S}_2
}\left\|f^*_{\iota, S_2}\right\|_2^2,
\end{equation}
which is equation (9.3) in \cite[page 87]{GWZ22}, with $n=3, n'=2$. This finishes recalling the outputs of the algorithm in \cite{GWZ22} and the properties of the outputs. \\

Before we proceed, let us first see how to prove \eqref{230828e4_17} for the smaller range $p\ge 10/3$. We combine \eqref{230828e4_33}, \eqref{230828e4_34}, \eqref{230828e4_30}
and obtain 
\begin{equation}\label{230828e4_36cc}
\begin{split}
\|
T^N f
\|_{\mathrm{BL}_{A_0}^p\left(\B^3_R\right)}
&
\lesim R^{C_p \delta_0}
(r_2 D_2)^{\frac{1}{2}(1-\beta_2)}
D_1^{\frac{\beta_2}{p_2}} D_2^{\frac{\beta_2}{p_2}} r_2^{
-\frac{1}{2}(\frac{1}{2}-\frac{1}{p})
}
D_1^{-(\frac{1}{2}-\frac{1}{p})}
D_2^{
-2(\frac{1}{2}-\frac{1}{p})
}
\norm{f}_2\\
&\lesim R^{C_p \delta_0} (r_2 D_2)^{\frac{1}{2}(1-\beta_2)}
D_1^{\frac{\beta_2}{2}-2(\frac{1}{2}-\frac{1}{p})} 
r_2^{
-\frac{1}{2}(\frac{1}{2}-\frac{1}{p})
}
D_2^{
\frac{\beta_2}{2}
-3(\frac{1}{2}-\frac{1}{p})
}
\norm{f}_2,
\end{split}
\end{equation}
where in the second inequality we used \eqref{230828e4_27cc}. Recall the relations in \eqref{230828e4_26cc}, and that $\beta_2$ is a free parameter to choose. We impose the constraint that 
\begin{equation}
\frac{\beta_2}{2}
-2(\frac12-\frac1p)\ge 0,
\end{equation}
and obtain 
\begin{equation}
\begin{split}
\eqref{230828e4_36cc} & 
\lesim R^{C_p \delta_0}
(r_2 D_2)^{\frac{1}{2}(1-\beta_2)}
r_2^{\frac{\beta_2}{2}-2(\frac{1}{2}-\frac{1}{p})} 
r_2^{
-\frac{1}{2}(\frac{1}{2}-\frac{1}{p})
}
D_2^{
\frac{\beta_2}{2}
-3(\frac{1}{2}-\frac{1}{p})
}
\norm{f}_2\\
& 
\lesim 
R^{C_p \delta_0}
r_2^{
\frac12
-\frac{5}{2}(\frac{1}{2}-\frac{1}{p})
}
D_2^{
\frac{1}{2}
-3(\frac{1}{2}-\frac{1}{p})
}
\norm{f}_2.
\end{split}
\end{equation}
Note that the exponent of $D_2$ is always negative as $p>3$. 
From this, we see the exponent $p\ge 10/3$.

\subsection{Polynomial Wolff axioms}\label{230922subsection4_3}

Now let us focus on the improved range \eqref{230828e4_18qq}. This will rely on polynomial Wolff axioms. Recall the definition of the function $\Phi: \R^2\times \R\times \R^2\to \R^2$ in \eqref{230828e4_10}. For $\kappa\in (0, 1)$, we use $\theta$ to denote a dyadic square on $\R^2$ of side length $\kappa$. Moreover, we use $\xi_{\theta}$ to denote the center of $\theta$. For $v\in \R^2$ with $|v|\le 1$, define 
\begin{equation}
T_{\xi_{\theta}, v, \Phi}(\kappa, 1):= 
\{(x, t)\in \R^2\times \R: 
|x-\Phi(v, t; \xi_{\theta})|\le \kappa, \ |t|\le 1
\}.
\end{equation}
In $T_{\xi_{\theta}, v, \Phi}(\kappa, 1)$, we use $1$ to indicate that it is a (curved) tube of length one; it is a rescaled version of the tubes defined in \eqref{230828e4_12}. The main reason of rescaling the curved tubes is for our notation to be consistent with that in \cite[Section 3]{GWZ22}; our $T_{\xi_{\theta}, v, \Phi}(\kappa, 1)$ is precisely $T_{\xi_{\theta}, v, \Phi}(\delta, 1)$ in equation (3.2) in \cite{GWZ22}, and we are not using $\delta$ because it was used previously. \\

For a collection $\T$ of tubes $\{T_{\xi_{\theta}, v, \Phi}(\kappa, 1)\}$, we say that the tubes in $\T$ point in different direction if no any two tubes share the same $\xi_{\theta}$. 

\begin{theorem}\label{230831theorem4_6}
Let $S\subset \B^3$ be a semi-algebraic set of complexity $\le E$. Let $\T=\{T_{\xi_{\theta}, v, \Phi}(\kappa, 1)\}$ be a collection of tubes pointing in different directions.  Then 
\begin{equation}
\#\{T \in \mathbb{T}: T \subset S\} \lesssim_{\phi, E, \epsilon'} 
\mc{L}^3(S)
 \kappa^{-2-\epsilon'},
\end{equation}
for every $\epsilon'>0$. 
\end{theorem}
\begin{proof}[Proof of Theorem \ref{230831theorem4_6}]
We will apply \cite[Theorem 3.1]{GWZ22}. To apply this theorem, it suffices to show that there exists $\epsilon_{\phi}>0$ such that 
\begin{equation}\label{230831e4_42}
\int_{|t|\le \epsilon_{\phi}}
\left|\operatorname{det}\left(\nabla_v \Phi(v, t; \xi) \cdot M+\nabla_{\xi} \Phi(v, t; \xi)\right)\right| \mathrm{d} t \gtrsim_{\phi} 1,
\end{equation}
for all $|v|\le \epsilon_{\phi}, |\xi|\le \epsilon_{\phi}$ and $2\times 2$ matrices $M$. In particular, the implicit constant is not allowed to depend on $v, \xi, M$. This is the step where we will use the assumption that $\phi$ is of contact order $\le k$ at the origin. \\

To prove \eqref{230831e4_42}, let us first prove that the notion of contact orders is stable under perturbations. To be more precise, we claim that the phase function $\phi$ is of contact order $\le k$ at $(x_0, t_0; \xi_0)$ whenever $|x_0|\le \epsilon_{\phi}, |t_0|\le \epsilon_{\phi}$ and $|\xi_0|\le \epsilon_{\phi}$. Denote 
\begin{equation}
\phi_0(x, t; \xi):=\phi(x+x_0, t+t_0; \xi+\xi_0)-\phi(x_0, t_0; \xi+\xi_0).
\end{equation}
Moreover, denote 
\begin{equation}\label{231015e4_44}
D_{ij}(t; x_0, t_0; \xi_0):= 
\partial_{\xi_i}
\partial_{\xi_j}
\phi_0(
X_0(t), t; 0
), \ \ 1\le i, j\le 2,
\end{equation}
and 
\begin{equation}
\mathfrak{D}(t; x_0, t_0; \xi_0):=
\det
\begin{bmatrix}
D_{11}(t; x_0, t_0; \xi_0), & D_{12}(t; x_0, t_0; \xi_0)\\
D_{21}(t; x_0, t_0; \xi_0), & D_{22}(t; x_0, t_0; \xi_0)
\end{bmatrix}
\end{equation}
Recall that our assumption is that $\phi$ is of contact order $\le k$ at $x_0=0, t_0=0, \xi_0=0$, that is, the matrix 
\begin{equation}\label{231015e4_46}
\begin{bmatrix}
\mathfrak{D}'(0; 0, 0; 0), & \mathfrak{D}''(0; 0, 0; 0), & \dots, & \mathfrak{D}^{(k)}(0; 0, 0; 0),\\
D_{11}'(0; 0, 0; 0), & D_{11}''(0; 0, 0; 0), & \dots, & D_{11}^{(k)}(0; 0, 0; 0)\\
D_{12}'(0; 0, 0; 0), & D_{12}''(0; 0, 0; 0), & \dots, & D_{12}^{(k)}(0; 0, 0; 0)\\
D_{22}'(0; 0, 0; 0), & D_{22}''(0; 0, 0; 0), & \dots, & D_{22}^{(k)}(0; 0, 0; 0)
\end{bmatrix}
\end{equation}
has rank $4$. The claim  follows from continuity because having full rank is stable under perturbations. \\

Now let us prove \eqref{230831e4_42}. Recall the definition of $\Phi$ from \eqref{230828e4_10} that 
\begin{equation}\label{230831e4_47}
\nabla_{\xi} \phi\pnorm{
\Phi(v, t; \xi), t; \xi
}=v.
\end{equation}
Taking derivatives in $v$ on both sides, we obtain 
\begin{equation}
\nabla_x \nabla_{\xi} \phi\pnorm{
\Phi(v, t; \xi), t; \xi
} \nabla_v \Phi(v, t; \xi)=I_{2\times 2}.
\end{equation}
Taking derivatives in $\xi$ on both sides of \eqref{230831e4_47}, we obtain 
\begin{equation}
\nabla_x \nabla_{\xi}\phi
\pnorm{
\Phi(v, t; \xi), t; \xi
} \nabla_{\xi} \Phi(v, t; \xi)+
\nabla_{\xi}^2 
\phi\pnorm{
\Phi(v, t; \xi), t; \xi
}=0.
\end{equation}
Therefore \eqref{230831e4_42} amounts to proving that 
\begin{equation}\label{231015e4_50}
\int_{|t|\le \epsilon_{\phi}}
\left|\operatorname{det}\left(M+\nabla^2_{\xi} 
\phi
\pnorm{
\Phi(v, t; \xi), t; \xi
}\right)\right| \mathrm{d} t \gtrsim_{\phi} 1,
\end{equation}
uniformly in $v, \xi$ and $M$. Write 
\begin{equation}\label{231015e4_51}
\begin{split}
& \operatorname{det}\left(M+\nabla^2_{\xi} 
\phi
\pnorm{
\Phi(v, t; \xi), t; \xi
}\right)\\
& =
\operatorname{det}\left(\widetilde{M}+\nabla^2_{\xi} 
\phi
\pnorm{
\Phi(v, t; \xi), t; \xi
}
-\nabla^2_{\xi} 
\phi
\pnorm{
\Phi(v, 0; \xi), 0; \xi
}
\right),
\end{split}
\end{equation}
for some new $2\times 2$ matrix
\begin{equation}
\widetilde{M}=
\begin{bmatrix}
\widetilde{m}_{11}, & \widetilde{m}_{12}\\
\widetilde{m}_{21}, & \widetilde{m}_{22}
\end{bmatrix}
\end{equation}
 We take $t_0=0$ in \eqref{231015e4_44}, $x_0$ in \eqref{231015e4_44} satisfying 
\begin{equation}
\nabla_{\xi} \phi(x_0, 0; \xi)=v,
\end{equation}
and obtain 
\begin{equation}
\begin{split}
\eqref{231015e4_51}& = 
\mathfrak{D}(t; x_0, 0; \xi)+ \widetilde{m}_{11} D_{22}(t; x_0, 0; \xi)+ 
\widetilde{m}_{22} D_{11}(t; x_0, 0; \xi)\\
&-(\widetilde{m}_{12}+ \widetilde{m}_{21}) D_{12}(t; x_0, 0; \xi)+\det(\widetilde{M}).
\end{split}
\end{equation}
Let us write 
\begin{equation}\label{231015e4_55}
\operatorname{det}\left(M+\nabla^2_{\xi} 
\phi
\pnorm{
\Phi(v, t; \xi), t; \xi
}\right)
= c_0+ \sum_{k'=1}^k c_{k'} t^{k'}+ \text{ higher order terms},
\end{equation}
where $c_{k'}$ are constants that depend on the choice of $v, \xi$ and $M$. From \eqref{231015e4_46} we know that 
\begin{equation}
\sum_{k'=1}^k |c_{k'}|\gtrsim_{\phi} 1+\|(
\widetilde{m}_{11}, \widetilde{m}_{22}, \widetilde{m}_{12}+ \widetilde{m}_{21}
)\|_{1},
\end{equation}
where 
\begin{equation}
\|(m_1, m_2, m_3)\|_{1}:=|m_1|+|m_2|+|m_3|,
\end{equation}
for $m_1, m_2, m_3\in \R$. By dividing the coefficients in \eqref{231015e4_55} by a constant, we see that the desired bound \eqref{231015e4_50} follows from 
\begin{claim}\label{231015claim4_7}
Let $k\ge 1$ be an integer. 
Let $W: \R\to \R$ be a smooth function of the form
\begin{equation}\label{231015e4_58}
c_0+\sum_{k'=1}^k c_{k'}t^{k'} +O_{\phi, k}(t^{k+1})
\end{equation}
satisfying 
\begin{equation}\label{231015e4_59}
\sum_{k'=1}^k |c_{k'}|\gtrsim_{\phi, k} 1,
\end{equation}
where the implicit constant in \eqref{231015e4_58} depends on $\phi$ and $k$. Then there exists $\epsilon_{\phi, k}$ such that 
\begin{equation}\label{231015e4_60}
\int_{
|t|\le \epsilon_{\phi, k} 
}|W(t)|dt\gtrsim_{\phi, k} 1.
\end{equation}
\end{claim}
\begin{proof}[Proof of Claim \ref{231015claim4_7}]
The proof is almost immediate. From \eqref{231015e4_59} we can conclude that there exist $\epsilon_{\phi, k}$ and $k''\in \{0, 1, \dots, k\}$ such that 
\begin{equation}
|W^{(k'')}(t)|\gtrsim_{\phi, k} 1,
\end{equation}
for all $|t|\le \epsilon_{\phi, k}$, which implies that $|W(t)|$ can not stay close to $0$ for too long time, and further implies the desired lower bound \eqref{231015e4_60}. 
\end{proof}

Claim \ref{231015claim4_7} implies the lower bound \eqref{231015e4_50}, and we have therefore finished the proof of Theorem \ref{230831theorem4_6}.

\end{proof}

Let us state a corollary of Theorem \ref{230831theorem4_6}. Recall the outputs of the polynomial partitioning algorithm. In particular, recall that $r_2=\rho(\mathfrak{S}_2)$, that elements in $\mathfrak{S}_2$ are of the form 
\begin{equation}\label{230827e4_24zzz}
B_{r_2}\cap \mc{N}_{
r_2^{1/2+\delta_2}
}(S_2),
\end{equation}
where $B_{r_2}\subset \B^3_R$ is a ball of radius $r_2$ and $S_2$ is an algebraic variety of dimension two, and that the function $f^*_{\iota, S_2}$ is built with a collection of tubes from $\T[B_{r_2}]$ that are contained in \eqref{230827e4_24zzz}. 
\begin{corollary}\label{230902corollary4_7}
Assume that $r_2=R$. Then for every $S_2$, it holds that 
\begin{equation}\label{230902e4_51cc}
\norm{
f^*_{\iota, S_2}
}_2^2 \lesim_{\phi, \epsilon}
r_2^{-1/2+\delta_0} 
\norm{f}_{\infty}^2.
\end{equation}
Here the relation for $\delta_2, \delta_0, \epsilon$ are given in \eqref{230902e4_19}, and the implicit constant depends on $\epsilon$. 
\end{corollary}
\begin{proof}[Proof of Corollary \ref{230902corollary4_7}]
Write $h:= f^*_{\iota, S_2}$. Let us write $h$ using the wave packet decomposition in Subsection \ref{230902subsection4_1}, 
\begin{equation}
h=\sum_{\theta} h_{\theta}, \ \ h_{\theta}:=\sum_{v} h_{\theta, v},
\end{equation}
where $\theta$ denotes a frequency square of side length $r_2^{-1/2}$. By the orthogonality estimate in \eqref{230828e4_31}, 
\begin{equation}\label{230902e4_54yy}
\norm{h}_2^2 \lesim \sum_{\theta} \norm{h_{\theta}}_2^2 \lesim 
r_2^{-1}
\sum_{\theta}
\norm{f}_{L^2_{\mathrm{avg}}(\theta)}^2,
\end{equation}
where the factor $r_2^{-1}$ comes from taking averaged integral over $\theta$. By Theorem \ref{230831theorem4_6}, the number of $\theta$ that contributes to the above sum is $\lesim_{\epsilon} r_2^{1/2+\delta_0}$. This, combined with \eqref{230902e4_54yy}, implies the desired bound. 
\end{proof}

Before we proceed to polynomial Wolff axioms at general scales (general $r_2$ instead of $r_2=R$), let us first see why \eqref{230902e4_51cc} is more favorable compared with previously used estimates in \eqref{230828e4_36cc}. In \eqref{230828e4_36cc}, we combined \eqref{230828e4_34}, \eqref{230828e4_30}, and obtained 
\begin{equation}\label{230902e4_54}
\max _{O \in 
\mfn^*_{\ell_0}
}\left\|f_{\iota, O}\right\|_2^2 
\lesssim
R^{C_p \delta_0}
r_2^{-\frac{1}{2}} D_{1}^{-1}
D_2^{-2} \norm{f}_2^2.
\end{equation}
Now we replace \eqref{230828e4_30} by \eqref{230902e4_51cc}, and obtain 
\begin{equation}\label{230902e4_55}
\begin{split}
\max _{O \in 
\mfn^*_{\ell_0}
}\left\|f_{\iota, O}\right\|_2^2 
& 
\lesssim
R^{C_p \delta_0}
r_2^{-\frac{1}{2}} D_{1}^{-1}
\max_{S_2\in 
\mathfrak{S}_2
}
\left\|f^*_{\iota, S_2}\right\|_2^2 \\
& 
\lesssim
R^{C_p \delta_0}
r_2^{-\frac{1}{2}} D_{1}^{-1} r_2^{-\frac12} \norm{f}_{\infty}^2.
\end{split}
\end{equation}
Recall the relation $r_2 D_2\le R$ in \eqref{230828e4_26cc}. In other words, if $r_2\simeq R$, then $D\simeq 1$. In this case, the bound \eqref{230902e4_55} is much better than the bound \eqref{230902e4_54}; the trade-off is that $\norm{f}_2$ in \eqref{230902e4_54} is replaced by $\norm{f}_{\infty}$, which we can afford. \\

Next, let us try to understand rescaled versions of the above polynomial Wolff axioms. More precisely, we will prove an analog of Corollary \ref{230902corollary4_7} for general $r_2\le R$. 
\begin{proposition}\label{230902prop4_8}
For every $S_2\in \mathfrak{S}_2$ and $r_2=\rho(\mathfrak{S}_2)$, it holds that 
\begin{equation}
\norm{
f^*_{\iota, S_2}
}_2^2 \lesim_{\phi, \epsilon}
\min\set{
\pnorm{
\frac{N}{r_2}
}^{k-2}
r_2^{-1/2},
\pnorm{
r_2^{-1/2}+\frac{r_2}{N}
}
}
(r_2)^{\delta_0}
\norm{f}_{\infty}^2.
\end{equation}
Here $\delta_0$ and $\epsilon$ are the same as in Corollary \ref{230902corollary4_7}. 
\end{proposition}
\begin{proof}[Proof of Proposition \ref{230902prop4_8}]
To simplify notation, let us without loss of generality assume that $S_2$ is contained in the ball $\B^3_{r_2}\subset \B^3_R$, that is, the ball of radius $r_2$ centered at the origin. In this case, wave packets are defined in \eqref{230828e4_12}, that is 
\begin{equation}\label{230828e4_12kk}
T_{\theta, v}:=
\Big\{(x, t): 
\anorm{
\frac{x}{N}-\Phi\pnorm{
\frac{v}{N}, \frac{t}{N}; \xi_{\theta}
}
}\le 
\frac{r_2^{1/2+\delta}}{N}, \ t\in [0, r]
\Big\}.
\end{equation}
Note that $T_{\theta, v}$ is a curved tube of length $r_2$. Let us rescale it by a factor of $r_2$, and consider everything at the unit scale: 
\begin{equation}\label{230902e4_58}
T^{\circ}_{\theta, v}:=
\Big\{(x, t): 
\anorm{
\frac{r_2 x}{N}-\Phi\pnorm{
\frac{r_2 v}{N}, \frac{rt}{N}; \xi_{\theta}
}
}\le 
\frac{r_2^{1/2+\delta}}{N}, \ t\in [0, 1]
\Big\}.
\end{equation}
Rewrite the defining equation for $T^{\circ}_{\theta, v}$ in \eqref{230902e4_58} as
\begin{equation}
\anorm{
x-
N_2
\Phi\pnorm{
\frac{v}{N_2}, \frac{t}{N_2}; \xi_{\theta}
}
}\le 
\frac{r^{1/2+\delta}}{r},
\end{equation}
where $N_2:=N/r_2$.  Denote 
\begin{equation}
\Phi_{N_2}(v, t; \xi):= 
N_2
\Phi\pnorm{
\frac{v}{N_2}, \frac{t}{N_2}; \xi
}.
\end{equation}
 Recall that $\Phi$ satisfies equation \eqref{230831e4_47}, that is, 
\begin{equation}\label{230902e4_60}
\nabla_{\xi} \phi\pnorm{
\Phi(v, t; \xi), t; \xi
}=v,
\end{equation}
which be written as 
\begin{equation}\label{231016e4_74}
\frac{v}{N_2}=
\nabla_{\xi} \phi 
\pnorm{
\frac{1}{N_2} \Phi_{N_2}(v, t; \xi), \frac{t}{N_2}; \xi
}=
\nabla_{\xi}
\phi_{N_2}(
\Phi_{N_2}(v, t; \xi), t; \xi
),
\end{equation}
with 
\begin{equation}\label{231016e4_75}
\phi_{N_2}(x, t; \xi):=
N_2 
\phi\pnorm{
\frac{x}{N_2}, \frac{t}{N_2}; \xi
}.
\end{equation}
Denote 
\begin{equation}
T_{\xi_{\theta}, v, \Phi_{N_2}}(\kappa, 1):= 
\{(x, t)\in \R^2\times \R: 
|x-\Phi_{N_2}(v, t; \xi_{\theta})|\le \kappa, \ |t|\le 1
\}.
\end{equation}
Similarly to Theorem \ref{230831theorem4_6}, we have 
\begin{claim}\label{230902claim4_9}
Let $S\subset \B^3$ be a semi-algebraic set of complexity $\le E$. Let $\T=\{T_{\xi_{\theta}, v, \Phi_{N_2}}(\kappa, 1)\}$ be a collection of tubes pointing in different directions.  Then 
\begin{equation}
\#\{T \in \mathbb{T}: T \subset S\} \lesssim_{\phi, E, \epsilon} 
\min\{
\mc{L}^3(S) (N_2)^{k-2}, \mathcal{L}^3(S_{N_2^{-1}})
\}
\kappa^{-2-\epsilon'},
\end{equation}
for every $\epsilon'>0$, where $k$ is the contact order of $\phi$ and $S_{N_2^{-1}}$ denotes the $N_2^{-1}$ neighborhood of $S$. 
\end{claim}
\begin{proof}[Proof of Claim \ref{230902claim4_9}]
The proof of the upper bound
\begin{equation}
\#\{T \in \mathbb{T}: T \subset S\} \lesssim_{\phi, E, \epsilon} 
\mc{L}^3(S) (N_2)^{k-2}
\kappa^{-2-\epsilon'}
\end{equation}
 is essentially the same as that of Theorem \ref{230831theorem4_6}. The key difference is that \eqref{230831e4_42} is no longer true anymore, and instead we have 
\begin{equation}\label{230831e4_66}
\int_{|t|\le \epsilon_{\phi}}
\left|\operatorname{det}\left(\nabla_v \Phi_{N_2}(v, t; \xi) \cdot M+\nabla_{\xi} \Phi_{N_2}(v, t; \xi)\right)\right| \mathrm{d} t \gtrsim_{\phi} 
(N_2)^{-k+2},
\end{equation}
uniformly in $v, \xi$ and $2\times 2$ matrices $M$. With \eqref{230831e4_66} in hand, one can repeat the proof of Theorem 3.1 in \cite{GWZ22}, and obtain Claim \ref{230902claim4_9}.\\

It remains to prove \eqref{230831e4_66}. Let us first write it using $\phi_{N_2}$. By \eqref{231016e4_74} and \eqref{231016e4_75}, it is equivalent to prove that 
\begin{equation}\label{230831e4_66zz}
\int_{|t|\le \epsilon_{\phi}}
\left|\operatorname{det}\left(M+\nabla^2_{\xi} \phi_{N_2}\Big(
\Phi_{N_2}(v, t; \xi)
, t; \xi
\Big)\right)\right| \mathrm{d} t \gtrsim_{\phi} 
(N_2)^{-k+2},
\end{equation}
uniformly in $v, \xi$ and $2\times 2$ matrices $M$. By using the stability phenomenon that we observed in \eqref{231015e4_46}, we can without loss of generality assume that we are working with $v=0, \xi=0$. Under this simplification, we have 
\begin{equation}
\Phi_{N_2}(0, t; 0)=0, \ \ \forall t,
\end{equation}
as $\phi$ is of a normal form, and therefore \eqref{230831e4_66zz} becomes 
\begin{equation}\label{230831e4_66zzz}
\int_{|t|\le \epsilon_{\phi}}
\left|\operatorname{det}\left(M+\nabla^2_{\xi} \phi_{N_2}(
0
, t; 0)\right)\right| \mathrm{d} t \gtrsim_{\phi} 
(N_2)^{-k+2},
\end{equation}
which follows from Claim \ref{231015claim4_7}. \\

Let us turn to the other upper bound. We without loss of generality assume that 
\begin{equation}
\phi(\bfx; \xi)=\inn{x}{\xi}+\frac12 t |\xi|^2 + \psi(\bfx; \xi),\ \ \bfx=(x, t),
\end{equation}
where 
\begin{equation}
\psi(\bfx; \xi)=O(|t||\xi|^3+ |\bfx|^2 |\xi|^2).
\end{equation}
Under this form, we have 
\begin{equation}
\phi_{N_2}(\bfx; \xi)=
\inn{x}{\xi}+\frac12 t|\xi|^2
+ 
N_2 \psi(
\frac{x}{N}, \frac{t}{N}; \xi
).
\end{equation}
Let us write 
\begin{equation}
\psi(\bfx; \xi)=t Q_3(\xi)+ \psi_2(\bfx; \xi),
\end{equation}
with 
\begin{equation}
Q_3(\xi)=O(|\xi|^3), \ \ \ 
\psi_2(\bfx; \xi)=
O(|\bfx|^2|\xi|^2).
\end{equation}
We have 
\begin{equation}
\phi_{N_2}(\bfx; \xi)=
\inn{x}{\xi}+\frac12 t|\xi|^2
+ 
tQ_3(\xi)+ 
N_2 \psi_2(
\frac{x}{N}, \frac{t}{N}; \xi
).
\end{equation}
The tube $T_{\xi_{\theta}, v, \Phi_{N_2}}(\kappa, 1)$ is given by the $\kappa$ neighborhood of the curve 
\begin{equation}\label{231022e4_87}
\{(x, t): \nabla_{\xi}
\phi_{N_2}(x, t; \xi_{\theta})= v\}.
\end{equation}
Define 
\begin{equation}
\widetilde{\phi}_{N_2}(\bfx; \xi)=
\inn{x}{\xi}+\frac12 t|\xi|^2
+ 
tQ_3(\xi).
\end{equation}
Define the tube $\widetilde{T}_{\xi_{\theta}, v, \Phi_{N_2}}(\kappa, 1)$ to be the $\kappa$ neighborhood of the curve 
\begin{equation}\label{231022e4_89}
\{(x, t): \nabla_{\xi}
\widetilde{\phi}_{N_2}(x, t; \xi_{\theta})= v\}.
\end{equation}
Take a point $(x, t)$ from the curve \eqref{231022e4_87} and a point $(x', t)$ from the curve \eqref{231022e4_89}. By adding a zero, we obtain 
\begin{equation}
\nabla_{\xi} \phi_{N_2}(x, t; \xi_{\theta})-\nabla_{\xi} \phi_{N_2}(x', t; \xi_{\theta})+ \nabla_{\xi} \phi_{N_2}(x', t; \xi_{\theta})-\nabla_{\xi} \widetilde{\phi}_{N_2}(x', t; \xi_{\theta})=0.
\end{equation}
By H\"ormander's non-degeneracy condition, we know that 
\begin{equation}
|\nabla_{\xi} \phi_{N_2}(x, t; \xi_{\theta})-\nabla_{\xi} \phi_{N_2}(x', t; \xi_{\theta})| \simeq |x-x'|.
\end{equation}
Moreover, 
\begin{equation}\label{231022e4_92}
|\nabla_{\xi} \phi_{N_2}(x', t; \xi_{\theta})-\nabla_{\xi} \widetilde{\phi}_{N_2}(x', t; \xi_{\theta})|\lesim N_2^{-1}.
\end{equation}
As a consequence, we know that the tube $T_{\xi_{\theta}, v, \Phi_{N_2}}(\kappa, 1)$ is in the $N_2^{-1}$ neighborhood of the tube $\widetilde{T}_{\xi_{\theta}, v, \Phi_{N_2}}(\kappa, 1)$.
\begin{claim}\label{231022claim4_11}
Let $S\subset \B^3$ be a semi-algebraic set of complexity $\le E$. Let $\widetilde{\T}=\{\widetilde{T}_{\xi_{\theta}, v, \Phi_{N_2}}(\kappa, 1)\}$ be a collection of tubes pointing in different directions.  Then 
\begin{equation}
\#\{\widetilde{T} \in \widetilde{\mathbb{T}}: \widetilde{T} \subset S\} \lesssim_{\phi, E, \epsilon} 
\mc{L}^3(S)
\kappa^{-2-\epsilon'},
\end{equation}
for every $\epsilon'>0$. 
\end{claim}
\begin{proof}[Proof of Claim \ref{231022claim4_11}]
The proof of Claim \ref{231022claim4_11} is the same as that of Theorem \ref{230831theorem4_6}. Let us give a very brief sketch here. The Hessian of $\widetilde{\phi}_{N_2}$ in the $\xi$ variables is 
\begin{equation}
\begin{bmatrix}
t+ t\partial_{\xi_1}\partial_{\xi_1} Q_3(\xi), & t\partial_{\xi_1}\partial_{\xi_2} Q_3(\xi)\\
t\partial_{\xi_2}\partial_{\xi_1} Q_3(\xi), & t+ t\partial_{\xi_2}\partial_{\xi_2} Q_3(\xi)
\end{bmatrix}
\end{equation}
From this, one can check easily that an analogue of \eqref{231015e4_50} holds, that is, 
\begin{equation}
\int_{|t|\le \epsilon_{\phi}}
\left|\operatorname{det}\left(M+\nabla^2_{\xi} 
\widetilde{\phi}_{N_2}
(
0, t; \xi
)\right)\right| \mathrm{d} t \gtrsim_{\phi} 1,
\end{equation}
uniformly in $M$ and $\xi$. Here in the argument of $\widetilde{\phi}_{N_2}$ we simply set $x=0$ because the Hessian of $\widetilde{\phi}_{N_2}$ in $\xi$ is constant in $x$ anyway. This finishes the proof of Claim \ref{231022claim4_11}. 
\end{proof}

As a consequence of \eqref{231022e4_92} and Claim \ref{231022claim4_11}, we obtain that 
\begin{equation}
\#\{
T\in \T: T\subset S
\}
\lesim_{
\phi, E, \epsilon
}
\mathcal{L}^3(S_{N_2^{-1}}) \kappa^{-2-\epsilon'},
\end{equation}
for every $\epsilon'>0$, where $S_{N_2^{-1}}$ denotes the $N_2^{-1}$ neighborhood of $S$. This finishes the proof of Claim \ref{230902claim4_9}. \end{proof}

Once we prove Claim \ref{230902claim4_9}, the desired estimate in the proposition is immediate, and the argument is exactly the same as in the proof of Corollary \ref{230902corollary4_7}. We apply Claim \ref{230902claim4_9} with $\mc{L}^3(S)=\kappa=r_{2}^{-1/2}$, and see that wave packets that are contained in $S_2$ point in at most 
\begin{equation}
\min
\set{
(N_2)^{k-2} r_2^{-1/2}, \pnorm{
r_2^{-1/2}+ \frac{r_2}{N}
}
} (r_2)^{1+\delta_0}
\end{equation}
 many different directions. This finishes the proof of the proposition. 

\end{proof}

Now we have all the tools to finish the proof of Theorem \ref{230826theorem4_2}. The starting point is again to apply \eqref{230828e4_33} and \eqref{230828e4_34}, which we write down again:
\begin{equation}\label{230902e4_67}
\begin{split}
& \|
T^N f
\|_{\mathrm{BL}_{A_0}^p\left(\B^3_R\right)}
\lesim
R^{C_p \delta_0}
(r_2 D_2)^{\frac{1}{2}(1-\beta_2)}
D_1^{\frac{\beta_2}{p_2}} D_2^{\frac{\beta_2}{p_2}}
\norm{f}_2^{\frac{2}{p_3}} 
\max_{O\in 
\mfn^*_{\ell_0}
} \norm{f_{\iota, O}}_2^{1-\frac{2}{p_3}}\\
& \lesim
R^{C_p \delta_0}
(r_2 D_2)^{\frac{1}{2}(1-\beta_2)}
D_1^{\frac{\beta_2}{p_2}} D_2^{\frac{\beta_2}{p_2}}
\norm{f}_2^{\frac{2}{p_3}} 
r_2^{-\frac{1}{2}(
\frac{1}{2}-\frac{1}{p_3}
)} D_{1}^{-(
\frac{1}{2}-\frac{1}{p_3}
)}
\max_{S_2\in 
\mathfrak{S}_2
}\left\|f^*_{\iota, S_2}\right\|_2^{1-\frac{2}{p_3}}.
\end{split}
\end{equation}
%
To control the right hand side, recall that in \eqref{230828e4_30}, we already proved 
\begin{equation}\label{230902e4_68}
\max_{S_2
\in \mathfrak{S}_2
} \norm{f^*_{\iota, S_2}}_2^2 \lesim R^{C_p \delta_0}
D_2^{-2} \norm{f}_2^2 \lesim R^{C_p \delta_0}
D_2^{-2} \norm{f}_{\infty}^2,
\end{equation}
where in the last step we used H\"older's inequality. Moreover, Proposition \ref{230902prop4_8} says that 
\begin{equation}\label{230902e4_69}
\norm{
f^*_{\iota, S_2}
}_2^2 \lesim
\min\set{
\pnorm{
\frac{N}{r_2}
}^{k-2}
r_2^{-1/2},
\pnorm{
r_2^{-1/2}+\frac{r_2}{N}
}
}
r_2^{\delta_0} 
\norm{f}_{\infty}^2,
\end{equation}
uniformly in $S_2$. Before we continue, let us simplify the right hand side of \eqref{230902e4_69}. We discuss two cases 
\begin{equation}
r_2\le N^{2/3}, \text{ or } r_2\ge N^{2/3}.
\end{equation}
In the former case, we have 
\begin{equation}\label{230902e4_69gg}
\norm{
f^*_{\iota, S_2}
}_2^2 \lesim
r_2^{-1/2}
r_2^{\delta_0} 
\norm{f}_{\infty}^2.
\end{equation}
In the latter case, we have 
\begin{equation}\label{230902e4_69ggg}
\norm{
f^*_{\iota, S_2}
}_2^2 \lesim
\min\set{
\pnorm{
\frac{N}{r_2}
}^{k-2}
r_2^{-1/2},
\frac{r_2}{N}
}
r_2^{\delta_0} 
\norm{f}_{\infty}^2 
\lesim 
r_2^{
-\frac{1}{2(k-1)}
+ \delta_0}
\norm{f}_{\infty}^2.
\end{equation}
Therefore, we always have 
\begin{equation}\label{230902e4_69gghh}
\norm{
f^*_{\iota, S_2}
}_2^2 \lesim
r_2^{
-\frac{1}{2(k-1)}
+ \delta_0}
\norm{f}_{\infty}^2.
\end{equation}
Let $\gamma\in [0, 1]$ to be determined. 
We take a geometric average of \eqref{230902e4_68} and \eqref{230902e4_69gghh}, and obtain 
\begin{equation}\label{230902e4_70}
\max_{S_2
\in \mathfrak{S}_2
} \norm{f^*_{\iota, S_2}}_2^2 \lesim 
 R^{C_p \delta_0}
r_2^{-
\frac{1}{2(k-1)}(1-\gamma)
} 
D_2^{-2\gamma} \norm{f}_{\infty}^2.
\end{equation}
We combine \eqref{230902e4_67} and \eqref{230902e4_70}, and obtain 
\begin{equation}\label{230902e4_71}
\begin{split}
\|
T^N f
\|_{\mathrm{BL}_{A_0}^p\left(\B^3_R\right)}
& \lesim
R^{C_p \delta_0}
(r_2 D_2)^{\frac{1}{2}(1-\beta_2)}
D_1^{\frac{\beta_2}{p_2}} D_2^{\frac{\beta_2}{p_2}}
r_2^{-\frac{1}{2}(
\frac{1}{2}-\frac{1}{p}
)} D_{1}^{-(
\frac{1}{2}-\frac{1}{p}
)}\\
& \times 
r_2^{-
\frac{1}{2(k-1)}(1-\gamma)
(
\frac{1}{2}-\frac{1}{p}
)
} 
D_2^{-2\gamma
(
\frac{1}{2}-\frac{1}{p}
)
}
\norm{f}_2^{\frac{2}{p}} 
\norm{f}_{\infty}^{1-\frac{2}{p}}.
\end{split}
\end{equation}
Here we unify the notation $N=R, p_3=p$. We simplify coefficient on the the right hand side of \eqref{230902e4_71}, and write it as 
\begin{equation}\label{231016e4_89}
\begin{split}
& R^{C_p \delta_0}
(r_2 D_2)^{\frac{1}{2}(1-\beta_2)}
D_1^{\frac{\beta_2}{p_2}} D_2^{\frac{\beta_2}{p_2}}
r_2^{-\frac{1}{2}(
\frac{1}{2}-\frac{1}{p}
)} D_{1}^{-(
\frac{1}{2}-\frac{1}{p}
)}\\
&
\times 
r_2^{-
\frac{1}{2(k-1)}(1-\gamma)
(
\frac{1}{2}-\frac{1}{p}
)
} 
D_2^{-2\gamma
(
\frac{1}{2}-\frac{1}{p}
)
}\\
& \lesim 
R^{C_p \delta_0}
(r_2 D_2)^{\frac{1}{2}(1-\beta_2)}
D_1^{
\frac{\beta_2}{2}
-
2(
\frac{1}{2}-\frac{1}{p}
)
} 
D_2^{
\frac{\beta_2}{2}
-(1+2\gamma)
(
\frac{1}{2}-\frac{1}{p}
)
} r_2^{-
(\frac{k}{2(k-1)}-\frac{\gamma}{2(k-1)})(
\frac{1}{2}-\frac{1}{p}
)}\\
& \lesim 
R^{C_p \delta_0}
(r_2)^{\frac{1}{2}(1-\beta_2)}
D_1^{
\frac{\beta_2}{2}
-
2(
\frac{1}{2}-\frac{1}{p}
)
} 
D_2^{
\frac{1}{2}
-(1+2\gamma)
(
\frac{1}{2}-\frac{1}{p}
)
} r_2^{-
(\frac{k}{2(k-1)}-\frac{\gamma}{2(k-1)})(
\frac{1}{2}-\frac{1}{p}
)}.
\end{split}
\end{equation}
Recall that $\beta_2\in [0, 1]$ is a free parameter we can choose. We give it the constraint that 
\begin{equation}
\frac{\beta_2}{2}
-
2(
\frac{1}{2}-\frac{1}{p}
)
\ge 0.
\end{equation}
Under this constraint, we apply \eqref{230828e4_26cc} and obtain 
\begin{equation}
\begin{split}
\eqref{231016e4_89}
& 
\le 
R^{C_p \delta_0}
r_2^{
\frac12- 
2(
\frac{1}{2}-\frac{1}{p}
)
}
D_2^{
\frac{1}{2}
-(1+2\gamma)
(
\frac{1}{2}-\frac{1}{p}
)
} r_2^{-
(\frac{k}{2(k-1)}-\frac{\gamma}{2(k-1)})(
\frac{1}{2}-\frac{1}{p}
)}.
\end{split}
\end{equation}
By letting the exponents of $r_2$ and $D_2$ be zero, we obtain 
\begin{equation}
p=
3+\frac{k-1}{3k-2}, \ \ \gamma=\frac{3k-2}{4k-3}.
\end{equation}
This finishes the proof of Theorem \ref{230826theorem4_2}.

%
%
%

\section{Proofs of Theorem \ref{230617thm2_1} and Theorem \ref{230323theorem3_4}}

In this section, we will prove Theorem \ref{230617thm2_1} and Theorem \ref{230323theorem3_4}.

\subsection{Preliminaries in Riemannian geometry II}

Let $\epsilon_{\mc{M}}>0$ be a small constant depending on $\mc{M}$. Fix $\epsilon\in (0, \epsilon_{\mc{M}})$. In Section \ref{230903section3_2}, we studied the distance function 
\begin{equation}
\phi_{\epsilon}(x, t; y):=\dist((x, t), (y, \epsilon)).
\end{equation}
We always consider $(x, t)$ in a small neighborhood of the origin, and $(y, \epsilon)$ in a small neighborhood of $(0, \epsilon)$. As a preparation for the proofs of Theorem \ref{230617thm2_1} and Theorem \ref{230323theorem3_4}, we will collect more properties of this function.

For given $\bfx_0=(x_0, t_0)$ and $y_0$, we let $\widetilde{\gamma}:[0, \tilde{L}]\to \mc{M}$ be the arc-length parametrized geodesic satisfying $\widetilde{\gamma}(0)=(x_0, t_0)$ and $\widetilde{\gamma}(\tilde{L})=(y_0, \epsilon)$. Define 
\begin{equation}\label{230719e4_3}
\vec{V}(x_0, t_0; y_0):= \frac{d}{ds}\widetilde{\gamma}(0).
\end{equation}
\begin{lemma}\label{230721claim4_2}
It holds that 
\begin{equation}
\vec{V}(x_0, t_0; y_0)\perp 
(
\partial_{x_1} \partial_{y_i}\phi_{\epsilon}
(x_0, t_0; y_0), \dots, 
\partial_{x_{n-1}} \partial_{y_i}\phi_{\epsilon}
(x_0, t_0; y_0),
\partial_{t} \partial_{y_i}\phi_{\epsilon}
(x_0, t_0; y_0))
\end{equation}
for every $i=1, \dots, n-1$. Here $\perp$ means perpendicular under the Euclidean inner product.  
\end{lemma}
\begin{proof}[Proof of Lemma \ref{230721claim4_2}]
This is a corollary of Lemma \ref{230411lemmaA_1}, from which we know that $\partial_{y_i}\phi_{\epsilon}(x, t; y_0)$ stays constant when $(x, t)$ moves along $\widetilde{\gamma}$. 
\end{proof}

\begin{lemma}\label{230719claim4_3}
For smooth functions $f(x, t)$, it holds that 
\begin{equation}\label{230719e4_5}
(\vec{V}\cdot \nabla_{\bfx})f
\Big|_{
\substack{
\bfx=
(x_0, t_0)\\
y=y_0
}
}= \frac{d}{ds}(f(
\widetilde{\gamma}(s)))\big|_{s=0}
\end{equation}
and 
\begin{equation}\label{230719e4_6}
(\vec{V}\cdot \nabla_{\bfx})^2 f
\Big|_{
\substack{
\bfx=
(x_0, t_0)\\
y=y_0
}
}
= \frac{d^2}{ds^2}(f(
\widetilde{
\gamma}
(s)))\big|_{s=0},
\end{equation}
where if we write 
\begin{equation}
\vec{V}(x, t; y)=(V_1(x, t; y), \dots, V_n(x, t; y)), 
\end{equation}
then 
\begin{equation}
\vec{V}\cdot \nabla_{\bfx}:= 
V_1(x, t; y) \partial_{x_1}+ \dots+
V_{n-1}(x, t; y) \partial_{x_{n-1}}+
V_n(x, t; y) \partial_{t},
\end{equation}
and 
\begin{equation}
(\vec{V}\cdot \nabla_{\bfx})^2:= 
\sum_{i, j} V_i V_j \partial_i \partial_j+ \sum_{i, j} V_j (\partial_j V_i) \partial_i.
\end{equation}
In the last equation, we used the convention that 
\begin{equation}
\partial_i=\partial_{x_i}, \ \ i=1, \dots, n-1; \ \ \partial_n=\partial_t,
\end{equation}
the same as previously used, say in \eqref{230903e3_19}. 
\end{lemma}
\begin{proof}[Proof of Lemma \ref{230719claim4_3}]
We only prove \eqref{230719e4_6}. 
The right hand side of \eqref{230719e4_6} equals 
\begin{equation}\label{230719e4_10}
\frac{d}{ds} \pnorm{
\frac{d}{ds} (f(
\widetilde{
\gamma
}
(s)))
}\Big|_{s=0}
=
\frac{d}{ds}
\pnorm{
\sum_i \partial_i f \frac{d 
\widetilde{\gamma}_i(s)}{ds}
}\Big|_{s=0}
\end{equation}
where 
\begin{equation}
\widetilde{
\gamma}
(s)=(
\widetilde{\gamma}_1(s), \dots, \widetilde{\gamma}_n(s)).
\end{equation}
By \eqref{230719e4_3}, we have 
\begin{equation}
\eqref{230719e4_10}= 
\frac{d}{ds}
\pnorm{
\sum_i (\partial_i f)(
\widetilde{\gamma}(s)) 
V_i(
\widetilde{\gamma}
(s); y_0)
}\Big|_{s=0},
\end{equation}
which, by the chain rule,  equals the left side of \eqref{230719e4_6}. This finishes the proof of the claim. 
\end{proof}

Before we proceed, let us state a corollary of Lemma \ref{230719claim4_3}. Recall from Definition \ref{230903definition1_4} that the phase function $\phi_{\epsilon}(x, t; y)$ is said to satisfy Bourgain's condition at $(x_0, t_0; y_0)$ if 
\begin{equation}\label{230903e5_18}
\left(\left(G_0 \cdot \nabla_{\mathbf{x}}\right)^2 \nabla_{y}^2 \phi_{\epsilon}
\right)\left(\mathbf{x}_0 ; y_0\right) \text { is a multiple of }\left(\left(G_0 \cdot \nabla_{\mathbf{x}}\right) \nabla_{y}^2 \phi_{\epsilon}
\right)\left(\mathbf{x}_0 ; y_0\right),
\end{equation}
where 
\begin{equation}\label{230903e5_19uu}
G_0(\bfx; y):=\partial_{y_1} \nabla_{\bfx} \phi_{\epsilon}(\bfx; y)\wedge \dots\wedge 
\partial_{y_{n-1}} \nabla_{\bfx} \phi_{\epsilon}
(\bfx; y),
\end{equation}
and $\nabla_y^2$ is the standard Euclidean Hessian. By Lemma 2.3 in \cite{GWZ22}, Bourgain's condition is invariant under multiplying $G_0$ by a non-zero scalar function which is allowed to depend on $x_0, y_0$. This, combined with Lemma \ref{230719claim4_3}, gives the following equivalent form of Bourgain's condition. 

\begin{corollary}\label{230903coro5_4}
Bourgain's condition holds for the phase $\phi_{\epsilon}$ at the point $(x_0, t_0; y_0)$ if and only if 
\begin{equation}\label{230720e4_13}
\frac{d^2}{ds^2} (\partial_{y_i} \partial_{y_j} \phi_{\epsilon}(
\widetilde{
\gamma
}(s); 
y_0))\big|_{s=0}= C(\bfx_0; y_0) \frac{d}{ds} (\partial_{y_i} \partial_{y_j} \phi_{\epsilon}
(\widetilde{
\gamma
}
(s); y_0))\big|_{s=0}
\end{equation}
where $C(\bfx_0; y_0)\in \R$ is allowed to depend on $\bfx_0$ and $y_0$, but not on $1\le i, j\le n-1$.
\end{corollary}

\subsection{Relation between Carleson-Sj\"olin on manifolds and H\"ormander's problems}

In this subsection, we will prove item a) of Theorem \ref{230617thm2_1}. This is well-known. Moreover, we also know that Carleson-Sj\"olin on manifolds is elliptic, in the sense that it satisfies (H1) (H2), and if we write the relevant phase function in the normal form as in \eqref{230324e1_4}, then the matrix $A$ is always positive definite.  \\

The proof is very short, and therefore we will include the proof here. Recall that a reduced Carleson-Sj\"olin operator is given by 
\begin{equation}
R_N^{(\mc{M}, \mc{M}')} f(\bfx)=\int_{\mc{M}'}
e^{i N\dist(\bfx, \bfy)} a(\bfx) f(\bfy) d\mc{H}^{n-1}(\bfy),
\end{equation}
where $\mc{M}'$ is a submanifold of $\mc{M}$. Let us without loss of generality assume that $\mc{M}'$ is given by $\{(y, \epsilon): y\in \R^2\}$ for some $\epsilon\neq 0$, and that the smooth amplitude function $a(\bfx)$ is supported in a small neighborhood of the origin. Our oscillatory integral becomes 
\begin{equation}
\int_{\R^{n-1}}
e^{
iN\phi_{\epsilon}(x, t; y) 
}
a(\bfx) f(y)dy,
\end{equation}
where $\bfx=(x, t)$ and 
\begin{equation}
\phi_{\epsilon}(x, t; y):=\dist((x, t), (y, \epsilon)). 
\end{equation}
We need to show that 
\begin{equation}\label{230903e5_23}
\rank \nabla_{x}\nabla_y \phi_{\epsilon}(x_0, t_0; y_0)=n-1,
\end{equation}
and that 
\begin{equation}\label{230903e5_24}
\left.\operatorname{det} \nabla_{y}^2\left\langle\nabla_{\mathbf{x}} \phi_{\epsilon}
(\mathbf{x}_0 ; y), G_0\left(\mathbf{x} ; y_0\right)\right\rangle\right|_{y=y_0} \neq 0,
\end{equation}
for all $(x_0, t_0)$ near $(0, 0)$ and $y_0$ near $0$, 
where $G_0$ is defined in \eqref{230903e5_19uu}. This follows immediately from continuity and the fact that \eqref{230903e5_23} and \eqref{230903e5_24} hold for the Euclidean distance function. 

We remark here that by applying the tools from Riemannian geometry we introduced above, one can avoid the continuity argument and make the choice of $\epsilon$ more explicit. To keep our presentation short, we will not pursue this direction.

\subsection{Relation between Carleson-Sj\"olin on manifolds and Nikodym on manifolds}\label{230904subsection5_3}

The goal of this subsection is to prove item b) of Theorem \ref{230617thm2_1}. As mentioned below Theorem \ref{230617thm2_1}, the proof is essentially the same as the proof of Theorem \ref{230706thm1_8}; the only extra input is Lemma \ref{230411lemmaA_1}. 

Let us be more precise. We continue to use the notation 
\begin{equation}
\phi_{\epsilon}(x, t; y):=\dist((x, t), (y, \epsilon)). 
\end{equation}
For the given phase function $\phi_{\epsilon}$, recall from Definition \ref{230617defi1_6} that curved tubes in the curved Kakeya problem associated to $\phi_{\epsilon}$ are given by 
\begin{equation}
T_y^{\delta, (\phi_{\epsilon})}(\bfx):=\left\{\bfx' \in \mathbb{R}^n\cap \B^n_{2\epsilon_{\phi
_{\epsilon}
}}:\left|\nabla_y \phi_{\epsilon}(\bfx'; y)-\nabla_y \phi_{\epsilon}(\bfx; y)\right|<\delta\right\}.
\end{equation}
However, Lemma \ref{230411lemmaA_1} says precisely that this is the $\delta$-neighborhood of a geodesic passing through $(y, \epsilon)$. The Nikodym maximal function $\mc{N}_{\delta, \lambda}f(y, \epsilon)$ is essentially the curved  Kakeya maximal function $\mc{K}_{\delta}^{(\phi_{\epsilon})}f(y)$. Here we need to assume that $\lambda<1$ just to avoid the singularities of the distance function along the diagonals. We refer the rest of the details to Wisewell's thesis \cite[page 26]{Wis03}. \\

Before finishing this subsection, let us make a remark that the Nikodym maximal function bound we can deduce here is stronger than what we need for $\mc{N}_{\delta, \lambda}$. More precisely, the Nikodym maximal function bound concerns the $L^p$ norm on $\B^n_{
\epsilon_{\mc{M}}/2
},$ which is an integral over an $n$-dimensional object. However, the bound we can deduce from the argument above concerns the $L^p$ norm on each hyperplane $\mc{M}'$, which is of course much stronger because of Fubini's theorem. This also explains one key difference between Nikodym maximal operators and curved Kakeya maximal operators.

\subsection{Proof of Theorem \ref{230323theorem3_4}: Part a)}

Let us without loss of generality assume that $\bfx$ is near $0\in \R^n$ and $\bfy$ is near $(0, \epsilon)$ where $0\in \R^{n-1}$ and $\epsilon>0$ is a fixed small number. What we need to  prove is 
\begin{equation}\label{230903e5_27}
\Norm{
\int_{\R^{n-1}}
e^{iN \phi_{\epsilon}
(x, t; y)} a(\bfx; y)g(y)dy
}_{L^p(\R^n)} \lesim_{\mc{M}, a, \epsilon'} 
N^{-\frac{n}{p}+ \epsilon'}
\norm{g}_{L^p(\R^{n-1})},
\end{equation}
for every $\epsilon'>0$, where 
\begin{equation}\label{230415e4_2}
\phi_{\epsilon}
(x, t; y):=\dist((x, t), (y, \epsilon)),
\end{equation}
and $a(\bfx; y)$ is a smooth bump function with $\bfx=(x, t)$ supported near $0\in \R^n$ and $y$ supported near $0\in \R^{n-1}$. To prove \eqref{230903e5_27}, by \cite[Theorem 1.3]{GWZ22}, it suffices to prove 
\begin{lemma}\label{230323lemma5_1}
The phase function $\phi_{\epsilon}(x, t; y)$ satisfies Bourgain's condition everywhere if $\mcm$ has a constant sectional curvature. 
\end{lemma}
\begin{proof}[Proof of Lemma \ref{230323lemma5_1}] Take two points $\bfx_0=(x_0, t_0)$ and $(y_0, \epsilon)$ on the manifold, and we would like to check Bourgain's condition. By Corollary \ref{230903coro5_4}, it suffices to prove that 
\begin{equation}\label{230903e5_29}
\frac{d^2}{ds^2} (\partial_{y_i} \partial_{y_j} \phi_{\epsilon}(
\widetilde{
\gamma
}(s); 
y_0))\big|_{s=0}= C(\bfx_0; y_0) \frac{d}{ds} (\partial_{y_i} \partial_{y_j} \phi_{\epsilon}
(\widetilde{
\gamma
}
(s); y_0))\big|_{s=0}
\end{equation}
where $\widetilde{\gamma}$ is the arc-length parametrized geodesic connecting $(x_0, t_0)$ and $(y_0, \epsilon)$, and  $C(\bfx_0; y_0)\in \R$ is allowed to depend on $\bfx_0$ and $y_0$, but not on $1\le i, j\le n-1$.\\

Denote 
\begin{equation}
\Phi(x, t; y, \tau):=
\dist(
(x, t), (y, \tau)
).
\end{equation}
We would like to connect $\partial_{y_i} \partial_{y_j} \phi_{\epsilon}$ with the covariant Hessian of $\Phi$. Without loss of generality, assume that $(x_0, t_0)=(0, 0)$ and $y_0=0$. Note that in \cite[Corollary 2.2]{GWZ22}, it is proved that Bourgain's condition is independent of the choice of coordinates. Therefore, we can for the sake of simplicity assume that we are in the same setting as in Subsection \ref{230903section3_2}, that is, we are in the Fermi coordinate based on the geodesic
\begin{equation}
\gamma(s)=(0, s), \ \ \forall s\in [0, \epsilon],
\end{equation}
and if we define
\begin{equation}
E_i(s):= \frac{\partial}{\partial y_i}\in T_{\gamma(s)}\mc{M}, \ \ i=1, \dots, n-1, \ \ E_n(s):= \frac{\partial}{\partial \tau}\in T_{\gamma(s)}\mc{M},
\end{equation}
then $\{E_i(s)\}_{i=1}^n$ forms an orthonormal basis for $T_{\gamma(s)}\mc{M}$. We will need Claim \ref{230820claim3_3} and Lemma \ref{230902lemma3_5gg} for manifolds of a general dimension $n$, not just $3$; however, this requires only notational changes. 

 Recall the notation from \eqref{230903e3_20xx}--\eqref{230903e3_23xx}.  We apply Claim \ref{230820claim3_1} and obtain 
\begin{equation}\label{230903e5_33}
\frac{\partial}{\partial y_i} \frac{\partial}{\partial y_j} \phi_{0, \epsilon}\Big|_{\substack{
(x, t)=\gamma(s),\\
 y=0
 }}= \pnorm{\hessian \Phi_0}\Big|_{
 \substack{
 (x, t)=\gamma(s), \\
 (y, \tau)=(0, \epsilon)}
 }\pnorm{
\frac{\partial}{\partial y_i}, \frac{\partial}{\partial y_j}
}
\end{equation}
for every $s\in [0, \epsilon)$, $1\le i, j\le n-1$. Here $\hessian$ is the covariant Hessian in the $(y, \tau)$ variables. In \eqref{230903e5_33}, we connected $\partial_{y_i} \partial_{y_j} \phi_{0, \epsilon}$ with the covariant Hessian of $\Phi_0$, but not $\partial_{y_i} \partial_{y_j} \phi_{\epsilon}$ with the covariant Hessian of $\Phi$. However, by taking derivatives $\nabla_{\dot{\gamma}}$ on both sides of \eqref{230903e5_33}, we  immediately obtain 
\begin{equation}
\frac{\partial^{\iota}}{\partial s^{\iota}} 
\frac{\partial}{\partial y_i} \frac{\partial}{\partial y_j} \phi_{\epsilon}(\gamma(s); 0)= 
\frac{\partial^{\iota}}{\partial s^{\iota}} 
\pnorm{\hessian \Phi(\gamma(s); (0, \epsilon))}
\pnorm{
\frac{\partial}{\partial y_i}, \frac{\partial}{\partial y_j}
}
\end{equation}
for every $\iota\in \N, \iota\ge 1$ and every $1\le i, j\le n-1, s\in [0, \epsilon)$. Therefore, to check Bourgain's condition, it is equivalent to check that 
\begin{equation}\label{230903e5_35}
\frac{\partial^{2}}{\partial s^{2}} 
\pnorm{\hessian \Phi(\gamma(s); (0, \epsilon))}
\pnorm{
\frac{\partial}{\partial y_i}, \frac{\partial}{\partial y_j}
}\Big|_{s=0}=C \frac{\partial}{\partial s} 
\pnorm{\hessian \Phi(\gamma(s); (0, \epsilon))}
\pnorm{
\frac{\partial}{\partial y_i}, \frac{\partial}{\partial y_j}
}\Big|_{s=0},
\end{equation}
for every $1\le i, j\le n-1$. Recall the second equation in Claim \ref{230820claim3_1} that 
\begin{equation}
\pnorm{\hessian \Phi_0}\Big|_{
 \substack{
 (x, t)=\gamma(s), \\
 (y, \tau)=(0, \epsilon)}
 }\pnorm{
Y, \frac{\partial}{\partial \tau}
}=0, \ \ \forall Y\in T_{(0, \epsilon)}\mc{M},
\end{equation}
for every $s\in [0, \epsilon)$, $1\le i\le n-1$. If we identify $\partial/\partial y_n$ with $\partial/\partial \tau$, then Bourgain's condition is equivalent to saying that \eqref{230903e5_35} holds for $1\le i, j\le n$, that is, 
\begin{equation}
\nabla_{\dot{\gamma}}^2 \ \hessian\ \Phi\big|_{
\substack{
(x, t)=(0, 0)\\
(y, \tau)=(0, \epsilon)}
}
= C \ \nabla_{\dot{\gamma}} \ \hessian\ \Phi\big|_{
\substack{
(x, t)=(0, 0)\\
(y, \tau)=(0, \epsilon)}
}
\end{equation}
for some $C\in \R$. \\

To prove this, the first few steps are the same as in those in Subsection \ref{230903section3_2}.  Define Jacobi fields (previously defined in \eqref{230820e3_19})
\begin{equation}\label{230904e5_38}
\begin{split}
& \nabla^2_{\dot{\gamma}(s')} X_j(s, s')+R(X_j(s, s'), \dot{\gamma}(s'))\dot{\gamma}(s')=0, \\
& X_j(s, s)=0, \ \ X_j(s, \epsilon)=E_j(\epsilon).
\end{split}
\end{equation}
By Claim \ref{230820claim3_3} and Lemma \ref{230902lemma3_5gg}, we obtain (see \eqref{230820e3_22})
\begin{equation}\label{230904e5_39}
\hessian \Phi\big|_{\gamma(s)}
\pnorm{
E_i(\epsilon), E_j(\epsilon)
}
=
g\pnorm{
\nabla_{\dot{\gamma}(s')}X_i(s, s'), X_j(s, s')
}\Big|_{s'=\epsilon}.
\end{equation}
Define $a_{ji}(s, s')$ by 
\begin{equation}
X_j(s, s')= a_{j1}(s, s') E_1(s')+ a_{j2}(s, s')E_2(s'), \ \ j=1, \dots, n-1.
\end{equation}
Denote 
\begin{equation}
A(s, s')=[a_{ij}(s, s')]_{1\le i, j\le n-1}.
\end{equation}
Then the Jacobi fields \eqref{230904e5_38} can be written as
\begin{equation}\label{230904e5_42}
\begin{split}
& \partial_{s'}^2 A(s, s')+ A(s, s') R(s')=0, \\
& A(s, s)=0, \ A(s, \epsilon)=I_{(n-1)\times (n-1)},
\end{split}
\end{equation} 
where 
\begin{equation}
R(s'):=[R_{innj}(s')]_{1\le i, j\le n-1}
\end{equation}
and 
\begin{equation}
R_{ijkl}(s'):=g(
R(E_i(s'), E_j(s'))E_k(s'), E_l(s')
).
\end{equation}
Moreover, as calculated in \eqref{230821e3_29}, we have 
\begin{equation}
\eqref{230904e5_39}=
\partial_{s'}
a_{ij}\big|_{s'=\epsilon}
\end{equation}
Recall that Bourgain's condition is equivalent to \eqref{230903e5_35}, and now it is further equivalent to 
\begin{equation}\label{230904e5_46}
\partial_s^2 
\partial_{s'}
A\big|_{s=0, s'=\epsilon}
= C\ 
\partial_s
\partial_{s'}
A\big|_{s=0, s'=\epsilon},
\end{equation}
for some $C\in \R$. \\

Let $\kappa\in \R$ be the sectional curvature of $\mc{M}$. As the sectional curvature of $\mc{M}$ is constant, we have (see for instance \cite[page 84]{Pet16})
\begin{equation}
R(w, v)v=
\kappa(
w-g(w, v)v
),
\end{equation}
where $g(v, v)=1$. Therefore $R(s')= \kappa I_{(n-1)\times (n-1)}$. We solve \eqref{230904e5_42} explicitly, and obtain 
\begin{equation}
A(s, s')=
A_0^{-1}(\epsilon-s)
A_0(s'-s),
\end{equation}
where 
\begin{equation}
A_0(r):=
\begin{cases}
\frac{
\sin \sqrt{\kappa}r
}{\sqrt{\kappa}} I_{(n-1)\times (n-1)}, & \text{ if } \kappa>0,\\
r I_{(n-1)\times (n-1)}, & \text{ if } \kappa=0,\\
\frac{
\sinh \sqrt{-\kappa}r
}{\sqrt{-\kappa}} I_{(n-1)\times (n-1)}, & \text{ if } \kappa<0.\\
\end{cases}
\end{equation}
Compute 
\begin{equation}
\partial_{s'}A\big|_{s'=\epsilon}
=
\begin{cases}
\frac{
\sqrt{\kappa}\cos (\sqrt{\kappa}
(\epsilon-s))
}{
\sin(\sqrt{\kappa}(\epsilon-s))
} I_{(n-1)\times (n-1)}, & \text{ if } \kappa>0,\\
\frac{1}{\epsilon-s} I_{(n-1)\times (n-1)}, & \text{ if } \kappa=0,\\
\frac{
\sqrt{-\kappa}
\cosh (\sqrt{-\kappa}(\epsilon-s)
)
}{
\sinh (
\sqrt{-\kappa}(\epsilon-s)
)
} I_{(n-1)\times (n-1)}, & \text{ if } \kappa<0.\\
\end{cases}
\end{equation}
From this, one can see that \eqref{230904e5_46} holds. This finishes the proof of the lemma. \end{proof}

\subsection{Proof of Theorem \ref{230323theorem3_4}: Part b)}

By \cite[Theorem 1.1]{GWZ22}, it suffices to prove that Bourgain's condition fails at least at one point, if the sectional curvature of $\mc{M}$ is not constant. We will argue by contradiction, assume that Bourgain's condition holds at every point, and then derive that the sectional curvature must be constant.  \\

Before we start the proof, let us mention that in the subsection we will need Claim \ref{230822claim3_5} for manifolds of a general dimension $n$; however in the proof of this claim, the dimension parameter actually does not appear explicitly, and the same argument works for all dimensions $n$.   

Recall the Jacobi fields \eqref{230904e5_42}, and that Bourgain's condition is equivalent to (see \eqref{230904e5_46})
\begin{equation}\label{230904e5_46xx}
\partial_s^2 
\partial_{s'}
A\big|_{s=0, s'=\epsilon}
= C\ 
\partial_s
\partial_{s'}
A\big|_{s=0, s'=\epsilon},
\end{equation}
for some $C\in \R$. To study \eqref{230904e5_42}, we follow the strategy in the proof of Claim \ref{230822claim3_5}, and introduce the following two systems of equations (see \eqref{230823e3_81} and \eqref{230823e3_82})
\begin{equation}\label{230904e5_52}
\begin{split}
& B''_1(s')+ B_1(s')R(s')=0_{(n-1)\times (n-1)},\\
& B_1(0)=I_{(n-1)\times (n-1)}, \ \ B'_1(0)=0_{(n-1)\times (n-1)},
\end{split}
\end{equation}
and 
\begin{equation}\label{230904e5_53}
\begin{split}
& B''_2(s')+ B_2(s')R(s')=0_{(n-1)\times (n-1)},\\
& B_2(0)=0_{(n-1)\times (n-1)}, \ \ B'_2(0)=I_{(n-1)\times (n-1)}.
\end{split}
\end{equation}
Below we abbreviate $0_{(n-1)\times (n-1)}$ to $0$, and $I_{(n-1)\times (n-1)}$ to $I$. Recall that in \eqref{230904e3_104}, we obtained that 
\begin{equation}\label{230904e3_104mm}
\partial^2_s\partial_{s'} A(0, \epsilon)=
2 B_2^{-1}(\epsilon) B_1(\epsilon) \partial_s\partial_{s'} A(0, \epsilon).
\end{equation}
As we assume that Bourgain's condition holds everywhere, we can find a scalar function $C(\epsilon)$ such that 
\begin{equation}\label{230904e5_55}
B_1(\epsilon)
=C(\epsilon)B_2(\epsilon).
\end{equation}
By taking the second order derivative in \eqref{230904e5_55}, we obtain 
\begin{equation}
B''_1(\epsilon)
=C''(\epsilon)B_2(\epsilon)
+
2C'(\epsilon)B'_2(\epsilon)+ C(\epsilon)B''_2(\epsilon).
\end{equation}
This, combined with the first equation in \eqref{230904e5_52}, implies that 
\begin{equation}\label{230904e5_57}
C''(\epsilon) B_2(\epsilon)
+ 2 C'(\epsilon) B'_2(\epsilon)=0.
\end{equation}
By taking a further derivative in $\epsilon$, we obtain 
\begin{equation}
C'''(\epsilon) B_2(\epsilon)
+ 3 C''(\epsilon) B'_2(\epsilon)
+
2 C'(\epsilon) B''_2(\epsilon)=0.
\end{equation}
By \eqref{230904e5_57} and the first equation in \eqref{230904e5_53}, we obtain 
\begin{equation}\label{230904e5_59}
C'(\epsilon)
C'''(\epsilon) B_2(\epsilon)
- \frac32 (C''(\epsilon))^2 B_2(\epsilon)
-
2 (C'(\epsilon))^2 R(\epsilon)B_2(\epsilon)=0.
\end{equation}
Because of the initial condition in \eqref{230904e5_53}, we see that $B_2(\epsilon)$ is always invertible, whenever $\epsilon>0$ is taken to be small enough. Consequently, \eqref{230904e5_59} is equivalent to 
\begin{equation}\label{230904e5_60}
C'(\epsilon)
C'''(\epsilon) I
- \frac32 (C''(\epsilon))^2 I
-
2 (C'(\epsilon))^2 R(\epsilon)=0.
\end{equation}
Recall in \eqref{230831e3_95}, we obtained 
\begin{equation}
\begin{split}
B_2^{-1}(\epsilon) B_1(\epsilon)
=
\frac{1}{\epsilon}
\pnorm{
I
-
\frac{R_0}{3} \epsilon^2
-
\frac{R'_0}{12} \epsilon^3+ O(\epsilon^4)
},
\end{split}
\end{equation}
where $R_0:=R(0), R'_0:=R'(0)$ and the implicit constant in $O(\epsilon)$ depends only on the manifild. This, combined with the definition of $C(\epsilon)$ in \eqref{230904e5_55}, implies that $C'(\epsilon)\neq 0$ whenever $\epsilon\neq 0$ is taken small enough. Therefore, \eqref{230904e5_59} can be further written as 
\begin{equation}
R(\epsilon)=
\pnorm{
\frac{
C'''(\epsilon)}{
2C'(\epsilon)
}
- \frac{3 (C''(\epsilon))^2}{4 (C'(\epsilon))^2}  } I
=:
 \kappa(\epsilon) I.
\end{equation}
For two given tangent vectors $E_i(\epsilon), E_n(\epsilon)$ with $1\le i\le n-1$, the sectional curvature at $\gamma(\epsilon)$ associated to these two tangent vectors is 
\begin{equation}
R(E_i(\epsilon), E_n(\epsilon), E_n(\epsilon), E_i(\epsilon))=
\kappa(\epsilon),
\end{equation}
which is independent of $1\le i\le n-1$. \\

Recall that $\{E_i(\epsilon)\}_{i=1}^n$ forms an orthonormal basis for $T_{\gamma(\epsilon)}\mc{M}$. So far we have proven that the sectional curvature at $\gamma(\epsilon)$ associated to the pair of vectors $E_i(\epsilon), E_n(\epsilon)$ with $1\le i\le n-1$ is independent of $i$. Let us write it as $\kappa(\gamma(\epsilon), n)$. The index $n$ plays a special role here because the geodesic $\gamma$ is chosen such that $\dot{\gamma}(\epsilon)=E_n(\epsilon)$. Now we consider all possible geodesics passing through the point $\gamma(\epsilon)$. Similarly, we will obtain that 
\begin{equation}
R(E_i(\epsilon), E_{n'}(\epsilon), E_{n'}(\epsilon), E_i(\epsilon))=
\kappa(\gamma(\epsilon), n'),
\end{equation}
where $\kappa(\gamma(\epsilon), n')\in \R$ is some constant that is independent of $i\neq n'$. By basic symmetries of Riemannian tensors, we can conclude that 
\begin{equation}
R(E_i(\epsilon), E_{j}(\epsilon), E_{j}(\epsilon), E_i(\epsilon))
\end{equation}
is the same for all choices of $i\neq j$. By Schur's lemma, this constant must also be independent of the choice of the point $\gamma(\epsilon)$, that is, $\mc{M}$ must have constant sectional curvature.

\subsection{Proof of Theorem \ref{230323theorem3_4}: Part c)}

To prove the $\lambda$-Nikodym maximal function bound in part c) of Theorem \ref{230323theorem3_4}, one just need to repeat the argument in Hickman, Roger and Zhang \cite{HRZ22}, similarly to how one proves the bound \eqref{230815e2_4} for  curved Kakeya maximal functions associated to phase functions satisfying Bourgain's condition.  Lemma \ref{230323lemma5_1} guarantees that the relevant phase functions satisfy Bourgain's condition, which further guarantees that we have (strong) polynomial Wolff axioms as required by \cite{HRZ22}, and Subsection \ref{230904subsection5_3} explains the connection between Nikodym maximal functions and curved Kakeya maximal functions. 

The Minkowski dimension bound in part c) of Theorem \ref{230323theorem3_4} follows from standard argument connecting Nikodym maximal function bounds and Minkowski dimensions of Nikodym sets, see for instance \cite[Corollary 2.2]{Sog99}.

%
%
%
%
%
%
%

\appendix

\section{More connections between curved Kakeya problems and Nikodym problems on manifolds}

In this section, we will show that not every curved Kakeya problem can be viewed as a Nikodym problem on manifolds. Let us be more precise. We will find a phase function $\phi(x, t; \xi): \R^2\times \R\times \R^2\to \R$ satisfying H\"ormander's non-degeneracy condition, and show that  no matter how we pick Riemannian metric tensor $\{g_{ij}(x, t)\}_{1\le i, j\le 3}$ on $\R^2\times \R$, the curves 
\begin{equation}
\{(x, t): \nabla_{\xi} \phi(x, t; \xi)=\omega\}
\end{equation}
where $\xi, \omega$ are parameters, are never geodesics. \\

Let us take a phase function $\phi(x, t; \xi)$ satisfying H\"ormander's non-degeneracy condition that will be picked later. Let us assume that we can find a Riemannian metric tensor $\{g_{ij}(x, t)\}_{1\le i, j\le 3}$ such that 
\begin{equation}
\{(x, t)\in \B^3_{2\epsilon_{\phi}}: \nabla_{\xi} \phi(x, t; \xi)=\omega\}
\end{equation}
is a geodesic for every $\xi\in \B^2_{\epsilon_{\phi}}$ and $\omega\in \B^2_{\epsilon_{\phi}}$. Here $\epsilon_{\phi}>0$ is a small real number that depends on $\phi$. 

Fix $(x_0, t_0)$. We write a geodesic passing through this point as $(X_{\xi}(t), t)$ where 
\begin{equation}\label{220927e1_3}
    (\nabla_{\xi}\phi)(X_{\xi}(t), t; \xi)=(\nabla_{\xi}\phi)(x_0, t_0; \xi), \ \ X_{\xi}(t_0)=x_0.
\end{equation}
We compute $\frac{\partial}{\partial t} X_{\xi}(t)$, and obtain the tangent vector at the starting point. To do so, we take the derivative of \eqref{220927e1_3} in $t$, and obtain 
\begin{equation}\label{220929e1_6}
    \nabla_x \nabla_{\xi}\phi(X_{\xi}(t), t; \xi)\cdot \frac{\partial}{\partial t} X_{\xi}(t)+\partial_t \nabla_{\xi}\phi(X_{\xi}(t), t; \xi)=0.
\end{equation}
Therefore, 
\begin{equation}
    \frac{\partial X_{\xi}(t)}{\partial t}=(\nabla_x\nabla_{\xi}\phi)^{-1}\cdot \partial_t \nabla_{\xi}\phi,
\end{equation}
and the tangent vector of the geodesic \eqref{220927e1_3} is parallel to 
\begin{equation}
    \begin{bmatrix}
    (\nabla_x\nabla_{\xi}\phi)^{-1}\cdot \partial_t \nabla_{\xi}\phi\\
    1
    \end{bmatrix}
\end{equation}
We take a further derivative in $t$ on both sides of \eqref{220929e1_6}, and obtain 
\begin{equation}\label{220929e1_10}
    \begin{split}
        & \nabla_x\nabla_{\xi}\phi\cdot \frac{\partial^2}{\partial t^2} X_{\xi}(t)+
    \pnorm{
    \frac{\partial}{\partial t} X_{\xi}(t)
    }^T \cdot 
    \nabla^2_x\nabla_{\xi}\phi \cdot \frac{\partial}{\partial t} X_{\xi}(t)\\
 &   +
    2 \partial_t \nabla_x\nabla_{\xi}\phi\cdot \frac{\partial}{\partial t} X_{\xi}(t)+\partial^2_t \nabla_{\xi}\phi=0. 
    \end{split} 
\end{equation}
By moving terms, we obtain 
\begin{equation}\label{230222e2_21}
    \begin{split}
        & \frac{\partial^2}{\partial t^2} X_{\xi}(t)=2(\nabla_x \nabla_{\xi} \phi)^{-1} \partial_t \nabla_x \nabla_{\xi}\phi \cdot (\nabla_x \nabla_{\xi}\phi)^{-1} \partial_t \nabla_{\xi}\phi+(\nabla_x \nabla_{\xi} \phi)^{-1} \partial^2_t \nabla_{\xi}\phi\\
        &+ (\nabla_x\nabla_{\xi}\phi)^{-1}\cdot 
        \pnorm{
         \pnorm{
    \frac{\partial}{\partial t} X_{\xi}(t)
    }^T \cdot 
    \nabla^2_x\partial_{\xi_{\iota}}\phi \cdot \frac{\partial}{\partial t} X_{\xi}(t)
        }_{1\le \iota\le n-1}^T
    \end{split}
\end{equation}
To simplify notations, we will write $x=(x_1, x_2)$ and treat $t$ as the third spatial variable; we will use $\Gamma^{\mu}_{\alpha\beta}$, $1\le \alpha, \beta, \mu\le 3$ for Christopher symbols. If we parametrize geodesics by using the third spatial variable $t$, that is, if 
\begin{equation}
(x^1(t), x^2(t), x^3(t))
\end{equation}
is a geodesic with $x^3(t)=t$, then 
\begin{equation}\label{230222e2_25}
    \frac{d^2 x^\mu}{d t^2}=-
    \sum_{
    \alpha, \beta
    }\Gamma_{\alpha \beta}^\mu \frac{d x^\alpha}{d t} \frac{d x^\beta}{d t}+
    \sum_{
    \alpha, \beta
    }
    \Gamma_{\alpha \beta}^3 \frac{d x^\alpha}{d t} \frac{d x^\beta}{d t} \frac{d x^\mu}{d t},
\end{equation}
where $\mu=1, 2, 3$. We therefore have \begin{equation}\label{220930e1_14}
    \begin{split}
        \begin{bmatrix}
    \frac{\partial^2 X_{\xi}(t)}{\partial t^2}\\
    0
    \end{bmatrix}
    & +
    \begin{bmatrix}
    \Big(
    \Big(
    \frac{\partial X_{\xi}(t)}{\partial t}
    \Big)^T, 
    1\Big)\cdot  [\Gamma^1_{bc}]_{1\le b, c\le 3}\cdot  \Big(
    \Big(
    \frac{\partial X_{\xi}(t)}{\partial t}
    \Big)^T, 1\Big)^T\\
    \Big(
    \Big(
    \frac{\partial X_{\xi}(t)}{\partial t}
    \Big)^T
    , 1\Big)\cdot  [\Gamma^2_{bc}]_{1\le b, c\le 3}\cdot  \Big(
    \Big(
    \frac{\partial X_{\xi}(t)}{\partial t}
    \Big)^T, 1\Big)^T\\
    0
    \end{bmatrix}\\
    &-
    \Big(
    \Big(
    \frac{\partial X_{\xi}(t)}{\partial t}
    \Big)^T, 1\Big)\cdot  [\Gamma^3_{bc}]_{1\le b, c\le 3}\cdot  \Big(
    \Big(
    \frac{\partial X_{\xi}(t)}{\partial t}
    \Big)^T
    , 1\Big)^T
    \begin{bmatrix}
    \frac{\partial X_{\xi}(t)}{\partial t}
    \\
    0
    \end{bmatrix}
    =0, \ \ \forall t.
    \end{split}
\end{equation}
Recall how to pass from metrics to Christopher symbols: 
\begin{equation}\label{231023a_12}
    \partial_i g_{j k}=\sum_{l}
    \Gamma_{i j}^l g_{l k}+
    \sum_{
    l
    }
    \Gamma_{i k}^l g_{j l}, 
\end{equation}
or 
\begin{equation}
    \Gamma^i_{k l}=\frac{1}{2} \sum_m g^{i m}\left(\frac{\partial g_{m k}}{\partial x^l}+\frac{\partial g_{m l}}{\partial x^k}-\frac{\partial g_{k l}}{\partial x^m}\right).
\end{equation}
The first identity \eqref{231023a_12} can be written as 
\begin{equation}
    \partial_i \bfg=
    \begin{bmatrix}
    \Gamma^1_{i1}, \ & \Gamma^2_{i1}, \ & \Gamma^3_{i1}\\
    \Gamma^1_{i2}, \ & \Gamma^2_{i2}, \ & \Gamma^3_{i2}\\
    \Gamma^1_{i3}, \ & \Gamma^2_{i3}, \ & \Gamma^3_{i3}
    \end{bmatrix}\cdot \bfg
    +\bfg\cdot 
    \begin{bmatrix}
    \Gamma^1_{i1}, \ & \Gamma^1_{i2}, \ & \Gamma^1_{i3}\\
    \Gamma^2_{i1}, \ & \Gamma^2_{i2}, \ & \Gamma^2_{i3}\\
    \Gamma^3_{i1}, \ & \Gamma^3_{i2}, \ & \Gamma^3_{i3}
    \end{bmatrix}
\end{equation}
where 
\begin{equation}
    \bfg:=(g_{ij})_{1\le i, j\le 3}.
\end{equation}

\bigskip 

We now start constructing the phase function $\phi$. We start with the phase function
\begin{equation}
    \phi_0(x, t; \xi)=\inn{x}{\xi}+t\xi_1\xi_2+\frac{1}{2} t^2 \xi_1^2,
\end{equation}
which was constructed by Bourgain \cite{Bou91}. 
Let us take a point $(z_1, z_2, \tau)$ and compute the Christopher symbols at this point.

We first compute all the geodesics passing through this point. Suppose 
\begin{equation}
    \begin{split}
        & x_1+t \xi_2+t^2 \xi_1=w_1, \\
        & x_2+t \xi_1=w_2. 
    \end{split}
\end{equation}
Then we see that 
\begin{equation}
\begin{split}
    & X_1(t)=z_1+\tau \xi_2+\tau^2 \xi_1-t \xi_2-t^2\xi_1, \\
    & X_2(t)=z_2+\tau \xi_1-t\xi_1.
\end{split}
\end{equation}
All geodesics passing through $(z_1, z_2, \tau)$ can be written as 
\begin{equation}\label{231023a_19}
    (X_1(t), X_2(t), t). 
\end{equation}
By the equation \eqref{230222e2_25}, we obtain 
\begin{equation}\label{231023a_20}
\begin{split}
     -2\xi_1& =-(-\xi_2-2\tau \xi_1, -\xi_1, 1) (\Gamma^1_{\alpha\beta})_{1\le \alpha, \beta\le 3} (-\xi_2-2\tau \xi_1, -\xi_1, 1)^T\\
     & + (-\xi_2-2\tau \xi_1) (-\xi_2-2\tau \xi_1, -\xi_1, 1)  (\Gamma^3_{\alpha\beta})_{1\le \alpha, \beta\le 3} (-\xi_2-2\tau \xi_1, -\xi_1, 1)^T,
\end{split}
\end{equation}
and 
\begin{equation}\label{231023a_21}
\begin{split}
     0& =-(-\xi_2-2\tau \xi_1, -\xi_1, 1)(\Gamma^2_{\alpha\beta})_{1\le \alpha, \beta\le 3} (-\xi_2-2\tau \xi_1, -\xi_1, 1)^T\\
     & + (-\xi_1) (-\xi_2-2\tau \xi_1, -\xi_1, 1)  (\Gamma^3_{\alpha\beta})_{1\le \alpha, \beta\le 3} (-\xi_2-2\tau \xi_1, -\xi_1, 1)^T.
\end{split}
\end{equation}
Christopher symbols are not unique. If we take 
\begin{equation}
\begin{split}
  &   (\Gamma^1_{\alpha\beta})_{1\le\alpha, \beta\le 3}=
    \begin{bmatrix}
        0, & 0, & 0\\
        0, & 0, & -1\\
        0, & -1, & 0
    \end{bmatrix}, \\
    & 
    (\Gamma^2_{\alpha\beta})_{1\le\alpha, \beta\le 3}=
    \begin{bmatrix}
        0, & 0, & 0\\
        0, & 0, & 0\\
        0, & 0, & 0
    \end{bmatrix},  
    (\Gamma^3_{\alpha\beta})_{1\le\alpha, \beta\le 3}=
    \begin{bmatrix}
        0, & 0, & 0\\
        0, & 0, & 0\\
        0, & 0, & 0
    \end{bmatrix}
    \end{split}
\end{equation}
then \eqref{231023a_20} and \eqref{231023a_21} are satisfied. Eventually, one can check directly that under the metric 
\begin{equation}
    \bfg(x_1, x_2, t)=\begin{bmatrix}
        1, & -t, & -x_2\\
        -t, & t^2+1, & x_2 t\\
        -x_2, & x_2 t, & x_2^2+1
    \end{bmatrix}
\end{equation}
the curve \eqref{231023a_19} is a geodesic for all $(z_1, z_2, \tau)$ and $\xi$. \\

Next, let us modify Bourgain's example and show that metrics may not always exist. Take 
\begin{equation}
    \phi(x, t; \xi)=\inn{x}{\xi}+t\xi_1\xi_2+t^2 P(\xi_1),
\end{equation}
where $P(\xi_1)=O(|\xi_1|^2)$ is to be chosen. 
Consider all the curves determined by $\phi$ passing through the origin. They can be written as 
\begin{align}
    & X_1(t)= -t \xi_2-t^2 P'(\xi_1), \\
    & X_2(t)=-t \xi_1.
\end{align}
Recall that if metric tensor exists, then we always have \eqref{230222e2_25}. We check this equation at the origin. Note that $  \frac{d x^{\alpha}}{dt}$ stays the same as in Bourgain's example. However, the left hand side of \eqref{230222e2_25} changes dramatically. More precisely, 
\begin{equation}
\frac{d x^{1}}{dt}\Big|_{t=0}= 
-\xi_2, \ \ \frac{d x^{2}}{dt}\Big|_{t=0}=-\xi_1, \ \ \frac{d x^{3}}{dt}\Big|_{t=0}=1. 
\end{equation}
Therefore, the right hand side of \eqref{230222e2_25} at the origin is (at most) cubic in $\xi$. However, the left hand side of \eqref{230222e2_25} is 
\begin{equation}
(-2P'(\xi_1), 0, 0)^T.
\end{equation}
If for instance we take $P(\xi_1)=\xi_1^{5}$, then there does not exist any metric tensor to make \eqref{230222e2_25} hold for every $\xi$.

%
%
%
%
%
%
%
%
%
%
%
%
%

%

\end{document}